\documentclass[11pt]{amsart}
 \usepackage{amsmath,amssymb,amsfonts}
 \usepackage{amsfonts}
 \usepackage{eucal}
    \usepackage{xcolor}
    \makeatletter
     \def\section{\@startsection{section}{1}%
     \z@{.7\linespacing\@plus\linespacing}{.5\linespacing}%
     {\bfseries \normalfont\scshape
     \centering
     }}
     \def\@secnumfont{\bfseries}
     \makeatother
  \newtheorem{theorem}{Theorem}[section]
   \newtheorem{lemma}[theorem]{Lemma}
 \newtheorem{corollary}[theorem]{Corollary}
\newtheorem{proposition}[theorem]{Proposition}
 
  \newtheorem{remark}[theorem]{Remark}
\newtheorem{remarks}[theorem]{Remarks}
\newtheorem{example}[theorem]{Example}
\newtheorem{problem}[theorem]{Problem}

\def\Romannum#1{\uppercase\expandafter{\romannumeral #1}}
 \scrollmode

\font\sevenrm =cmr10 at  7  pt
\def\ddate {\sevenrm \ifcase\month\or January\or
February\or March\or April\or May\or June\or July\or
August\or September\or October\or November\or December\fi\! {\the\day}, \!{\sevenrm\the\year}}

at  6 pt
\scrollmode

  \scrollmode
\font\sec= cmb10 at 10  pt 
 
\font\sevenrm= cmr10 at 7,3 pt
 at 9,3 pt
 at 12,5 pt

\font\tf=cmbxsl10 at 18pt
 
\scrollmode
\def\ddate {\sevenrm \ifcase\month\or January\or
February\or March\or April\or May\or June\or July\or
August\or September\or October\or November\or December\fi\! {\the\day}, \!{\sevenrm\the\year}}

 \def\beq{\begin{equation}}
\def\eeq{\end{equation}}
  
\def\ben{\begin{eqnarray}}
\def\een{\end{eqnarray}}

\def \a{{\alpha}}
\def \b{{\beta}}
\def \D{{\Delta}}
 \def \d{{\delta}}
\def \e{{\varepsilon}}
 \def \g{{\gamma}}

\def \k{{\kappa}}
\def \l{{\lambda}}

\def \o{{\omega}}
\def \O{{\Omega}}
\def \p{{\varphi}}

 \def \m{{\mu}}
\def \s{{\sigma}}
 
  \def \A{{\mathcal A}}

\def \qq{{\qquad}}
\def \H{{\mathcal H}}
 \def \T{{\mathcal T}}
\def \X{{\mathcal X}}

  \def \noi{{\noindent}}
  \def \={{\!-\!}}
  \def \+{{\!+\!}}

\def\E{{\mathbb E\,}}

\def\P{{\mathbb P}}

\def\R{{\mathbb R}}
\def \dd{{\rm d}}
 
\def\Z{{\mathbb Z}}

\def\N{{\mathbb N}}

\font\tenit=cmti7  

at  6 pt
\scrollmode

  \scrollmode
\font\pH=cmcsc10 at 10 pt\font\sec= cmb10 at 11  pt 
 
\font\sevenrm= cmr10 at 7,3 pt
 at 9,3 pt
 \font\tfa=cmbxsl10 at 8,7pt
\font\tf=cmbxsl10 at 9,9pt
 \font\tffs=cmbxsl10 at 10pt

\def\ddate {\sevenrm \ifcase\month\or January\or
February\or March\or April\or May\or June\or July\or
August\or September\or October\or November\or December\fi\! {\the\day}, \!{\sevenrm\the\year}}
  
\begin{document}

    \title[Upper tail  probability estimates of suprema of Gaussian processes]{{\sec Upper tail  probability Estimates of suprema of Gaussian processes by matrices,      {A     number theoretic example}\\ {\tenit with an appendix}\\{\tfa From randomness to non-randomness:\! some\! directions,\! tools\! and\! bridges}}\vskip 1 pt 
    }
    \author{Michel J.\,G. Weber}  \address{Michel Weber: IRMA, UMR 7501, 10  rue du G\'en\'eral Zimmer, 67084
 Strasbourg Cedex, France}
  \email{michel.weber@math.unistra.fr}

       \begin{abstract}  We prove new lower bounds for the upper tail probabilities of  suprema    of Gaussian processes. We introduce  a proper matrix approach. This is continuing   recent works of the author on deviations  of     suprema and decoupling inequalities for     parametrized families of Gaussian processes.   
       An application is given   to a number theoretic example related to Wintner random M\"obius function.
       \end{abstract}


\maketitle
  
\section{\bf Introduction-Results}\label{s1}
 In  this article
  we are concerned with   deviations  of     suprema of  sample bounded  Gaussian processes $\{X(t), t\in T\}$, more precisely we seek              sharp bounds of type       
 \beq  \label{db.type} \P\big\{ \sup_{t\in T}
 X(t) \le \Theta \big\}\le e^{-\Upsilon} ,\qq \qquad \Theta, \Upsilon\to +\infty,\eeq
which are  valid under  widely tractable correlation conditions;   in view of   applications to  deviations  of     suprema of    $x$-parametrized Gaussian processes $\{X _x(t), t\in T _x\}$, namely the search of    bounds of type       
 \beq  \label{db.type1} \P\big\{ \sup_{t\in T _x}
 X _x(t) \le \Theta _x \big\}\le e^{-\Upsilon _x} ,\qq \qquad \Theta _x, \Upsilon _x\to +\infty,\quad x\to  +\infty,\eeq
and decoupling inequalities. 
  We    recently proved in Weber \cite{W1}  a general  decoupling inequality for finite  Gaussian vectors,   in which     a new zone of $p$  admissible exponents was identified.    
 \begin{theorem}[\cite{W1},\,Th.\,1.2]
\label{dec.ineq.improv} Let $ X=\{X_i, 1\le i\le n\}$ be a centered Gaussian vector  with invertible  covariance matrix $C$, and let  $\s_i^2=\E X_i^2>0$ for each  $1\le i\le n$, $\underline{\g}=(\s^2_1, \ldots ,\s^2_n)$.  
 \vskip 2 pt {\bf (i)}
  Let  $1< p<\infty$.  There exists a  matrix $R$      expressed in terms of $C$ and $ I(\underline{\g})$ only, defined in {\rm (6)}, with columns  $r^1, \ldots, r^n$  such that  letting 
\beq \label{det.B.cond.}
  \xi_j=\frac{\langle     I(\underline{\g})\,   r^j, r^j\rangle}{\langle C\, r^j, r^j\rangle}>0, \qq j=1, \ldots, n, 
\eeq
 \beq \label{det.B.cond.pos.sign.enonce}       S= \Big\{\frac{1}{\xi_i}, \ i=1,\ldots, n\Big\}^c\bigcap\Big\{p: \#\big\{i:\, 1\le i\le n;\,  p\xi_i< 1\big\}\  \hbox{\it  is even} \Big\} , \eeq
 if $p\in S$ then, 
        \beq\label{estimate.prod.c.enonce}
\bigg| \E\Big( \prod_{i=1}^n f_i(X_i)\Big)\bigg|
    \le      \Big(\prod_{i=1}^n\s_i\Big)^{\frac{1}{p}}\det(C)^{-\frac{1}{2p}}   \Big(  \prod_{i=1}^n \Big|1-\frac{1}{p\, \xi_j }\Big|\Big)^{-\frac12(1-\frac{1}{p})} \
 \prod_{i=1}^n  \big\| f_i(X_i) \big\|_p,
\eeq 
 for any complex-valued measurable  functions $f_1, \ldots, f_n$ such that $f_i\in L^p(\R)$, for all $1\le i\le n$. 
\vskip 2 pt {\bf (ii)} Condition \eqref{det.B.cond.pos.sign.enonce} is satisfied if $p> \max_{ j=1 }^n\, \frac{1}{\xi_i}$. 
 \vskip 3  pt {\bf (iii)} {\rm (Determination of $R$)} 
 Let $ \underline{\m}=\{\m_1,\ldots,\m_n\}$ be  the eigenvalues   $C$, and  $U$   an orthogonal matrix such that ${}^tUCU=I(\underline{\m})$. Let $D=I(\underline{\m}^{-1/2})$,  $W= UD$  so that  ${}^tW  C W=I$, and let matrix  $H= {}^tWKW$. The matrix  $H$ is real symmetric, positive definite.  
Let $\D$ be a real diagonal matrix, and  ${}^tVV=I$ such that ${}^tVHV=\D=I(\underline{\l})$.
  \vskip 2 pt Then $R$ equals the product\beq\label{comp.R} R=UDV,
  \eeq  $R$  is invertible, and we have the relations $\l_j=p\, \xi_j$, $j=1, \ldots, n$. 
 \end{theorem}
The bound  \eqref{estimate.prod.c.enonce}
 is    practical in Statistics,  by using   computing methods for computing eigenvalues from Numerical Analysis.    The set of $p$  admissible exponents,  a disconnected set,     seems   hardly improvable.  The proof is based on     a   simultaneous diagonalization argument. 
 \vskip 3 pt   In this   project we  develop  recent related works  in Weber    \cite{W1,W6,W3}, also  \cite{W10}, from a matrix point of view.   We  highlight    a general  transverse  matrix approach for    estimating upper tail probabilities of their suprema,   the bounds obtained  are new, sharp and    easy to apply. 
We also study a   central problem (Problem \ref{problem3}). 
We think it is opening a   line of study  in this part of the general  theory of Gaussian processes, which we always found the intersection with Matrix analysis  too limited. Further developments should be expected, for instance by restricting the setting  to  remarkable classes of matrices from the theory.
  
   \vskip 3 pt   
 \subsection{Problems considered.} Consider  the following problem.
   \begin{problem}\label{problem2}\rm
  Let   be given reals $u,\l  $ such that $0<\l\le u<1$, and consider  the class of       
centered Gaussian processes $X =\{X(t), t\in T\}$ with unit variance  satisfying  the following property: 
\vskip 3 pt {\it There exists a finite partition $\{T_j,j\in J\}$ of $T$, with $\#\{T_j\}=k$ for each $j$, $\#\{T \}=Nk$, such that}
\ben  \label{assumption2} \begin{cases} \E X(s)X(t)  \le u,  \qquad \forall {s\neq t\in T_i}, \ \forall i\in J, \cr |\E X(s)X(t)| \le  \l,  \qquad  \forall {s \in T_i, \forall t\in T_j}, \ \forall i,j\in J, \, i\neq j.
\end{cases} 
\een
 Find sharp bounds of   
 $$ \P\Big\{ \sup_{j\in J}\,\sup_{t\in T_j} X(t) \le \Theta\Big\}.$$ 
 \end{problem}   
   Assumptions made describe a typical situation  in analysis of Gaussian processes trajectories: let $d(s,t)=\| X(s)-X(t)\|_2 $ be the pseudo-metric induced by $X$ on $T$, and let $B_d(t, \e )$ denotes the $d$-closed ball of radius $\e>0$, with center   $t\in T$.   In our setting, we have for $0<\e_1\le \e_2$, ($\e_1= \sqrt{2(1-u)}$, $\e_2= \sqrt{2(1-\l)}$)
\ben   \begin{cases}    \min\{d(s,t),\, s,t\in T_i,\,  s\neq t,\,i\in J \}\,   \ge \e_1   , \cr    \min\{d(s,t),\,s\in T_i, t\in T_j,\, i,j\in J,\, i\neq j\} \,   \ge  \e_2 .
\end{cases} 
\een
Let $\e>0$. A finite set of points $T'\subset T$ is said to be $\e$-separated,  with respect to  $d$, if $\min\{d(s,t),  s,t\in T', s \neq t \}  \ge \e$. If $(T,d)$ is pre-quasi compact, it is easy to see there is  $T'\subset T$, which is    $\e$-separated with finite maximal cardinality. By a simple reasoning it also follows that the union of balls $B_d(t, \e )$, $t\in T'$ forms a $d$-covering of $T$.  
Let each $T_i$ be maximal in the sense defined above, maximality is taken for $\e$, with respect to $T_i$. The union of balls $B_d(t, \e )$,  $t\in T_i$,   $i$ running  in $J$, forms a $d$-covering of $T$.  We can for instance construct these sets the following way clarifying our purpose. Let $T_{\e_2}=\{s_i,i\in J\}$ be  $\e_2$-maximal, next for each $i\in J$, let $T_i\subset B_d(s_i, \e_2)$ be defined as being $\e_1$-maximal  with respect to $B_d(s_i, \e_2)$. The reader will have understood that in doing so we built with the union of $T_i$, $i\in J$, a finer covering of $T$    of  order $e_1$,  in a same time that we also have a way to exhaust $T$  in   building a tree, namely an inductive limit of sets. So this second example really concerns a fundamental situation in the study of the regularity of Gaussian processes.   
  \vskip 5 pt
 We first demonstrate under these simple assumptions a general sharp bound  which states as follows.
\begin{theorem}\label{th.1} Let $X=\{X(t), t\in T\}$, $T$ finite, be a centered Gaussian process with unit variance, satisfying assumption \eqref{assumption2} with $u=1-\e$, $0<\e<1$. Let  $0<c<1$, $\p(k)>c$ and non-decreasing with $k$.
 Then we have the bound in which $C$ is an absolute constant,
\begin{align}\label{bound.complete} \P\Big\{  \sup_{j=1}^N   \sup_{s\in T_j}    X(s)\le  \sqrt{(2/\e)\log (k\p(k))} \Big\} \, \le \,
e^{ -\frac{ CscN   }{    \p(k) \sqrt {2\log(k \p(k))}}}
   +    \frac{\l k^2N}{4\pi\sqrt{1-\l^2}}   e^{-\frac{  2 \log (k\p(k))  }{\e (1+\l )} } .\end{align}
\end{theorem}
The given formulation renders its application easy  in practice. On the symptomatic and well-known example below,  of  which the generic class of associated Dirichlet polynomials   becomes rapidly  complicate to investigate, we obtain sharp bounds (Theorem \ref{demi.logx1}).
  A  more specific problem to which Theorem \ref{th.1} applies is:\begin{problem}\label{problem1}\rm Let $\{X(t), t\in T\}$ be a centered Gaussian process with unit  variance. Given any finite set $S\subset T$, find minimal conditions on   $ X,S  $ under which the tail probability 
$$\P\big\{\sup_{s\in S} X(s)\le \Theta\big\}$$ has order $\le  e^{- C  \sqrt{ \log \#(S)  }  }$,  $C$ an absolute constant, and $\Theta>0$ be as large as possible. 
\end{problem}  \vskip 3 pt 
Let  $d(s,t) =\sqrt{\E|X(s)-X(t)|^2}$ be the  the pseudo-metric induced on $T$ by $X$.    The determination of conditions which are linking $\#(S)$ with the   $d$-distance between elements of $S$, $x$ and $y$, and render this possible are thus the object  of problem. 
    \vskip 2 pt  We investigate this problem having in view an application to Wintner random M\"obius function (works of Halasz \cite{Ha},  Harper \cite{H}).  Let   $2\le y< x $, $t\ge 0$  and, $\mathcal P$ denoting  the set prime numbers,  let $\{g_p,p\in \mathcal P\}$ a sequence of iid centered Gaussian random variables with unit variance. Consider the Gaussian processes,   
      \begin{eqnarray*}\begin{cases}  X_{y,x}(t)&    =\ \sum_{y<p\le x}\frac{g_p\cos t\log p}{p^{\frac12+ \frac{1}{\log x}}},\cr
 X_y(t)&= \ \sum_{p>y}\frac{g_p\cos t\log p}{p^{\frac12+ \frac{1}{\log x}}}   
,\end{cases}\qq\quad t\ge 0
   \end{eqnarray*}
together with  their $L^2$-normalization $\widetilde X_{y,x}(t) =\ \frac{X_{y,x}(t)}{\sqrt{\E X_{y,x}^2(t) }}$, $\widetilde X_y(t)=\  \frac{X_y(t)}{\sqrt{\E X_y^2(t) }}$.
  
 We prove 
\begin{theorem}\label{demi.logx1} Let $x\ge y\ge 2$. 
{\bf (i)} Let $s_i\ge 1$, $i=1,\ldots,k$,  be satisfying the separation  condition
 $$  \min \big((|s_i-s_j|, |s_i+s_j|, 2s_i  ) : i,j=1,\ldots ,k ,  i\neq j\big)\ge \varpi,$$
for some positive real  $\varpi $  such that $$c_0/\log x \le  \varpi\le   e^{ a (\log y)^{3/5}   }/  \log y . $$ Let
$   k = (\eta /2C^*)   (\log  \frac{\log x}{\k\log y}  )   \varpi\log y$  (constants $c_0$ and $C^*$  are the same   as in Lemma \ref{bound.cor.}). Then \begin{eqnarray*} \P\bigg\{ \sup_{ 1\le t\le k \varpi } X_{y,x}(t) \le \sqrt{   \Big(  \log  (\frac{\log x}{\k\log y} )   \Big)\Big(   2  \log\frac{k}{\log k}   -   \frac{\eta   \log k }{k} \Big)}\bigg\} 
  &\le &  C_\eta\, e^{- C'_\eta \sqrt{ \log k  }  },
\end{eqnarray*}
also (constants $C_\eta,\,C'_\eta$ arise from Corollary \ref{ass1th}),
\begin{eqnarray*} \P\bigg\{ \sup_{ 1\le t\le k \varpi }    X_{y }(t) \le \sqrt{   \Big(  \log  (\frac{\log x}{\k\log y} )   \Big)\Big(   2  \log\frac{k}{\log k}   -   \frac{\eta   \log k }{k} \Big)}\bigg\} 
  &\le &  C_\eta\, e^{- C'_\eta \sqrt{ \log k  }  }.
\end{eqnarray*}
{\bf (ii)}
 Let $0<\e<1$, $\{s_i,1\le i\le k\}\subset [1,2]$ and assume that $\E  \widetilde X_{y,x}(n+s_i)\widetilde X_{y,x}(n+s_j)\le 1-\e$, for all $i\neq j$ and all $0\le n\le N$. 
     Let $\l= {C^*}/{  (  \log  \frac{\log x}{\k\log y}      ) \log y}$. Let also   $0<c<1$, $\p(k)>c$ and non-decreasing with $k$. 
    Then 
 \begin{align*}  \P\Big\{  \sup_{n=1}^N   \sup_{i=1}^k  \widetilde  X_{y,x}(n+s_i)\le & \sqrt{(2/\e)\log (k\p(k))} \Big\} \le 
e^{ -\frac{ CN   }{    \p(k) \sqrt {2\log(k \p(k))}}}
   +    \frac{\l k^2N}{4\pi\sqrt{1-\l^2}}   e^{-\frac{  2 \log (k\p(k))  }{\e (1+\l )} } .\end{align*}
{\bf (iii)} Let $1/2\le b<1$. There exists a positive absolute constant $C$, and  
for any  positive number $M$,   a positive number $x(M)$ such that for all $x\ge x(M)$,
\begin{align*}  \P\bigg\{  \sup_{1\le t\le 2(\log\log x)^2}       X_{y,x}(t)&\le  \sqrt{2\Big(\log  \frac{\log x}{\log y}\Big)^2+ \Big(\log  \frac{\log x}{\log y}\Big)\Big\{ (\log\log x)^b\log\Big(\frac{\log  \frac{\log x}{\log y}}{(\log\log x)^b}\Big)-2\log \k \Big\}} \bigg\} \cr  & \qq \qq \le \,\frac{ C  }{(\log x)^2 e^{     M(\log\log x)^{ b}  }}.\end{align*}
 \end{theorem}
 It follows   from (iii) that,
\begin{align*}  \P\Big\{  \sup_{1\le t\le 2 (\log\log x)^2}       X_{y,x}(t)&\le  \sqrt{2}(\log \log x -\log \log y) \Big\} 
\le \,\frac{ C  }{(\log x)^2 e^{     M(\log\log x)^{ b}  }}.\end{align*}
 \subsection{A general  upper tail  probability estimate.}   Problem \ref{problem2} suggests
 to study the family of block-matrices   below.  \vskip 3 pt

 \begin{problem}\label{problem3}\rm     
  Let    $R=R(\l)$ be the $k\times k$ matrix whose entries all equal to $\l$, and $C=C(u) $ be the $k\times k$ matrix whose entries off the diagonal are all equal to $u$, and to 1 otherwise. Define the $Nk\times Nk$ matrix composed of $N$ diagonal blocks  $C $ and of  blocks $R $ elsewhere,
 \begin{eqnarray}\label{ulamb}& &C_{k,N}(u,\l)\,=\,C(u,\l)\,=\,{\begin{pmatrix}
  C  &{}&{}&{}&{}\cr 
   {} &C  &{} &R &{}\cr
  {}&{}  &\ddots &{}&{}
    \cr
 {}       &R  &{}&\ddots& {}
  \cr 
 {}&{}&{}&{} &C \end{pmatrix}}
  \end{eqnarray} 
  \end{problem} 

  
Let $\{ X(t), t\in T\}$ be a Gaussian process with covariance matrix $C(u,\l)$. We prove the following bound. 
  
   \begin{theorem}\label{gen.case.th} {\bf (i)} For   any positive real   $\Theta$, 
\begin{eqnarray*}\P\Big\{ \sup_{t\in T}  X(t)\le \Theta\Big\}&\le  &\int_{\sup_{i=1}^{Nk} x_i\le \Theta}  \frac{e^{-  \frac{ 1}{2 }  \, Q(u,\l)(\underline x)   }    }{(2\pi)^{Nk/2}[ (  1- u)^{(k-1)N}\big(  1+u( k-1 )\big))^N]^{1/2}}\,\dd \underline x , 
 \end{eqnarray*}
where for each $\underline x\in \R^{Nk}$,
\begin{eqnarray*} Q(u,\l)(\underline x)
&  =&- \frac{   \l  }{(1-u+ k(u-\l))\,( (1-u)+Nku +(N-1)k\l )}\Big(\sum_{ v=1}^{Nk} x_{v}\Big)^2
\cr & &- \frac{ (u-\l)  }{(1-u)(1-u+ k(u-\l))}\sum_{j=0}^{N-1}\Big(\sum_{ l=1 }^k x_{jk+ l}\Big)^2 +\frac{1}{1-u}\Big(\sum_{ v=1}^{Nk} x^2_{v}\Big).
\end{eqnarray*}

{\bf (ii)} There exists a positive absolute constant $C^*$ such that for   any positive reals   $\Theta$ and $\tau$,
\begin{eqnarray*}
      \P\Big\{ \sup_{i=1}^n X_i\le \Theta \Big\}
   &\le  & e^{C^*\tau^2/ N^2 } \bigg( \int_{ \sup_{i=1}^{ k} x_i \le \Theta}\frac{e^{-  \frac{ 1}{2(1-u) }    \sum_{ v=1}^{ k} x^2_{v}    }    }{  [(2\pi) (  1- u)^{ k-1 } \big(  1+u( k-1 )\big)  ]^{ 1/2}}\,\dd \underline x \bigg)^N  
   \cr & & +4\, e^{- C^*\tau^2/ (kN)^2}.\end{eqnarray*}
 
{\bf (iii)} Let 
  $$\b(u,\l) =\Big(  \frac{1}{ 1-u+ k(u-\l) }  \Big) \Big(  \frac{  1-u+Nku-k\l }{    1-u +Nk(u +\l)-k\l  }  \Big). $$
 Assume that $Nu>\l$ and $(k-1)u>1$. Then $0<\b(u,\l)<1$, $$\b(u,\l)\asymp\frac{1}{ 1-u+ k(u-\l) }  \frac{   ( N u- \l )}{     (N  u   - \l ) +N\l} . $$ 
 Further for    any positive real   $\Theta$, 
\begin{eqnarray*}\P\Big\{ \sup_{t\in T}  X(t)\le \Theta\Big\} &\le &\int_{\sup_{i=1}^n x_i\le \Theta}  \frac{e^{-  \frac{ \b(u,\l)}{2 }    \sum_{   1\le v\le Nk}x_{v}^2  }    }{(2\pi)^{n/2}[ (  1- u)^{(k-1)N}\big(  1+u( k-1 )\big))^N]^{1/2}}\,\dd x
  .
 \end{eqnarray*}
 \end{theorem} 

 \vskip 3 pt  

\bigskip\par
 \subsection{Intermediate results.} 
 We establish     two tools   (Theorem  \ref{Ac1c} and Corollary \ref{ass1th}) which 
 provide direct bounds, both in the asymptotic case and non asymptotic case, without requiring   an additional intermediate comparison step to pass by for their application.  Assumptions  made are simple, and easy to verify.
   Let $$
\Phi(x)  ={1\over \sqrt{2\pi}} \int^\infty_x e^{-t^2/2} \dd t, \qq \Psi(x) = 1-\Phi(x).
 $$
 
    In the simplest case when $X= \{X_1, \ldots, X_n\}$ is a centered Gaussian vector, such that  for some  $0<\l <1$,  \beq\label{assumption}\hbox{$\E X^2_u=1 $, \qq  $\E X_uX_v\le \l$ , \qq   $u\neq v$,}
\eeq
 a sharp quantitative deviation bound  with Gumbel's type version can be established.  
  \begin{theorem}
\label{Ac1c} 
{ \rm (i)} For   any positive real   $\Theta$,  
  \begin{eqnarray*}\P\Big\{ \sup_{i=1}^n X_i\le \Theta\Big\}  &\le &\Big(1+\frac{  \l n  }{   1-\l }\Big)^{(n-1)/2}  \ \Phi \Big(\frac{\Theta}{\sqrt{ 1+\l(n-1) }}\Big)^n   .
 \end{eqnarray*}
 {\rm (ii)}   Set
$b_n= \big( \log \frac{n^2}{4\pi \log n} \big)^{1/2}$.  
  Given any positive real $\e$, for $n$ sufficiently large, $n\ge n(\e)$,  and all $x$ such that  $x\ge -b_n^2$,
\begin{eqnarray*}\P\Big\{ \sup_{i=1}^n X_i\le \Big({\frac{1}{b_n}x+ b_n}\Big) {\sqrt{ 1+\l(n-1)}} \Big\}     \,\le \, \Big(1+\frac{  \l n  }{   1-\l }\Big)^{(n-1)/2}  e^{- e^{-x} (1-\e)   } .
 \end{eqnarray*}  \end{theorem} 
When $\l$ varies between $0$ and $1$, $\P\big\{ \sup_{i=1}^n X_i\le \Theta\big\}$ passes from the value $\Phi(\Theta)^n$ to $\Phi(\Theta)  $ respectively; the way how  $\P\big\{ \sup_{i=1}^n X_i\le \Theta\big\}$ remains $\ll \Phi(\Theta)$ when $\l$ approaches the value $1$   has some clear importance
although it is not simple to describe in general.  This is examined in  sub-section \ref{small}. 

We refer to     Berman \cite[Ch.\,9]{Be}, concerning Gumbel law, extreme value theory, namely  convergence in law results for sequences of partial maxima, usually obtained by applying derivation   theorems.

\vskip 3 pt Assume now that assumption \eqref{assumption} is slightly strengthened as follows: for some real $0<\l<1$,
  \beq\label{assumption1}\hbox{$\E X^2_u=1 $, \qq  $|\E X_uX_v|\le \l$, \qq   $u\neq v$.}
\eeq
 \begin{corollary}\label{ass1th} Let $0<\eta<1$ be fixed. There exist positive constants $C_\eta, C'_\eta$ depending on $\eta$ only, such that for any centered Gaussian vector $X=\{X_1,\ldots, X_n\}$, $n\ge 3$, satisfying  assumption \eqref{assumption1}   for some positive real $ \l  $ such that  $ \l\le \eta/2n  $,  we have 
\begin{eqnarray*}
 \P\Big\{ \sup_{i=1}^n X_i\le \sqrt{    2  \log n -2 \log\log n   -   \eta  (\log n)/      n }\Big\}
 &\le & C_\eta\, e^{- C'_\eta \sqrt{ \log n  }  }.
\end{eqnarray*}
\end{corollary}
\begin{remark}\label{F.decoupling.gen.case.th}\rm   Due to  
 the length reached by this paper,  the motivation  to write it in particular  as clearly as possible, applications    of Theorem \ref{gen.case.th} and of decoupling Theorem \ref{F.decoupling.th} are the object of  a separate work. 
 \end{remark} 


   
 \bigskip \par  The example investigated laying at the interface Probability theory/Analytic number theory, we took that opportunity  to add an Appendix  '{\it From randomness to non-randomness}',  the  sections of which are part  from the work in preparation Weber \cite{W0...}, and   
gather a few of some very  remarkable 
 tools, methods and tricks that are efficient devices on  working at the interface Probability theory/Number theory in Appendix, with some examples and references. These are presented in a such way  that an analyst or a analytic number theorist can easily understand the principle of these tools, also catch how there are established, and the way to manipulate them.  

 \section{\bf The matrix $\boldsymbol{ C (\l )}$}\label{s2}
  Let $\mathcal C_n $ be the family   of $n\times n$ matrices $ C(\l)$,   defined for each $0< \l<1$ by  
\begin{eqnarray} & &C(\l)=C_n(\l)\,=\,{ \left[\begin{matrix}
1   & \l  & \l  &   \cdots  &  \cdots &\l \cr
\l    &1  &\l  &  \cdots &\cdots& \l      \cr
   \vdots         &    \vdots   &\vdots   &\vdots& \vdots & \vdots   
 \cr
 \vdots         &    \vdots   & \vdots& \vdots & \ddots &\vdots
 \cr \l & \l &\l   &\cdots &\l  &1    \end{matrix} \right]} 
 \end{eqnarray}

     \begin{proposition}\label{C.enonce}{\rm (i)} $ \det\big( C(\l)\big) =(  1- \l )^{n-1}\big(  1+\l( n-1 )\big)$.  
\vskip 2 pt\noi  {\rm (ii)} Let 
 $\tau =\frac{\l  }{  1-\l +n\l  }$, $\theta=\frac{-\tau}{1 -\tau}$.  
Then$$C(\l)^{-1} \,=\,\Big(\frac{(1 -\tau)^n}{1-\l}\Big) C(\theta).$$
 \vskip 2 pt\noi  {\rm (iii)}
 $$-\frac12 {}^t \underline{x}\, C(\l)^{-1}\underline{x}   \,=\,- \Big(\frac{ 1}{2(1-\l)}\Big) \Big( \sum_{i=1}^n x_i^2- \tau \Big(\sum_{i=1}^n x_i \Big)^2 \Big),$$
where $\tau=\frac{\l  }{ (1-\l)+n\l } $. \end{proposition}
Claim (ii)   shows that the class $\mathcal C$ is closed for the inversion operation (up to a multiplicative factor). 
  \subsection{Proof of Proposition  \ref{C.enonce}}\label{ss2}
   We prove (i) (other proofs are available)  using the {\it companion  matrix}   of a polynomial (\cite{B}
   ).
  \begin{lemma} \label{ch.eq}The matrix
 \begin{eqnarray*} A\,=\,  { \left[\begin{matrix}
0 & 0   &  \cdots&  \cdots  &  \cdots &0 &-a_n \cr
 1  &0    &\cdots& \cdots &\cdots&  0&-a_{n-1}     \cr
0        & 1           &0 &  \cdots& \cdots    &0    \cr
 \vdots         &   \vdots                  &\vdots   &\vdots   &\ddots& \vdots & \vdots \cr 0 &  0 &0 &0  &\cdots  &0 &-a_2 
 \cr 0 & 0&0 &0  &\cdots & 1 &-a_1   \end{matrix} \right] }
  \end{eqnarray*}  
has   characteristic polynomial $f_{A}(\rho) = \det(\rho I -A)=\rho^n+a_1\rho^{n-1}+a_2\rho^{n-2}+\ldots +a_{n-1}\rho+a_n $.
\end{lemma}
 \begin{proposition}\label{Ap1} We have 
$$ \det\big(p I -C(\l)\big) =(p-1+ \l )^{n-1}\big(p-1-\l( n-1 )\big). $$
In particular,
$$ \det\big( C(\l)\big) =(  1- \l )^{n-1}\big(  1+\l( n-1 )\big). $$
\end{proposition}
\begin{proof}
 By replacing in the matrix $p I -C(\l)$ the columns $C_i(\l) $ by $C_i(\l) -C_{i+1}(\l) $ for each $i=1,\ldots , n-1$, $C_n(\l) $ being left unchanged,   we get  
 {\small      \leftskip =0,4cm \rightskip=0,4cm  \begin{eqnarray*}& &\det\big(p I -C(\l)\big)\,=\,{ \left|\begin{matrix}
(p-1)   & -\l  & -\l  &\cdots&  \cdots  &  \cdots &-\l \cr
-\l    &(p-1)  &-\l  &\cdots& \cdots &\cdots& -\l      \cr
-\l         &-\l            &(p-1) & -\l  &\cdots& \cdots    & -\l   \cr
-\l         &-\l            &-\l & (p-1)  &\cdots& \cdots    & -\l   \cr \vdots         &   \vdots                  &\vdots   &\vdots   &\vdots& \vdots & \vdots   
 \cr
 \vdots         &   \vdots                  &\vdots   &\vdots   &\ddots& \vdots & \vdots \cr -\l &  -\l &-\l &-\l  &\cdots  &(p-1) &-\l 
 \cr -\l & -\l &-\l &-\l   &\cdots &-\l  &(p-1)    \end{matrix} \right| }
\cr&\cr \cr  &=&{ \left|\begin{matrix}
(p-1) + \l  & 0   & 0 &\cdots&  \cdots  &  \cdots &-\l \cr
-\l   -(p-1)  &(p-1) + \l&0  &\cdots& \cdots &\cdots& -\l      \cr
0        &-\l   -(p-1)           &(p-1)+\l & 0  &\cdots& \cdots    &- \l   \cr
0         &0            &-\l  -(p-1)& (p-1)+\l  &\cdots& \cdots    &- \l   \cr \vdots         &   \vdots                  &\vdots   &\vdots   &\vdots& \vdots & \vdots   
 \cr
 \vdots         &   \vdots                  &\vdots   &\vdots   &\ddots& \vdots & \vdots \cr 0 &  0 &0 &0  &\cdots  &(p-1)+\l &-\l 
 \cr 0 & 0&0 &0  &\cdots &-\l   -(p-1) &(p-1)    \end{matrix} \right|}
 \cr&\cr \cr
   &=&(p-1+ \l )^{n-1}{\left|\begin{matrix}
1 & 0   & 0 &\cdots&  \cdots  &  \cdots &-\l \cr
-1  &1 &0  &\cdots& \cdots &\cdots&- \l      \cr
0        &-1           &1 & 0  &\cdots& \cdots    & -\l   \cr
0         &0            &-1& 1  &\cdots& \cdots    & -\l   \cr \vdots         &   \vdots                  &\vdots   &\vdots   &\vdots& \vdots & \vdots   
 \cr
 \vdots         &   \vdots                  &\vdots   &\vdots   &\ddots& \vdots & \vdots \cr 0 &  0 &0 &0  &\cdots  &1 &-\l 
 \cr 0 & 0&0 &0  &\cdots &-1 &(p-1)    \end{matrix} \right|}\,=\, (p-1 + \l )^{n-1}\det (  I +M ),
 \end{eqnarray*}  }
 where \begin{eqnarray*}  M  \,=\, { \left[\begin{matrix}
0 & 0   & 0  &  \cdots  &  \cdots &-\l \cr
-1  &0 &0   & \cdots &\cdots& -\l      \cr
0        &-1           &0    &\cdots& \cdots    &- \l   \cr
 \vdots         &   \vdots                  &\vdots       &\vdots& \vdots & \vdots   
 \cr
  0 &  0 &0    &\cdots  &0 &-\l 
 \cr 0 & 0&0    &\cdots &-1 &(p-2)    \end{matrix} \right] }
\qq\qq\qq\ \  
\end{eqnarray*}  
  The characteristic polynomial $f_{-M}(\rho)$ of the matrix $-M$ equals
$$f_{-M}(\rho) = \det(\rho I -(-M))= \rho^n -(2-p)\rho^{n-1}-\l\big(\rho^{n-2}+\ldots +\rho +1\big) . $$
Therefore
 $ \det (  I +M )= f_{-M}(1)
 =p-1-\l( n-1 ),$  
and so
$$ \det\big(p I -C(\l)\big)=(p-1+ \l )^{n-1} \det (  I +M )=(p-1+ \l )^{n-1}\big(p-1-\l( n-1 )\big) ,$$
hence $ \det\big(   C(\l)\big)=(  1- \l )^{n-1}\big(  1+          \l( n-1 )\big).$
\end{proof}
\begin{remark}\label{Ap2}Similarly one can   show that,
 $$ \det\big(I(\underline{d}) -C\big)=\prod_{j=1}^{n-1}(d_j-1+ \l ) \det (  I +M )=\prod_{j=1}^{n-1}(d_j-1+ \l )\big(d_n-1-\l( n-1 )\big) .$$
\end{remark}
    \noi  (ii)  Follows from standard inverse matrix calculations. 
Let $X= {}^t\underline x$, $Y= {}^t\underline y$ and  set $s=\sum_{i=1 }^ny_i$. 
 The system of linear equations $C(\l)Y=X$     reduces to
 \begin{eqnarray*}& & {\left[\begin{matrix}
1   & \l  & \l   &  \cdots  &  \cdots &\l \cr
\l    &1  &\l   & \cdots &\cdots& \l      \cr
  \vdots         &   \vdots                    &\vdots   &\vdots& \vdots & \vdots   
  \cr
 \vdots         &   \vdots                      &\vdots   &\ddots& \vdots & \vdots \cr \l &  \l &\l    &\cdots  &1 &\l 
 \cr \l & \l &\l     &\cdots &\l  &1    \end{matrix} \right]}
{\left[\begin{matrix}
y_1 \cr
y_2          \cr
 \vdots   
  \cr
 \vdots  \cr y_{n-1} 
 \cr y_n   \end{matrix} \right]}  \,=\, {\left[\begin{matrix}
(1-\l)y_1 + \l s \cr
(1-\l)y_2 + \l s    \cr
  \vdots
  \cr
 \vdots    \cr (1-\l)y_{n-1} + \l s \cr (1-\l)y_n + \l s    \end{matrix} \right]}
 \,=\,{\left[\begin{matrix}
x_1 \cr
x_2          \cr
  \vdots   
  \cr
 \vdots  \cr x_{n-1} 
 \cr x_n   \end{matrix} \right]}
 \end{eqnarray*}
  
Hence we get
 \begin{eqnarray*}
  \left\{\begin{matrix}
(1-\l)y_1 + \l s=x_1 \cr
(1-\l)y_2 + \l s =x_2   \cr
  \vdots
  \cr
 \vdots    
 \cr (1-\l)y_{n-1} + \l s=x_{n-1}
  \cr (1-\l)y_n + \l s   = x_n. \end{matrix} \right\}
  \end{eqnarray*}
By summing up  and  setting $t=\sum_{i=1 }^nx_i$, we obtain
 $ \big((1-\l)+n\l\big) s=  t$. 
Thus
\begin{eqnarray*}
   y_j     = \frac{x_j}{1-\l}  - \frac{\l t}{(1-\l)\big((1-\l)+n\l\big)} , \qq \quad j=1, \ldots, n.
  \end{eqnarray*}
 Write this as follows
$$
   y_j      =  x_j\Big(\frac{ 1}{1-\l} -\frac{\l  }{(1-\l)\big((1-\l)+n\l\big)}\Big) - \sum_{k=1\atop k\neq j}^n x_k\Big(\frac{\l  }{(1-\l)\big((1-\l)+n\l\big)}\Big)    
    =  \sum_{k =1}^n d_{j,k}x_k, $$with 
$$  d_{j,j} =\frac{ 1}{1-\l} -\frac{\l  }{(1-\l)\big((1-\l)+n\l\big)},\qq  d_{j,k} =  \frac{-\l  }{(1-\l)\big((1-\l)+n\l\big)} , \quad k\neq j   . $$
 Let $D= \{d_{j,k}\}_{j,k=1}^n$. We find that the   solution to the matrix equation $CY=X$ is, for any column vector $X$,
$  Y=DX$.
It also follows that $Y=C(\l)^{-1}X$ for any $X$; thus  $C(\l)^{-1}=D$. Let 
\beq \label{tau.theta}\tau =\frac{\l  }{ (1-\l)+n\l }, \qq \quad  \theta=\frac{-\tau}{1 -\tau}.  
\eeq
Then $d_{j,j}=\frac{1-\tau}{1-\l}$, $d_{j,k}=\frac{ -\tau}{1-\l}$, and  \begin{eqnarray*} \label{C(lambda)-1} C(\l)^{-1}&= &\Big(\frac{1}{1-\l}\Big)  { \left[\begin{matrix}
1 -\tau   & -\tau  & -\tau   &  \cdots  &  \cdots &-\tau \cr
-\tau    &1 -\tau  &-\tau   & \cdots &\cdots& -\tau      \cr  \vdots         &   \vdots                  &\vdots       &\vdots& \vdots & \vdots   
 \cr
 \vdots         &   \vdots                  &\vdots       &\ddots& \vdots & \vdots \cr -\tau &  -\tau &-\tau    &\cdots  &1 -\tau &-\tau 
 \cr -\tau & -\tau &-\tau     &\cdots &-\tau  &1 -\tau    \end{matrix} \right]}\,=\,\Big(\frac{ 1 -\tau  }{1-\l}\Big) C(\theta).
 \end{eqnarray*}
  (iii) It is easy to check that
 $$-\frac12 {}^t \underline{x}\, C(\l)^{-1}\underline{x}   \,=\,- \Big(\frac{ 1}{2(1-\l)}\Big) \Big( \sum_{i=1}^n x_i^2- \tau \Big(\sum_{i=1}^n x_i \Big)^2 \Big),$$
where $\tau=\frac{\l  }{ (1-\l)+n\l } $. 
  


   \section{\bf    Upper tail  probability estimates derivation (I)} \label{s3}

Let $X^n=(X_1,\ldots X_n)$ be  a centered Gaussian vector with covariance matrix   $C(\l)$. Put    \beq\label{Fnl} \mathcal F_{n,\l} (\Theta)=\P\Big\{ \sup_{i=1}^n X_i\le \Theta\Big\}, \qq \Theta\in \R.
 \eeq

As a consequence we have 
 \begin{corollary}\label{Ac1.enonce} For any real $\Theta$,
 \begin{eqnarray*}\mathcal F_{n,\l} (\Theta)&= &\int_{\sup_{i=1}^n x_i\le \Theta}  \frac{e^{-  \frac{1}{2(1-\l)}   ( \sum_{i=1}^n x_i^2- \tau  (\sum_{i=1}^n x_i  )^2  ) } }{(2\pi)^{n/2}[(1-\l)^{n-1}(1+\l(n-1))]^{1/2}}\,\dd x,
  \end{eqnarray*}
moreover,
   $$\mathcal F_{n,\l} (\Theta)\ \le \ \Big(1+\frac{  \l n  }{   1-\l }\Big)^{(n-1)/2}\, \Phi  \Big(\frac{\Theta}{\sqrt{ 1+\l(n-1) }}\Big)^n.$$
 \end{corollary}
  
Let $Y^n=(Y_1,\ldots, Y_n)$ be  a centered Gaussian vector satisfying assumption \eqref{assumption}.  To convey Corollary  \ref{Ac1.enonce} from $X^n$ to $Y^n$, it suffices to apply Slepian lemma, which we recall below.
   \begin{lemma}[\cite{Sl}, Lemma\,1,\,p.\,468] Let $Y=\{Y_1,\ldots, Y_n\}$, $X=\{X_1,\ldots, X_n\}$, be two centered Gaussian vectors. Assume that
\begin{equation}
\begin{cases} \E Y_i^2=\E X_i^2 &\qquad 1\le i\le n, \cr
\E Y_iY_j\le\E X_iX_j &\qquad 1\le i,j\le n.
\end{cases} 
\end{equation}
Then for any positive real number $x$,
\begin{equation}
 \P\Big\{\sup_{i=1}^n Y_i<x\Big\}\,\le \,\P\Big\{\sup_{i=1}^n X_i<x\Big\}.
\end{equation}
\end{lemma} 
We have stated a weak form of it, in fact one can replace $x>0 $ by   $x_i>0$.

\subsection{The case $\boldsymbol{  \l=1-\e ,  \e>0}$ small}\label{small}The problem mentioned in  Introduction can be   formulated as follows. Let    $0<c<1$.
Can we improve the trivial bound $\mathcal F_{n,\l} (\Theta)\le 1-\Psi(\Theta)$    to $\mathcal F_{n,\l} (\Theta)\le 1-c\kappa \Psi(\Theta)$ where $\kappa\gg 1$, and $c\kappa \Psi(\Theta)<1$? In which case   we would have for any integer $N$,
\beq \label{}
\mathcal F_{n,\l} (\Theta)^N\le (1-ck\Psi (\Theta))^N\le e^{-Nck\Psi (\Theta)}  . \eeq We use the following proposition, which is resulting from Lemma 10.1.2 in Weber \cite{We}, see  p.\,523.   \begin{proposition}\label{dsmp}Let $(X_1,\ldots, X_n)$ be a Gaussian centered vector. For all $x\ge 0$,  
$$
  1- \max_{j=1}^k \Bigg(
\sum_{i=1}^{j-1}\,\exp\bigg\{-{1\over 2}\bigg( {x\|X_i-X_j\|_2\over
2\|X_j\|^2}\bigg)^2\bigg\} \Bigg)\le
 {{{\P}\big\{\sup_{j=1}^k\,X_j>x\big\}}\over \ {\sum_{j=1}^{k}\,{\P}\big \{X_j>x\big\}}}\le 1.$$
\end{proposition} 
\noi  Proposition \ref{dsmp} is  concretely just  what the so-called double sum method ready-to-use is in the present context.
\vskip 3pt
For instance if $\|X_j\|=1$  for each $j$, and  the two following conditions are fulfilled: 
$$\ \,  {\rm (a)}\qq\ \  ck\Psi (\Theta)<1,\qq\qq\qq\qq\qq\qq\qq\quad\  $$ 
    $$\ \,  {\rm (b)}\quad    \max_{j=1}^k \Bigg(
\sum_{i=1}^{j-1}\,\exp\bigg\{-{1\over 2}\bigg( {\Theta \|X_i-X_j\|_2\over
2 }\bigg)^2\bigg\} \Bigg)\le  1-c,   $$
  then it follows that
\beq \label{dsmp.a}\P\big\{ \sup_{i=1}^k X_i> \Theta\big\}\ge ck\Psi (\Theta). 
\eeq Let $0<\e<1$. As an application we get that if   
 \begin{eqnarray*}{\rm (c)}&&\qq \E   X_i  X_j\le 1-\e, \qq  \forall  i\neq j ,   \qq\qq\qq\quad \cr  {\rm (d)}&&\ \, \qq    
c k \,e^{-{\e\over 2}   {\Theta^2     }   }     <1, \qq\qq\qq 
\end{eqnarray*}
\vskip -  2 pt \noi then as $\Psi (y)\sim 1/ye^{y^2/2}$, 
\vskip - 15 pt \beq \label{dsmp.1}\P\big\{ \sup_{i=1}^k    X_i\le  \Theta\big\} \le 1-ck\Psi (\Theta)  \le e^{- ck\Psi (\Theta)} \le   e^{ -\frac{ C c k\ }{\Theta \sqrt\e e^{\Theta^2\e/2 }} }. 
\eeq
\vskip -  2 pt \noi Taking $\Theta= \sqrt{(2/\e)\log (k\p(k))}$, $\p(k)>c$ and non-decreasing with $k$, gives 
 $\frac{c k}{e^{ { \e\over 2}   {\Theta^2    }   }  }=\frac{c}{\p(k)}<1,
 $
so that (d) is satisfied  and  
    $ \frac{ C  k }{\Theta \sqrt\e e^{\Theta^2\e/2 }} = \frac{ C    }{\Theta \sqrt\e  \p(k)}= \frac{ C    }{    \p(k) \sqrt {2\log(k \p(k))}}.$   We therefore get
\begin{proposition} \label{dsmp.2}Let $(X_1,\ldots, X_n)$ be a Gaussian centered vector such that $\|X_j\|_2=1$  for each $j$, and for some $0<\e<1$,  $\E   X_i  X_j\le 1-\e$,  for all $i\neq j$. Let $0<c<1$, $\p(k)>c$ and non-decreasing with $k$. Then for some absolute constant $C$, 
\vskip -  12 pt$$\P\Big\{ \sup_{i=1}^k   X_i\le  \sqrt{(2/\e)\log (k\p(k))}\Big\} \le
    e^{ -\frac{ C c   }{    \p(k) \sqrt {2\log(k \p(k))}}}. 
$$
\end{proposition}
\begin{remarks}\label{dsm} \rm (i)  Condition (c) better reads as $\|X_i-  X_j\|_2\ge \eta $   for all  $i\neq j$, $\eta=\sqrt{ 2 \e}$. 
\vskip 2 pt\noi  (ii) We refer to Section 1 of Chapter 10 in \cite{We}, where basic estimates of Gaussian tail distribution, correlation inequalities for Gaussian pairs (Chung, Erd\"os and Sirao...) and for Gaussian vectors are collected, and first useful applications are provided.
\end{remarks} 
  
 \subsection{Proofs of Theorem  \ref{Ac1c}-(i) and  Corollary \ref{ass1th}} We first recall  that
 $-\frac12 {}^t \underline{x}\, C(\l)^{-1}\underline{x}   \,=\,- \big(\frac{ 1}{2(1-\l)}\big) \big( \sum_{i=1}^n x_i^2- \tau \big(\sum_{i=1}^n x_i \big)^2 \big)$, 
where $\tau=\frac{\l  }{ (1-\l)+n\l } $. 
 Arrived at this point, it is necessary to return  to the assumption made and observe  that it has some interesting consequence concerning the squared sum 
 $$ \tau \Big(\sum_{i=1}^n x_i \Big)^2.$$
Note that     $\tau \le   \frac{1}{n }$ with equality if $\l=1$. We now would like to say that  a slightly strenghtened assumption  imply   the $L^2$-norm associated with this sum can be well-estimated. 
More precisely,  since
 $n-2n(n-1)\l \le \E  \big(\sum_{i=1}^n X_i \big)^2 =n+ 2\sum_{1\le i<j\le n} \E X_iX_j\le 
n+2n(n-1)\l, $
 it follows that 
   $$\E  \Big(\sum_{i=1}^n X_i \Big)^2 =(1+o(1))n,$$
 if $\l=o(1/n)$. This suggests to    introduce a truncation,  bounding as follows,
  \begin{align}\label{truncation}\P\Big\{ \sup_{i=1}^n X_i\le \Theta, \,&\frac{1}{n } \big|\sum_{i=1}^n X_i\big|^2\le M\Big\} \le\ \P\Big\{ \sup_{i=1}^n X_i\le \Theta,\, \tau \big|\sum_{i=1}^n X_i\big|^2\le M\Big\}
  \cr  &=  \ \int_{\sup_{i=1}^n x_i\le \Theta}  \frac{e^{-  \frac{1}{2(1-\l)}   ( \sum_{i=1}^n x_i^2- \tau  (\sum_{i=1}^n x_i  )^2  ) } }{(2\pi)^{n/2}[(1-\l)^{n-1}(1+\l(n-1))]^{1/2}}\,\dd x
 \cr &\le \  e^{\frac{M}{2(1-\l)}  }\int_{\sup_{i=1}^n x_i\le \Theta}  \frac{e^{-  \frac{1}{2(1-\l)}     \sum_{i=1}^n x_i^2  } }{(2\pi)^{n/2}[(1-\l)^{n-1}(1+\l(n-1))]^{1/2}}\,\dd x
 \cr  (x_i= y_i\sqrt{ 1-\l} )\quad&= \  \frac{\sqrt{ 1-\l} \,
 e^{\frac{M}{2(1-\l)} }} {  (1+\l(n-1)) ^{1/2}}\Big(\int_{  \sup_{i=1}^n y_i \le \frac{\Theta}{\sqrt{1-\l}}   }   e^{-  \frac{1}{2 }   \sum_{i=1}^n y_i  ^2       } \frac{\,\dd y}{(2\pi)^{1/2} }  \Big)^n
  \cr   &= \  \frac{\sqrt{ 1-\l} \,e^{\frac{M}{2(1-\l)}}}{  (1+\l(n-1)) ^{1/2}}\, \Phi  \Big(\frac{\Theta}{\sqrt{1-\l}}\Big)^n
  .
 \end{align}
Under assumption  \eqref{assumption1} we can thus write \begin{eqnarray}\label{bound3}
 \P\Big\{ \sup_{i=1}^n X_i\le \Theta \Big\}
 &\le & \P\Big\{ \sup_{i=1}^n X_i\le \Theta, \frac{\big|\sum_{i=1}^n X_i\big|^2}{n+1}  \le M\Big\}
  + \P\Big\{  \frac{\big|\sum_{i=1}^n X_i\big|^2}{n+1}  > M\Big\}
 \cr \big(V=\frac{\sqrt{M(n+1)}}{\big\|\sum_{i=1}^n X_i\big\|_2}\big)&\le &\frac{\sqrt{1-\l}\,e^{\frac{M}{2(1-\l)}}}{  (1+\l(n-1)) ^{1/2}}\, \Phi  \Big(\frac{\Theta}{\sqrt{1-\l}}\Big)^n+ \Psi \big(V\big)
.\cr &&
\end{eqnarray} 
By Lemma \ref{boyd.mills}-(1),
$$\Phi  \Big(\frac{\Theta}{\sqrt{1-\l}}\Big)^n=\Big(1-\Psi\Big(\frac{\Theta}{ \sqrt{1-\l}}\Big)\Big)^n\le  \exp\Big\{ -n\Psi\Big(\frac{\Theta}{ \sqrt{1-\l }}\Big)\Big\} 
\le   \exp\Big\{ -\frac{nC  \sqrt{1-\l } }{\Theta e^{\Theta^2/2(1-\l)}}\Big\}.$$   
  Let $n$ be such that $2n\l\le \eta<1$. Then,       $  \E  \big(\sum_{i=1}^n X_i \big)^2 \le   
n+2n(n-1)\l\le n(1+\eta)$,   so that $V\ge  \sqrt{M(n+1)/n(1+\eta)} \ge\sqrt{M / (1+\eta)} $. Thus  $\Psi(V)\le \Psi(\sqrt{M / (1+\eta)} )$, and since by Lemma \ref{boyd.mills}-(1) again,  $\Psi(x)\le Cx^{-1}e^{-x^2/2}$, we have 
$$\Psi(\sqrt{M / (1+\eta)} )\le C_\eta \,e^{-M/2(1+\eta)}/\sqrt{M}.$$  Therefore
     \begin{eqnarray*}
 \P\Big\{ \sup_{i=1}^n X_i\le \Theta \Big\}
 &\le & C  e^{\frac{M}{2(1-\l)} -\frac{nC  \sqrt{1-\l } }{\Theta e^{\Theta^2/2(1-\l)}} }+C_\eta \,e^{-M/2(1+\eta)}/\sqrt{M}   .
\end{eqnarray*}   

Choose $\Theta=   \sqrt{    2(1-\l)\log \frac{n}{   \log n}   }$, then
$$
\frac{n   \sqrt{1-\l } }{\Theta e^{\Theta^2/2(1-\l)}}=   \frac{  n}{\frac{n}{   \log n} \sqrt{    2\log \frac{n}{   \log n}   } } =    \frac{ \log n  }{ \sqrt{    2\log \frac{n}{   \log n}   } } 
\ge   \sqrt{(\log n ) /2}
 $$
Thus      \begin{eqnarray*}
 \P\Big\{ \sup_{i=1}^n X_i\le \sqrt{    2(1-\l)\log \frac{n}{   \log n}   }\Big\}
 &\le & C e^{\frac{M}{2(1-\l)} -\sqrt{(\log n ) /2}}+C_\eta \,e^{-M/2(1+\eta)}/\sqrt{M}.
\end{eqnarray*}
  Letting  finally  $M=\a  \sqrt{ \log n  }$, where we choose $\a$ so that
 \beq\label{eta}\frac{\a(2-\l+\eta)}{(1-\l)(1+\eta)}<\sqrt 2 \eeq
 we consequently get
   \begin{eqnarray*}
 \P\Big\{ \sup_{i=1}^n X_i\le \sqrt{    2(1-\l)\log \frac{n}{   \log n}   }\Big\}
 &\le & C \, e^{\frac{\a  \sqrt{ \log n  }}{2(1-\l)} -\sqrt{(\log n ) /2}}+\frac{C_\eta}{\sqrt{\a}} \,\frac{e^{-{\frac{\a  \sqrt{ \log n  }}{2(1+\eta)}}}}{   ( \log n  )^{1/4} }
 \cr   &\le & C\, e^{-(\frac{ 1}{\sqrt2}-\frac{   \a  }{2(1-\l)}    ) \sqrt{ \log n  }}+\frac{C_\eta}{\sqrt{\a}} \,\frac{e^{-{\frac{\a  \sqrt{ \log n  }}{2(1+\eta)}}}}{   ( \log n  )^{1/4} }
 \cr   &\le & C_\eta\, e^{-{\frac{\a  \sqrt{ \log n  }}{2(1+\eta)}}}.
\end{eqnarray*}
Now as  
  $$ \qquad 2\l \log \frac{n}{    \log n}\le     \frac{ \eta  \log n}{     n} ,\qq \sqrt{    2 \log \frac{n}{   \log n} - 2\l \log \frac{n}{   \log n}}\ge \sqrt{    2 \log \frac{n}{   \log n} - \frac{ \eta  \log n}{     n} },$$
it follows that
  \begin{eqnarray*}
   \P\Big\{ \sup_{i=1}^n X_i\le \sqrt{    2 \log \frac{n}{   \log n} - \frac{ \eta  \log n}{     n} }\Big\}
   &\le & C_\eta\, e^{- C'_\eta \sqrt{ \log n  }  }.\end{eqnarray*}  This concludes the proof of   Corollary \ref{ass1th}.
 \vskip 2 pt 
Next  under the initial assumption  \eqref{assumption}, we proceed as follows. 
First by Cauchy-Schwarz's inequality,
$\tau  (\sum_{i=1}^n x_i  )^2 \le \tau n \sum_{i=1}^n x_i^2 =\frac{\l n  }{ (1-\l)+n\l }\sum_{i=1}^n x_i^2 $.
Thus$$ \sum_{i=1}^n x_i^2- \tau  (\sum_{i=1}^n x_i  )^2   \ge  \big(\sum_{i=1}^n x_i^2\big)\big(1- \frac{\l n  }{ (1-\l)+n\l }\big) = \big(\sum_{i=1}^n x_i^2\big) \big( \frac{ 1-\l  }{  1+\l (n-1) }\big). $$

This will induce in place of the factor $\frac{1}{ 1-\l} $,   the less satisfactory factor   $\frac{1}{  1+\l (n-1)} $, as it follows below,
    \begin{eqnarray*}\P\Big\{ \sup_{i=1}^n X_i\le \Theta\Big\}&= &\int_{\sup_{i=1}^n x_i\le \Theta}  \frac{e^{-  \frac{1}{2(1-\l)}   ( \sum_{i=1}^n x_i^2- \tau  (\sum_{i=1}^n x_i  )^2  ) } }{(2\pi)^{n/2}[(1-\l)^{n-1}(1+\l(n-1))]^{1/2}}\,\dd x
 \cr &\le &\int_{\sup_{i=1}^n x_i\le \Theta}  \frac{e^{-  \frac{ 1}{2( 1+\l(n-1) ) }       (\sum_{i=1}^n x_i^2 )} }{(2\pi)^{n/2}[(1-\l)^{n-1}(1+\l(n-1))]^{1/2}}\,\dd x 
 \cr &= &\Big(\frac{1-\l}{   1+\l(n-1)   }\Big)^{1/2}\Big(\int_{  x \le \Theta}   e^{-  \frac{ x^2}{2( 1+\l(n-1) ) }         } \frac{\,\dd x}{ \sqrt{2\pi(1-\l)}}  \Big)^n
\cr(x_i=y_i\sqrt{ 1+\l(n-1) }\,)\quad  
 &= &\Big(1+\frac{  \l n  }{   1-\l }\Big)^{(n-1)/2}\, \Phi  \Big(\frac{\Theta}{\sqrt{ 1+\l(n-1) }}\Big)^n  ,
 \end{eqnarray*}
which proves Theorem  \ref{Ac1c}-(i). Note that the front factor may grow  exponentially even if $n\l =o(1)$.


 \subsection{Proof of Theorem \ref{Ac1c}-(ii)}
  By using the well-known bound  $ {-u\over  1-u }\le \log 1-u \le -u $,  $0\le u< 1$, we get with $u=\Psi(y)$, $y\in\R$,
\begin{eqnarray}\label{lambda=0}   e^{- n \frac{\Psi(y)}{1-\Psi(y)} }\,\le \, \P\Big\{ \sup_{i=1}^n X_i\le y\Big\}&=& e^{n\log (1-\Psi(y))}\,\le \,  e^{- n \Psi(y) }.
  \end{eqnarray} 
  Thus for $0<\e<1$, for $y$ suitably large, $y\ge y_\e$ say,
  $$e^{-   \frac{n }{1-\e}\Psi(y) }\,\le \, \P\Big\{ \sup_{i=1}^n X_i\le y\Big\}$$
  Recall Boyd's 
   refinement of Mills' ratio.
 \begin{lemma}\label{boyd.mills}
 Let for $y\ge 0$, 
  $$ R_1(y)={\pi\over \sqrt{y^2+2\pi} +(\pi -1)y}, \qq R_2(y)= {\pi\over \sqrt{(\pi -2)^2y^2 +2\pi} +2y} .$$
  Then the following assertions are true,
   \begin{itemize}
 \leftskip6mm
 \itemsep=4pt
 \item[{\rm (1)}\ ]  For all $y\ge 0$,   $R_1(y)\le
 e^{y^2/2}\int_y^\infty e^{-t^2/2}\ \dd t\le
R_2(y)$, 
 \item[{\rm (2)}\ ] $\Psi(y)= \int_y^\infty e^{-t^2/2}\, \frac{\dd t}{\sqrt{2\pi}}\sim {  e^{- y^2/2} \over y\sqrt{2\pi}   }, \qq y\to +\infty,$
 \item[{\rm (3)}\ ] $\lim_{y\to 0}R_1(y)= \lim_{y\to 0}R_2(y)=\sqrt{{\pi\over 2}}$.
\end{itemize}
\end{lemma}
  
See Mitrinovi\'c \cite[Sec.\,2.26]{Mi}.  \begin{proposition}\label{Ac1b} Let us consider  the generic case $\l=0$. Given any positive real $\e$, for $n$ sufficiently large, $n\ge n(\e)$,  and all $x$ such that  $x\ge -b_n^2$,
\begin{eqnarray*}\P\Big\{ \sup_{i=1}^n X_i\le \Big({\frac{1}{b_n}x+ b_n}\Big) {\sqrt{ 1+\l(n-1)}} \Big\}     \,\le \, \,    e^{- e^{-x} (1-\e) }      .
 \end{eqnarray*}
 \end{proposition}

 \begin{proof}
  Let $X= \frac{1}{b_n}x+ b_n$, noting that $b_n \sim \sqrt{2\log n}$.   Apply asymptotics (2) with the choice $y=X$, assuming $X\ge 0$, or else $x\ge -b_n^2$. We get for $X$ large
 \begin{eqnarray*}\Psi(X) =(1+o(1)) {  e^{- X^2/2} \over X\sqrt{2\pi}   }  &=& (1+o(1)) {  e^{-x -(x^2/b_n^2 + b_n^2 )/2} \over  ((x/b_n)+ b_n)\sqrt{2\pi}   }. \end{eqnarray*}
Hence 
\begin{eqnarray}\label{lambda=0.}    e^{- n \Psi(X) }\le     e^{- e^{-x}\frac{n (1+o(1))}{((x/b_n)+ b_n)\sqrt{2\pi}   }  \, e^{  -(x^2/b_n^2 + b_n^2 )/2}   }  .
  \end{eqnarray}
     As $x/{b_n}+ b_n$ is large,    the righthand side is  
$$ \le\  e^{- n (1+o(1)) \big(\frac{ 1}{  b_n e^{   b_n^2/2}\sqrt{2\pi}   } \big)\, e^{-x}   }  .$$
Choosing 
 $ b_n= \big( \log \frac{n^2}{4\pi \log n} \big)^{1/2} $,  
  we have $b_n e^{   b_n^2/2}= \big( \log \frac{n}{\sqrt{4\pi \log n} }\big)^{1/2}\frac{n}{\sqrt{4\pi \log n}}$, so that
 
$$\frac{ n}{  b_n e^{   b_n^2/2}\sqrt{2\pi}   }=\Big( \frac{\log n}{\log \frac{n}{ 4\pi \log n} }\Big)^{1/2}\ge 1 -o(1)$$
for $n$ large. 
 
\vskip 3 pt Given any positive real $\e$, for $n$ sufficiently large, $n\ge n(\e)$,  and all $x$ such that  $x\ge -b_n^2$,
\begin{eqnarray}   e^{- n \Psi(X) }  \le\  e^{- n (1+o(1)) \big(\frac{ 1}{  b_n e^{   b_n^2/2}\sqrt{2\pi}   } \big)\, e^{-x}   }\le \     e^{- e^{-x} (1-\e)   }.
  \end{eqnarray}
    This along with Eq. \ref{lambda=0} imply 
we have the semi-asymptotic bound: given any positive real $\e$, for $n$ sufficiently large, $n\ge n(\e)$,  and all $x$ such that  $x\ge -b_n^2$,
\begin{eqnarray}\label{lambda=0.estim}   \P\Big\{ \sup_{i=1}^n X_i\le \frac{1}{b_n}x+ b_n\Big\} \le    e^{- n \Psi(X) }    \le e^{- e^{-x} (1-\e)   }.
  \end{eqnarray}

\end{proof}

  Let now $X=(X_1,\ldots X_n)$ be a centered Gaussian vector with covariance matrix   $C(\l)$, and recall that for any real   $\Theta$ and any $0\le \l <1$,  setting  $\Theta '= {\Theta}/{\sqrt{ 1+\l(n-1)}} $,
  
\begin{eqnarray*}\P\Big\{ \sup_{i=1}^n X_i\le {\Theta'} {\sqrt{ 1+\l(n-1)}}\Big\}   \,\le \, \Big(1+\frac{  \l n  }{   1-\l }\Big)^{(n-1)/2}\Phi (\Theta' )^n      \,\le \, \Big(1+\frac{  \l n  }{   1-\l }\Big)^{(n-1)/2}\, e^{- n \Psi(\Theta' ) }   .
 \end{eqnarray*}
  Take $\Theta'=X$ as above defined. By the calculations made before:  given any positive real $\e$, for $n$ sufficiently large, $n\ge n(\e)$,  and all $x$ such that  $x\ge -b_n^2$,
\begin{eqnarray*}\P\Big\{ \sup_{i=1}^n X_i\le \Big({\frac{1}{b_n}x+ b_n}\Big) {\sqrt{ 1+\l(n-1)}} \Big\}     \,\le \, \Big(1+\frac{  \l n  }{   1-\l }\Big)^{(n-1)/2}   e^{- e^{-x} (1-\e)   }     .
 \end{eqnarray*}
    Hence we have proved Theorem \ref{Ac1c}-(ii). 

\section{\bf Proof of Theorem \ref{th.1}} \label{s4}
 


    
  Let $0<\e<1$ and $\{s^n_i,1\le i\le k\}\subset T$, $n=1,\ldots,N$ and assume that $\E   X(s^n_i) X(s^n_j)\le 1-\e$, $i\neq j$, $1\le n\le N$. 
  By Proposition \ref{dsmp.2}
  it follows that for $1\le n\le N$,
 \beq\label{dsmp.2.}\P\Big\{ \sup_{i=1}^k   X(s^n_i)\le  \sqrt{(2/\e)\log (k\p(k))}\Big\} \le
    e^{ -\frac{ C  c  }{    \p(k) \sqrt {2\log(k \p(k))}}}. 
\eeq

We now recall the second (complementary) condition: 
   for some   positive real $\l<1$, we have  
$$|\E  X( s^n_i)  X( s^m_j)|\le \l, $$    
for any  $1\le n\neq m\le N$, $n$, $m$ even and any $1\le i, j\le k$. 
\vskip 2 pt We use a         
 forgotten  lemma.
We have noted for any function $g:\R^2\to \R$ and any real $h$ the difference  operators
 \begin{eqnarray*}\ \D^1_{h}\ g(u,v)= g(u+h,v)-g(u,v),  \qq\qq\qq\qq
 \cr &\cr\ \ \ \D^2_{h}\ g(u,v)= g(u ,v+h)-g(u,v).\qq\qq\qq\qq\,  \end{eqnarray*}


 \begin{lemma} \label{comp.gp}
Let $X=(X_1,\dots ,X_N)$ be a Gauss\-ian
centered vector such that $\E X_n^2=1$ for $1\le n\le N$ and let
$r(n,m)=\E X_nX_m$ be  its covariance function. Let $A$ be a partition
of $\{1,\dots, N\}$ and denote by $\s$ a generic element of $A$.
Let $x=(x_1,\dots ,x_N)$ and $y =(y_1,\dots ,y_N)$ with distinct
coordinates, be such that $-\infty\le x_n< y_n\le  +\infty $, for $1\le
n\le N$. Denote also by $I_n$ the interval $(x_n,y_n)$, and put for each
$\s\in A$,
$$
V_\s= \prod_{n\in \s} I_n,\qq V= \prod_{\s\in A} V_\s, \qq X_{(\s)}= (X_n, n\in \s).
$$
Then 
$$\Big|\P\{ X\in V\}-\prod_{\s\in A}\P\{ X_{(\s)}\in V_\s\}\Big|\le \frac12\sum_{\s\not=\s'}
\sum_{n\in \s }\sum_{m\in  \s'} k(n,m)|r(n,m)|,
$$
where 
$$ k(n,m)=\int_0^1\D^1_{y_n-x_n}\circ\D^2_{y_m-x_m}\big(\Phi(x_n,x_m,\l r(n,m))\big) \dd \l,$$
and 
$$ \Phi(x,y,\rho)=\frac{1}{2\pi\sqrt{1-\rho^2}}e^{-\frac{x^2+y^2-2\rho xy}{2(1-\rho^2)} }.$$
\end{lemma}
\noi See Weber \cite{W3m},\,Lemma\,1, also references therein.
  It is well-known that $\Phi(x ,x ,r)=(2\pi)^{-1}(1-r^2)^{-1/2}\exp\big\{-\frac{ x^2 }{  1+r  }\big\} $ is  a monotonic   function of $r$ on $[0,1]$. This implies in particular if for some real $\rho\in]0,1[$,
$$ |r(n,m)|\le \rho, \qq n\neq m,$$ 
 that 
 \begin{eqnarray}\label{comparison}\Big|\P\{  \sup_{n\in A} X_{n}\le x\}-\prod_{\s\in A}\P\{ \sup_{n\in \s} X_{n}\le x\}\Big|
 &\le &\frac12\sum_{\s\not=\s'}\sum_{n\in \s }\sum_{m\in  \s'}  \frac{|r(n,m|}{2\pi\sqrt{1-r(n,m)^2}} \,e^{-\frac{ x^2 }{ (1+r(n,m) )} }\cr
 &\le &  \frac{\rho}{4\pi\sqrt{1-\rho^2}}   e^{-\frac{ x^2 }{ (1+\rho )} } \Big( \sum_{\s\in A}
 \#(\s)^2\Big)   .
\end{eqnarray}
It follows 
 in view of  \eqref{dsmp.2.}, by taking $\rho= \l$ that
 \begin{align*}  \P\Big\{  \sup_{n=1}^N   \sup_{i=1}^k    X(s^n_i)\le   \sqrt{(2/\e)\log (k\p(k))} \Big\} \le 
e^{ -\frac{ CcN   }{    \p(k) \sqrt {2\log(k \p(k))}}}
   +    \frac{\l k^2N}{4\pi\sqrt{1-\l^2}}   e^{-\frac{  2 \log (k\p(k))  }{\e (1+\l )} } .\end{align*}
This   proves Theorem \ref{th.1}.


  
 \section{\bf The matrix $\boldsymbol{ C_{k,N}(u,\l)}$} \label{s5}
Let $\l,u $ such that $0<\l\le u<1$, $k,N$ be positive integers and let $\mathcal C_{k,N}(u,\l) $ be the convex set  of $Nk\times Nk$ matrices $ C_{k,N}(u,\l)$    defined in \eqref{ulamb}, and which we have noted 
  more simply $ C (u,\l)$ in the sequel. Let $X =\{  X(t), t\in T\}$ with unit variance  satisfying  the following restricted assumption \eqref{assumption2},
\vskip 2 pt {\it There exists a finite partition $\{T_j,j\in J\}$ of $T$, with $\#\{T_j\}=k$ for each $j$, $\#\{T \}=Nk$, such that}
\ben  \label{assumption2a} \begin{cases} \E  X(s)  X(t)  = u  \qquad \forall {s\neq t\in T_i}, \ \forall i\in J, \cr \E   X(s)  X (t) =  \l  \qquad  \forall {s \in T_i, \forall t\in T_j}, \ \forall i,j\in J, \, i\neq j.
\end{cases} 
\een
Put 
\beq\label{FulkN}\mathcal F_{u \l,k N}(\Theta)=\P\Big\{\sup_{t\in T}    X(t)\le \Theta  \Big\}, \qq \Theta\in \R.
\eeq

 In liaison with   $\mathcal F_{k,u}$ defined in \eqref{Fnl}, we will also prove the decoupling inequality 
 \begin{theorem}\label{F.decoupling.th}   For any real $\Theta$, any positive real $M$,
 \begin{eqnarray}\label{F.decoupling}
      \mathcal F_{u\l, kN}(\Theta)
    &\le  & e^{M }\mathcal F_{ u, k }(\Theta)^N 
       + C\frac {\sqrt{  A}\|\sum_{ v=1}^{Nk} X_{v}\|_2}{\sqrt{ M }
 }e^{-(\frac{ { M }}{ 2{  A}})/\|\sum_{ v=1}^{Nk} X_{v}\|_2^2  }.\end{eqnarray}
 where $A= \frac{   \l  }{2[1-u+ k(u-\l)][ (1-u)+Nku +(N-1)k\l ]}$.
  \end{theorem}

 \subsection{Matrix apparatus}   
  \subsubsection{\tf Calculation of the determinant}
We compute
    {\tiny      \leftskip =0,4cm \rightskip=0,4cm    
\[  \det(C(u,\l)) \,=\, \left| \begin{array}{ c c  c c c | c c c c c | c c c c c | c   c c c c c c c c c c c c}
1    & u  & u  &   \cdots    &u& \l&\l &\l &   \cdots        &\l & \l & \l &\l&\cdots    &\l& \l &     \cdots 
\\	
  u  &1   &u  &  \cdots &  u      & \l&\l &\l &   \cdots        &\l &  \l&\l&\l&  \cdots    &\l&   \l&    \cdots
  \cr
u         &    u    & 1 & \vdots &\vdots        &    \vdots   &  \vdots &  \vdots& \vdots & \l  &
\vdots         &    \vdots   &  \vdots & \vdots &\vdots&\vdots &\vdots 
\cr  \vdots         &    \vdots   &  \vdots & \ddots &\vdots        &    \vdots   &  \vdots &  \vdots& \ddots & \l &    \vdots&\vdots&\vdots   &\ddots    &\l& \vdots& \vdots
\cr  u&u&u&\cdots &1 &\l&\l &\l&\cdots &\l&    \l&\l&\l   &\cdots    &\l& \vdots& \vdots     
  \cr  \hline \l & \l &\l   &\cdots    &\l& 1     & u  & u  &   \cdots        &u   &
   \l &\l&\l   &\cdots    &\l&\vdots   &\vdots  
    \cr  \l & \l &\l   &\cdots    &\l&   u &1 & u  &   \cdots    &u  &  \vdots   &    \vdots   &  \vdots & \vdots          &    \l   &  \vdots &  \vdots     \cr 
     \vdots   &    \vdots   &  \vdots & \vdots          &    \l   &  u &  u& 1 &\vdots&u  
& \vdots         &    \vdots   &  \vdots & \vdots &\vdots        &     \vdots &     \vdots \cr
 \vdots         &    \vdots   &  \vdots & \ddots &\vdots        &    \vdots   &  \vdots &  \vdots& \ddots &u &\vdots    & \vdots  &\vdots  &   \ddots    &\l& \l &\cdots  
 \cr  \l  &\l & \l &\cdots &\l&u &u&u &\cdots    &1  &\l  &\l   &\l  &  \cdots & \l     &\l&\cdots
\cr \hline
 \l & \l &\l  &\cdots    &\l&\l&\l &\l &   \cdots        &\l &1         &   u   &  u & \cdots &u        &   \l   &\cdots
 \cr \l & \l &\l  &\cdots    &\l&\l&\l &\l &   \cdots        &\l &u         &   1   &  u & \cdots &u        &    \vdots   &\ddots\cr 
 \vdots    & \vdots  &\vdots&\cdots    & \vdots  &\vdots& \vdots    & \vdots  &\vdots& \vdots    & \vdots  &\vdots&\vdots    & \ddots  &\vdots&    \vdots   &\cdots
\end{array} \right|
\]}
We replace the $(k-1)$ first columns $C_\ell$ by $C_\ell-C_{\ell +1}$. These linear transformations let the determinant invariant.  We get 
 {\tiny      \leftskip =0,4cm \rightskip=0,4cm    
\[  \det(C(u,\l)) \,=\, \left| \begin{array}{ c c  c c c | c c c c c | c c c c c | c   c c c c c c c c c c c c}
1\=u   & 0  & 0  &   \cdots    &u& \l&\l &\l &   \cdots        &\l & \l & \l &\l&\cdots    &\l& \l &     \cdots 
\\	
  u\=1  &1  \=u &0  &  \cdots &  u      & \l&\l &\l &   \cdots        &\l &  \l&\l&\l&  \cdots    &\l&   \l&    \cdots
  \cr
0         &    u \=1   & 1\=u & \vdots &\vdots        &    \vdots   &  \vdots &  \vdots& \vdots & \l  &
\vdots         &    \vdots   &  \vdots & \vdots &\vdots&\vdots &\vdots 
\cr  \vdots         &    \vdots   &  \vdots & \ddots &\vdots        &    \vdots   &  \vdots &  \vdots& \ddots & \l &    \vdots&\vdots&\vdots   &\ddots    &\l& \vdots& \vdots
\cr  0&0&\cdots&u\=1 &1 &\l&\l &\l&\cdots &\l&    \l&\l&\l   &\cdots    &\l& \vdots& \vdots     
  \cr  \hline0 & 0 &0   &\cdots    &\l& 1     & u  & u  &   \cdots        &u   &
   \l &\l&\l   &\cdots    &\l&\vdots   &\vdots  
    \cr  0 & 0 &0   &\cdots    &\l&   u &1 & u  &   \cdots    &u  &  \vdots   &    \vdots   &  \vdots & \vdots          &    \l   &  \vdots &  \vdots     \cr 
     \vdots   &    \vdots   &  \vdots & \vdots          &    \l   &  u &  u& 1 &\vdots&u  
& \vdots         &    \vdots   &  \vdots & \vdots &\vdots        &     \vdots &     \vdots \cr
 \vdots         &    \vdots   &  \vdots & \ddots &\vdots        &    \vdots   &  \vdots &  \vdots& \ddots &u &\vdots    & \vdots  &\vdots  &   \ddots    &\l& \l &\cdots  
 \cr  0  &0 & 0 &\cdots &\l&u &u&u &\cdots    &1  &\l  &\l   &\l  &  \cdots & \l     &\l&\cdots
\cr \hline
0 &0 &0  &\cdots    &\l&\l&\l &\l &   \cdots        &\l &1         &   u   &  u & \cdots &u        &   \l   &\cdots
 \cr 0 & 0 &0  &\cdots    &\l&\l&\l &\l &   \cdots        &\l &u         &   1   &  u & \cdots &u        &    \vdots   &\ddots\cr 
 \vdots    & \vdots  &\vdots&\cdots    & \vdots  &\vdots& \vdots    & \vdots  &\vdots& \vdots    & \vdots  &\vdots&\vdots    & \ddots  &\vdots&    \vdots   &\cdots
\end{array} \right|
\]}
We note that the upper left block was $C(u)$, and the linear transformations we just made let its determinant invariant; so that the new upper left block has determinant equal to $\det(C(u))$.
  Next we replace the $k$-th row by the sum of the $k$ first rows. This is not affecting the determinant value of  the whole matrix, neither of  the one of the upper left block.  Further we multiply the  new $k$-th row by the factor $ {\l}/({1+(k-1)u)}$. This is   affecting the determinant of the whole matrix, also separately the one of the upper left block by a factor $ {\l}/({1+(k-1)u)}$. We   set $\l_u= {k\l^2}/({1+(k-1)u})$. The new $k$-th row now begin with ($k-1$) zeros. We can write
 {\tiny      \leftskip =0,4cm \rightskip=0,4cm    
\[    \,=\, \frac{1\+(k\=1)u}{\l} \left| \begin{array}{ c c  c c c | c c c c c | c c c c c | c   c c c c c c c c c c c c}
1 \=u   & 0  & 0  &   \cdots    &u&\l&\l &\l &   \cdots        &\l & \l & \l &\l&\cdots    &\l& \l &     \cdots 
\\	
  u\=1  &1  \=u &0  &  \cdots &  u      & \l&\l &\l &   \cdots        &\l &  \l&\l&\l&  \cdots    &\l&   \l&    \cdots
  \cr
0         &    u \=1   & 1\=u & \vdots &\vdots        &    \vdots   &  \vdots &  \vdots& \vdots & \l  &
\vdots         &    \vdots   &  \vdots & \vdots &\vdots&\vdots &\vdots 
\cr  \vdots         &    \vdots   &  \vdots & \ddots &\vdots        &    \vdots   &  \vdots &  \vdots& \ddots & \l &    \vdots&\vdots&\vdots   &\ddots    &\l& \vdots& \vdots
\cr  0&0&\cdots&0 &\l &\l_u&\l_u &\l_u&\cdots &\l_u&    \l_u&\l_u&\l_u   &\cdots    &\l_u& \l_u& \l_u   
  \cr  \hline0 & 0 &0   &\cdots    &\l& 1     & u  & u  &   \cdots        &u   &
   \l &\l&\l   &\cdots    &\l&\vdots   &\vdots  
    \cr  0 & 0 &0   &\cdots    &\l&   u &1 & u  &   \cdots    &u  &  \vdots   &    \vdots   &  \vdots & \vdots          &    \l   &  \vdots &  \vdots     \cr 
     \vdots   &    \vdots   &  \vdots & \vdots          &    \l   &  u &  u& 1 &\vdots&u  
& \vdots         &    \vdots   &  \vdots & \vdots &\vdots        &     \vdots &     \vdots \cr
 \vdots         &    \vdots   &  \vdots & \ddots &\vdots        &    \vdots   &  \vdots &  \vdots& \ddots &u &\vdots    & \vdots  &\vdots  &   \ddots    &\l& \l &\cdots  
 \cr  0  &0 & 0 &\cdots &\l&u &u&\cdots & u    &1  &\l  &\l   &\l  &  \cdots & \l     &\l&\cdots
\cr \hline
0 &0 &0  &\cdots    &\l&\l&\l &\l &   \cdots        &\l &1         &   u   &  u & \cdots &u        &   \l   &\cdots
 \cr 0 & 0 &0  &\cdots    &\l&\l&\l &\l &   \cdots        &\l &u         &   1   &  u & \cdots &u        &    \vdots   &\ddots\cr 
 \vdots    & \vdots  &\vdots&\cdots    & \vdots  &\vdots& \vdots    & \vdots  &\vdots& \vdots    & \vdots  &\vdots&\vdots    & \ddots  &\vdots&    \vdots   &\cdots
\end{array} \right|
\]}
Now we repeat the same determinant-preserving linear transformations, limited this time to the second group of $k$ columns. We successively obtain
{\tiny      \leftskip =0,4cm \rightskip=0,4cm    
\[  \    \,=\, \Big(\frac{1\+(k\=1)u}{\l}\Big)\left| \begin{array}{ c c  c c c | c c c c c | c c c c c | c   c c c c c c c c c c c c}
1 \=u   & 0  & 0  &   \cdots    &u& 0&0 &0 &   \cdots        &\l & \l & \l &\l&\cdots    &\l& \l &     \cdots 
\\	
  u\=1  &1  \=u &0  &  \cdots &  u      & 0&0 &0 &   \cdots        &\l &  \l&\l&\l&  \cdots    &\l&   \l&    \cdots
  \cr
0         &    u \=1   & 1\=u & \vdots &\vdots        &    \vdots   &  \vdots &  \vdots& \vdots & \l  &
\vdots         &    \vdots   &  \vdots & \vdots &\vdots&\vdots &\vdots 
\cr  \vdots         &    \vdots   &  \vdots & \ddots &\vdots        &    \vdots   &  \vdots &  \vdots& \ddots & \l &    \vdots&\vdots&\vdots   &\ddots    &\l& \vdots& \vdots
\cr  0&0&\cdots&0 &\l &0&0 &0&\cdots &\l_u&    \l_u&\l_u&\l_u   &\cdots    &\l_u& \l_u& \l_u   
  \cr  \hline0 & 0 &0   &\cdots    &\l& 1 \=u    &  0 & 0  &   \cdots        &u   &
   \l &\l&\l   &\cdots    &\l&\vdots   &\vdots  
    \cr  0 & 0 &0   &\cdots    &\l&   u\=1 &1\=u & 0  &   \cdots    &u  &  \vdots   &    \vdots   &  \vdots & \vdots          &    \l   &  \vdots &  \vdots     \cr 
     \vdots   &    \vdots   &  \vdots & \vdots          &    \l   &  0 &  u\=1& 1\=u &\vdots&u  
& \vdots         &    \vdots   &  \vdots & \vdots &\vdots        &     \vdots &     \vdots \cr
 \vdots         &    \vdots   &  \vdots & \ddots &\vdots        &    \vdots   &  \vdots &  \vdots& \ddots &u &\vdots    & \vdots  &\vdots  &   \ddots    &\l& \l &\cdots  
 \cr  0  &0 & 0 &\cdots &\l&0 &0&\cdots & u\=1    &1  &\l  &\l   &\l  &  \cdots & \l     &\l&\cdots
\cr \hline
0 &0 &0  &\cdots    &\l&0&0 &0 &   \cdots        &\l &1         &   u   &  u & \cdots &u        &   \l   &\cdots
 \cr 0 & 0 &0  &\cdots    &\l&0&0 &0 &   \cdots        &\l &u         &   1   &  u & \cdots &u        &    \vdots   &\ddots\cr 
 \vdots    & \vdots  &\vdots&\cdots    & \vdots  &\vdots& \vdots    & \vdots  &\vdots& \vdots    & \vdots  &\vdots&\vdots    & \ddots  &\vdots&    \vdots   &\cdots
\end{array} \right|
\]
 \leftskip =0,4cm \rightskip=0,4cm    
\[ \,=\, \Big(\frac{1\+(k\=1)u}{\l}\Big)^2\left| \begin{array}{ c c  c c c | c c c c c | c c c c c | c   c c c c c c c c c c c c}
1 \=u   & 0  & 0  &   \cdots    &u& 0&0 &0 &   \cdots        &\l & \l & \l &\l&\cdots    &\l& \l &     \cdots 
\\	
  u\=1  &1  \=u &0  &  \cdots &  u      & 0&0 &0 &   \cdots        &\l &  \l&\l&\l&  \cdots    &\l&   \l&    \cdots
  \cr
0         &    u \=1   & 1\=u & \vdots &\vdots        &    \vdots   &  \vdots &  \vdots& \vdots & \l  &
\vdots         &    \vdots   &  \vdots & \vdots &\vdots&\vdots &\vdots 
\cr  \vdots         &    \vdots   &  \vdots & \ddots &\vdots        &    \vdots   &  \vdots &  \vdots& \ddots & \l &    \vdots&\vdots&\vdots   &\ddots    &\l& \vdots& \vdots
\cr  0&0&\cdots&0 &\l &0&0 &0&\cdots &\l_u&    \l_u&\l_u&\l_u   &\cdots    &\l_u& \l_u& \l_u   
  \cr  \hline0 & 0 &0   &\cdots    &\l& 1 \=u    &  0 & 0  &   \cdots        &u   &
   \l &\l&\l   &\cdots    &\l&\vdots   &\vdots  
    \cr  0 & 0 &0   &\cdots    &\l&   u\=1 &1\=u & 0  &   \cdots    &u  &  \vdots   &    \vdots   &  \vdots & \vdots          &    \l   &  \vdots &  \vdots     \cr 
     \vdots   &    \vdots   &  \vdots & \vdots          &    \l   &  0 &  u\=1& 1\=u &\vdots&u  
& \vdots         &    \vdots   &  \vdots & \vdots &\vdots        &     \vdots &     \vdots \cr
 \vdots         &    \vdots   &  \vdots & \ddots &\vdots        &    \vdots   &  \vdots &  \vdots& \ddots &u &\vdots    & \vdots  &\vdots  &   \ddots    &\l& \l &\cdots  
 \cr  0  &0 & 0 &\cdots &\l&0 &0&0 &\cdots    &\l  &\l_u  &\l_u   &\l_u  &  \cdots & \l_u     &\l_u&\cdots
\cr \hline
0 &0 &0  &\cdots    &\l&0&0 &0 &   \cdots        &\l &1         &   u   &  u & \cdots &u        &   \l   &\cdots
 \cr 0 & 0 &0  &\cdots    &\l&0&0 &0 &   \cdots        &\l &u         &   1   &  u & \cdots &u        &    \vdots   &\ddots\cr 
 \vdots    & \vdots  &\vdots&\cdots    & \vdots  &\vdots& \vdots    & \vdots  &\vdots& \vdots    & \vdots  &\vdots&\vdots    & \ddots  &\vdots&    \vdots   &\cdots
\end{array} \right|
\]}
Let  (we represent here  the case $k=6$ for more clarity)
 {\tiny      \leftskip =0,4cm \rightskip=0,4cm \begin{eqnarray*} {M=  \begin{pmatrix}
1 \=u   & 0  & 0  &   \cdots &   \cdots   &u
\cr u\=1  &1  \=u &0  &  \cdots &   \cdots &  u
\cr 0         &    u \=1   & 1\=u & \vdots &\vdots &   \vdots
\cr  \vdots         &    \vdots   &  \ddots & \ddots &\vdots &\vdots
\cr  0&\cdots&\cdots& u\=1&1\=u &u
\cr  0&0&\cdots&\cdots&0 &\l \end{pmatrix}
} \  \ {N=  \begin{pmatrix}
0   & 0  & 0  &   \cdots    &\l 
\cr 0   & 0  & 0  &   \cdots    &\l 
\cr 0   & 0  & 0  &   \cdots    &\l 
\cr \vdots         &    \vdots & \vdots & \vdots &\vdots 
\cr  \vdots         &    \vdots   &  \vdots & \ddots &\l 
\cr  0&0&\cdots&0 &\l_u \end{pmatrix}
 }   \  \  {P= \begin{pmatrix}
0   & 0  & 0  &   \cdots    &\l 
\cr0   & 0  & 0  &   \cdots    &\l 
\cr 0   & 0  & 0  &   \cdots    &\l 
\cr \vdots         &    \vdots & \vdots & \vdots &\vdots 
\cr  \vdots         &    \vdots   &  \vdots & \ddots &\l 
\cr  0&0&\cdots&0 &\l  \end{pmatrix}
  }
 \end{eqnarray*}}

Note   by what preceeds that $M$ has determinant equal to $\det(C(u))$.
Operating so we successively generate (towards a reduction to an upper triangular block matrix) a matrix determinant of the following lower triangular form
{\tiny      \leftskip =0,4cm \rightskip=0,4cm  \begin{eqnarray}\label{ulamb1}& &
{ \det(C(\l,u))=\Big(\frac{1\+(k\=1)u}{\l}\Big)^{N-1}\left|\begin{matrix}
  M  &N&{}&{}&{}\cr 
   P &M   &N &{} &{}\cr
  P&P  &M  &{}&{}
  \cr P&P  &P  &\ddots&{}  \cr
 \vdots       &\vdots  &\vdots&\ddots& {}
  \cr 
 P&P&P&{} &M \end{matrix} \right|}
  \end{eqnarray} }
The \lq descent\rq\,process ends up right before   the last right diagonal block   below, which is left  invariant. If (...\,only) after we proceed with    the linear transformation which consists of replacing   the $l$-th row with its difference with the $(l-1)$-th row, for $l=Nk, Nk-1, \ldots, 2$; this 'kills' all $\l$'s below the diagonal, we mean in terms of $k\times k$ blocks,   replacing them by   the $k\times k$ all-zeros matrix {\rm O},  namely in fact we exactly obtain   an upper block triangular   matrix. It should be noted that  this transformation let the determinant  invariant.
We obtain
{\small      \leftskip =0,4cm \rightskip=0,4cm  \begin{eqnarray*}\label{ulamb2}& &
{ \det(C(\l,u))=\Big(\frac{1\+(k\=1)u}{\l}\Big)^{N-1}\left|\begin{matrix}
  M'  &N'&{}&{}&{}\cr 
   {\rm O}&M'   &N' &{} &{}\cr
  {\rm O}&{\rm O}  &M'  &{}&{}
  \cr {\rm O}&{\rm O}  &{\rm O}  &\ddots&{}  \cr
 \vdots       &\vdots  &\vdots&\ddots& {}
  \cr 
 {\rm O}&{\rm O}&{\rm O}&{} &M' \end{matrix} \right|}
  \end{eqnarray*} }
    and 
 {\small      \leftskip =0,4cm \rightskip=0,4cm \begin{eqnarray*} {M'=  \begin{pmatrix}
1 \=u   & 0  & 0  &   \cdots &   \cdots   &u
\cr 2(u\=1)  &1  \=u &0  &  \cdots &   \cdots &  0
\cr 0         &  2(u\=1)   & 1\=u & \cdots &\cdots & 0\cr  \vdots         &    \vdots   &  \ddots & \ddots &\vdots &\vdots
\cr  0&\cdots&\cdots& 2(u\=1)&1\=u &0
\cr  0&0&\cdots&1\=u&u\=1&\l\=u \end{pmatrix}
} \qq {  N'=   \begin{pmatrix}
0   & 0  & 0  &   \cdots    &\l 
\cr 0   & 0  & 0  &   \cdots    &0 
\cr \vdots         &    \vdots & \vdots & \vdots &\vdots 
\cr  \vdots         &    \vdots   &  \vdots & \ddots &0 
\cr  0&0&\cdots&0 &\l_u\=\l \end{pmatrix}  } 
\end{eqnarray*}}
Now the change made read on $M$, $M\!\to\!M'$ does not change the value of the determinant of $M$. 
We have $\det(M')=\det(M)=\det(C(u)) $. We  use the following Lemma.
  \begin{lemma}[Franklin\,\cite{Fr},\,Th.\,7\&Problems 2-3,\,pp.\,7-8]\label{block.det} Let $A$ and $B$ be square matrices.
Then 
\begin{eqnarray*} {\left| \begin{array}{ c |  c }
A &A' 
 \cr \hline
 {\bf 0 } &B  
\end{array}
 \right|}=|A||B|.
\end{eqnarray*}\end{lemma}  
Repeated applications of this lemma yield 
\begin{equation*}  \det(C(u,\l)) =\Big(\frac{1\+(k\=1)u}{\l}\Big)^{N-1}\det(M')^{N-1}\det(C(u)).\end{equation*}
We therefore deduce that
\begin{proposition}
\begin{equation*}  \det(C(u,\l)) =\Big(\frac{1\+(k\=1)u}{\l}\Big)^{N-1} \det(C(u))^{N}.\end{equation*}
\end{proposition}



\subsubsection{\tf Determination of the inverse.}
\begin{lemma}\label{gen.lem} {\rm (i)} We have $C(u,\l)^{-1}= \{\dot d_{\nu,m}\}_{1\le \nu,m\le Nk}$ where for $j=0, \ldots , N-1$ and $l=1, \ldots, k$, 
   \begin{eqnarray*} \dot  d_{jk+ l,m}=\begin{cases} \frac{(1-u)^3+Nk(3u-\l)(1-u)-k(u-\l)(1-u-k\l)+Nk^2(u^2-\l^2)}{(1-u)(1-u+k(u-\l)(1-u+Nku+(N-1)k\l)}   & {\rm if}\ m= jk+ l, 
     \cr &\cr 
-\frac{Nk(\l-u)-k\l+u(1-u)  }{(1-u)(1-u+k(u-\l)(1-u+Nku+(N-1)k\l) }   & {\rm if}\ m= jk+ h, 1\le h\le k,h\neq l,
  \cr &\cr   
  - \frac{   \l  }{(1-u+ k(u-\l))( (1-u)+Nku +(N-1)k\l )}   & {\rm if}\ m\notin \{   jk+ h, 1\le h\le k\}.
\end{cases}
 \end{eqnarray*}
{\rm (ii)}   For any reals $\{x_m, 1\le m\le Nk\}$,

$$   \sum_{n,m=1}^{Nk}\dot  d_{n,m} x_n x_m\,=\,
- \frac{   \l  }{[1-u+ k(u-\l)][ (1-u)+Nku +(N-1)k\l ]}\Big[\sum_{   1\le l\le k\atop0\le j<N}x_{jk+ l}\Big]^2
$$ $$- \frac{ (u-\l)  }{(1-u)(1-u+ k(u-\l))}\sum_{   0\le j<N}\Big[\sum_{   1\le l\le k } x_{jk+ l}\Big]^2 +\frac{1}{1-u}\Big[\sum_{   1\le l\le k\atop0\le j<N}x_{jk+ l}^2\Big]  .
$$ 
\end{lemma}
\begin{proof} The proof being 
 not  straight,   we have to provide some details.
\vskip 2pt
(i) Let $X= {}^t\underline x$, $Y= {}^t\underline y$. 
 The system of linear equations $C(u,\l)Y=X$       reduces   by regrouping terms to
 \begin{eqnarray*} 
  \left\{\begin{matrix}(1-u)y_1+u(y_1+\ldots+y_k)+\l(y_{ k+1}+\ldots+y_{Nk})&=&x_1
  \cr \vdots&{}&\vdots
  \cr (1-u)y_k+u(y_1+\ldots+y_k)+\l(y_{k+1}+\ldots+y_{Nk})&=&x_k
  \cr \l(y_1+\ldots+y_k)+(1-u)y_{k+1}+u(y_{k+1}+\ldots+y_{2k})+\l(y_{2k+1}+\ldots+y_{Nk})&=&x_{k+1}
 \cr \vdots&{}&\vdots
  \cr   \l(y_1+\ldots+y_k)+(1-u)y_{2k}+u(y_{k+1}+\ldots+y_{2k})+\l(y_{2k+1}+\ldots+y_{Nk})&=&x_{2k}
 \cr \vdots&{}&\vdots
 \cr \vdots&{}&\vdots
 \cr \l(y_1+\ldots+y_{(N-1)k})+(1-u)y_{(N-1)k+1}+u(y_{(N-1)k+1}+\ldots+y_{Nk})&=&x_{(N-1)k+1}
 \cr \vdots&{}&\vdots 
 \cr \l(y_1+\ldots+y_{(N-1)k})+(1-u)y_{Nk}+u(y_{(N-1)k+1}+\ldots+y_{Nk})&=&x_{Nk}.\end{matrix}\right\}
   \end{eqnarray*}

By summing up   
     we obtain the relation  
 $$     \sum_{i=1 }^{Nk}y_i =  \frac{\sum_{i=1 }^{Nk}x_i}{ (1-u)+Nku +(N-1)k\l } .$$ 
   Set $s=\sum_{i=1 }^{Nk}y_i$, $t=\sum_{i=1 }^{Nk}x_i$. 
By summing the $k$ first  equations  only we get,  $ (1-u+ ku)(y_1+\ldots+y_k)+k\l(y_{ k+1}+\ldots+y_{Nk})= (1-u+ k(u-\l)(y_1+\ldots+y_k)+k\l s=x_1+\ldots+x_k $, so that, as $ s =  \frac{t}{ (1-u)+Nku +(N-1)k\l }$,
 \begin{eqnarray*}
& & y_1+\ldots+y_k    \,=\,\frac{x_1+\ldots+x_k  -  k\l s  }{1-u+ k(u-\l)} 
 \cr & =&\frac{x_1+\ldots+x_k   }{1-u+ k(u-\l)}- \frac{   k\l t  }{(1-u+ k(u-\l))((1-u)+Nku +(N-1)k\l )}   .
 \end{eqnarray*}
 In fact  for $ j=0,\ldots,N-1$, we have 
 $ (1-u+ ku)(y_{jk+1 }+\ldots+y_{j(k+ 1) })+k\l\big((s-(y_{jk+1 }+\ldots+y_{j(k+ 1) })\big)=(1-u+ k(u-\l))(y_{jk+1 }+\ldots+y_{j(k+ 1) })+k\l s=x_{jk+1 }+\ldots+x_{j(k+ 1) } $, and so
 \begin{eqnarray*}
& &  y_{jk+1 }+\ldots+y_{j(k+ 1) }\,=\, \frac{x_{jk+1 }+\ldots+x_{j(k+ 1) }   -  k\l s  }{1-u+ k(u-\l)} 
 \cr & =&\frac{x_{jk+1 }+\ldots+x_{j(k+ 1) }   }{1-u+ k(u-\l)}
  - \frac{   k\l t  }{(1-u+ k(u-\l))((1-u)+Nku +(N-1)k\l )}     . \end{eqnarray*}

 Now for $l=1,\ldots, k$,
 $ (1-u)y_{jk+l }=x_{jk+l }-\big((u-\l)(y_{jk+1 }+\ldots+y_{j(k+ 1) })+\l s \big)   $, and  by reporting 
   \begin{eqnarray*}  (1-u)y_{jk+l }&=&x_{jk+l }-\big((u-\l)(y_{jk+1 }+\ldots+y_{j(k+ 1) })+\l s \big)\cr      &=& x_{jk+l }-\bigg((u-\l)\Big( \frac{x_{jk+1 }+\ldots+x_{j(k+ 1) }   -  k\l s  }{1-u+ k(u-\l)} 
\Big)+\l s \bigg)
\cr      &=& x_{jk+l }-\bigg((u-\l)\Big( \frac{x_{jk+1 }+\ldots+x_{j(k+ 1) }  }{1-u+ k(u-\l)} 
\Big)  -   \frac{    (u-\l)k\l s  }{1-u+ k(u-\l)} 
 +\l s  
 \bigg)
 \cr      &=& x_{jk+l }-\bigg((u-\l)\Big( \frac{x_{jk+1 }+\ldots+x_{j(k+ 1) }  }{1-u+ k(u-\l)} 
\Big)  -   \l s\frac{    (u-\l)k - (1-u+ k(u-\l))  }{1-u+ k(u-\l)} 
  \bigg)
   \cr      &=& x_{jk+l }-\bigg((u-\l)\Big( \frac{x_{jk+1 }+\ldots+x_{j(k+ 1) }  }{1-u+ k(u-\l)} 
\Big)  +   \l s\frac{   1-u  }{1-u+ k(u-\l)} 
  \bigg)
    \cr      &=& x_{jk+l }- (u-\l)  \frac{x_{jk+1 }+\ldots+x_{j(k+ 1) }  }{1-u+ k(u-\l)} 
 \cr & &  -    \Big(\frac{   \l ( 1-u  )}{1-u+ k(u-\l)} \Big)
   \frac{t}{ (1-u)+Nku +(N-1)k\l }.  
       \end{eqnarray*}
 
Hence,
  \begin{eqnarray*}  y_{jk+l }   &=& \frac{x_{jk+l }}{1-u}-    \frac{ (u-\l)  }{(1-u)(1-u+ k(u-\l))}( x_{jk+1 }+\ldots+x_{j(k+ 1) }) 
 \cr & &  -    \Big(\frac{   \l  }{[1-u+ k(u-\l)][ (1-u)+Nku +(N-1)k\l ]} \Big)
 \sum_{h= 1}^{Nk}x_h
 \cr &=& a x_{jk+l } + b \Big(\sum_{1\le l'\le k\atop l'\neq l}x_{jk+ l'} \Big)+ c \Big(\sum_{{1\le l'\le k\atop 0\le j'<N}\atop j'\neq  j }x_{j'k+l'}  \Big),  
       \end{eqnarray*}
where $a,b,c$ rely on $u,\l, N,k$ only and are defined     as follows:   
\begin{eqnarray*} 
 a&=&\frac{1}{1-u}- \frac{ (u-\l)  }{(1-u)(1-u+ k(u-\l))}- \frac{   \l  }{[1-u+ k(u-\l)][ (1-u)+Nku +(N-1)k\l ]}  
\cr
 b&=&  - \frac{ (u-\l)  }{(1-u)(1-u+ k(u-\l))}- \frac{   \l  }{[1-u+ k(u-\l)][ (1-u)+Nku +(N-1)k\l ]}   \cr
c&=& - \frac{   \l  }{[1-u+ k(u-\l)][ (1-u)+Nku +(N-1)k\l ]}.
\cr
&&
  \end{eqnarray*}  By reducing them, we   therefore get
     $ y_{jk+ l}=\sum_{m=1 }^{N k} \dot d_{jk+ l,m} x_{  m}$ with 
     \begin{eqnarray*} \dot  d_{jk+ l,m}=\begin{cases} \frac{(1-u)^3+Nk(3u-\l)(1-u)-k(u-\l)(1-u-k\l)+Nk^2(u^2-\l^2)}{(1-u)(1-u+k(u-\l)(1-u+Nku+(N-1)k\l)}   & {\rm if}\ m= jk+ l, 
     \cr &\cr 
-\frac{Nk(\l-u)-k\l+u(1-u)  }{(1-u)(1-u+k(u-\l)(1-u+Nku+(N-1)k\l) }   & {\rm if}\ m= jk+ h, 1\le h\le k,h\neq l,
  \cr &\cr   
  - \frac{   \l  }{(1-u+ k(u-\l))( (1-u)+Nku +(N-1)k\l )}   & {\rm if}\ m\notin \{   jk+ h, 1\le h\le k\}.
\end{cases}
 \end{eqnarray*}
Let $\dot D$ denote the $NK\times NK$ matrix  with coefficients above. We have proved that
 $C(u,\l)^{-1}\,=\,\dot D$. 
     \vskip 3 pt

 \vskip 3 pt (ii)
Let $j,l$ be fixed.  
We have 
 $$ \sum_{m=1}^{Nk}\dot  d_{jk+ l,m} x_m= c\Big(\sum_{{1\le l'\le k\atop 0\le j'<N}\atop j'\neq  j }x_{j'k+l'}  \Big)+b\Big(\sum_{1\le l'\le k\atop l'\neq l}x_{jk+ l'} \Big)+ax_{jk+ l}.$$ 
 Let $Z_j=\sum_{   1\le l\le k } x_{jk+ l}$, $j=0, \ldots, N-1$. Then,
$$   \sum_{n,m=1}^{Nk}\dot  d_{n,m} x_n x_m\,=\,
\sum_{   0\le j<N}Z_j\bigg[ c\sum_{ 0\le j'<N\atop  j'\neq  j }Z_j'+bZ_j+(a-b)x_{jk+ l}\bigg]
   $$ 
   $$  \,=\,c\Big(\sum_{   0\le j<N}Z_j\Big)^2+  (b-c)\sum_{   0\le j<N}Z_j^2+   (a-b)\sum_{   0\le j<N}\sum_{   1\le l\le k }\, x^2_{jk+ l}$$
   $$
   =\, c\Big[\sum_{   1\le l\le k\atop0\le j<N}x_{jk+ l}\Big]^2
+(b-c)\sum_{   0\le j<N}\Big[\sum_{   1\le l\le k } x_{jk+ l}\Big]^2
  +(a-b)\Big[\sum_{   1\le l\le k\atop0\le j<N}x_{jk+ l}^2\Big]  .
$$ 
As  
\begin{align*}a-b&= \frac{1}{1-u}
\qq  b-c = - \frac{ (u-\l)  }{(1-u)(1-u+ k(u-\l))}
\cr c&=- \frac{   \l  }{[1-u+ k(u-\l)][ (1-u)+Nku +(N-1)k\l ]}, 
\end{align*}
we get
$$    \sum_{n,m=1}^{Nk}\dot  d_{n,m} x_n x_m\,=\,
- \frac{   \l  }{[1-u+ k(u-\l)][ (1-u)+Nku +(N-1)k\l ]}\Big[\sum_{   1\le l\le k\atop0\le j<N}x_{jk+ l}\Big]^2
$$ $$- \frac{ (u-\l)  }{(1-u)(1-u+ k(u-\l))}\sum_{   0\le j<N}\Big[\sum_{   1\le l\le k } x_{jk+ l}\Big]^2 +\frac{1}{1-u}\Big[\sum_{   1\le l\le k\atop0\le j<N}x_{jk+ l}^2\Big] . 
$$ 
 \end{proof}


 \section{\bf  Upper tail  probability estimates derivation (II)} \label{s6}
 The following corollary is now immediate.
\begin{corollary}\label{FulkN.cor}For   any  real   $\Theta$, 
\begin{eqnarray*}\mathcal F_{u \l,k N}(\Theta)&=  &\int_{\sup_{i=1}^{Nk} x_i\le \Theta}  \frac{e^{-  \frac{ 1}{2 }  \, Q(u,\l)(\underline x)   }    }{(2\pi)^{Nk/2}[ (  1- u)^{(k-1)N}\big(  1+u( k-1 )\big))^N]^{1/2}}\,\dd \underline x , 
 \end{eqnarray*}
where $\mathcal F_{u \l,k N}$ is defined in \eqref{FulkN} and for each $\underline x\in \R^{Nk}$,
\begin{eqnarray*} Q(u,\l)(\underline x)
&  =&- \frac{   \l  }{(1-u+ k(u-\l))\,( (1-u)+Nku +(N-1)k\l )}\Big(\sum_{ v=1}^{Nk} x_{v}\Big)^2
\cr & &- \frac{ (u-\l)  }{(1-u)(1-u+ k(u-\l))}\sum_{j=0}^{N-1}\Big(\sum_{ l=1 }^k x_{jk+ l}\Big)^2 +\frac{1}{1-u}\Big(\sum_{ v=1}^{Nk} x^2_{v}\Big).
\end{eqnarray*}
\end{corollary}

In   the next sub-section, we have isolated in the course of the proofs,   other upper tail probability estimates of independent interest.

 \subsection{Proofs of Theorems \ref{gen.case.th} and \ref{F.decoupling.th}}   
{\bf (A)} Claim  { (i)} of Theorem  \ref{gen.case.th}    is a direct consequence of Slepian Lemma and Corollary \ref{FulkN.cor}.  
  \vskip 3 pt
 {\bf (B)} We now prove claim    {(ii)} of Theorem  \ref{gen.case.th}. Let $$2A=\frac{   \l  }{[1-u+ k(u-\l)][ (1-u)+Nku +(N-1)k\l ]}, \  2B= \frac{ (u-\l)  }{(1-u)(1-u+ k(u-\l))},$$
so that we can  rewrite
$$ -\frac12 Q(\l,\m)(\underline x)\,=\,
A\sum_{   0\le j<N}\Big[\sum_{  l=1 }^k x_{jk+ l}\Big]^2 +B\sum_{   0\le j<N}\Big[\sum_{  l=1 }^k x_{jk+ l}\Big]^2 -\frac{1}{2(1-u)}\Big[\sum_{   1\le l\le k\atop0\le j<N}x_{jk+ l}^2\Big]  .
$$  We first bound   as follows:
   \begin{eqnarray*}
  \P\Big\{ \sup_{i=1}^{Nk} X_i\le \Theta \Big\} 
  &\le& \P\Big\{ \sup_{i=1}^n X_i\le \Theta, A \big(\sum_{ v=1}^{Nk} X_{v}\big)^2  \le M\Big\}
  +  \P\Big\{  A \big(\sum_{ v=1}^{Nk} X_{v}\big)^2  > M \Big\}
  .
 \end{eqnarray*}
Then, \begin{align*}
    \P\Big\{ \sup_{i=1}^n X_i\le \Theta,  A \big(\sum_{ v=1}^{Nk} X_{v}\big)^2  &\le M\Big\}
 \cr &\le  \ e^{M} \int_{\sup_{i=1}^{Nk} x_i\le \Theta}\frac{e^{-  \frac{ 1}{2(1-u) }    \sum_{ v=1}^{Nk} x^2_{v} +B\sum_{   0\le j<N} [\sum_{  v=1 }^k 
 x_{v} ]^2   }    }{  [(2\pi) (  1- u)^{ k-1 }\big(  1+u( k-1 )\big)  ]^{N /2}}\,\dd \underline x.
\end{align*}

 
We shall arrange the exponent differently. We have 
$$-  \frac{ 1}{2(1-u) }    \sum_{ v=1}^{Nk} x^2_{v} +B\sum_{   0\le j<N}\Big[\sum_{  l=1 }^k 
 x_{l} \Big]^2 \qq\qq\qq\qq\qq\qq\qq\qq\qq $$$$\qq\qq\qq =   \frac{ 1}{2(1-u) }\sum_{   0\le j<N}\Big( -     \sum_{  l=1 }^k  x^2_{v} +\frac{ (u-\l)  }{ (1-u+ k(u-\l))}  \Big[\sum_{  l=1 }^k 
 x_{l} \Big]^2\Big) $$
and  
  $$\frac{u  }{  1-u +ku  }-\frac{ (u-\l)  }{ (1-u+ k(u-\l))}=\frac{u-u^2+ku^2-uk \l -u+\l+u -u\l-ku^2+ku\l}{( 1-u +ku)(1-u+ k(u-\l))} $$
  $$ =\frac{  \l+u -(\l+u)u  }{( 1-u +ku)(1-u+ k(u-\l))}  =\frac{   (\l+u)(1-u)  }{( 1-u +ku)(1-u+ k(u-\l))}, $$
so that  $$\frac{ (u-\l)  }{ (1-u+ k(u-\l))}= \frac{u  }{  1-u +ku  }-\frac{   (\l+u)(1-u)  }{( 1-u +ku)(1-u+ k(u-\l))}\le  \frac{u  }{  1-u +ku  }. $$   By reporting we continue   our estimation as follows
$$ =   \frac{ 1}{2(1-u) }\sum_{   0\le j<N}\Big( -     \sum_{  l=1 }^k  x^2_{v} +\Big\{\frac{u  }{  1-u +ku  }-\frac{   (\l+u)(1-u)  }{( 1-u +ku)(1-u+ k(u-\l))}\Big\}  \Big[\sum_{  l=1 }^k 
 x_{l} \Big]^2\Big)  $$
$$  =   \frac{ 1}{2(1-u) }\sum_{   0\le j<N}\Big( -     \sum_{  l=1 }^k  x^2_{v} + \frac{u  }{  1-u +ku  }  \Big[\sum_{  l=1 }^k 
 x_{l} \Big]^2\Big)
  \ \qq\qq\qq\qq\qq\qq\qq\qq$$$$\qq\qq\qq- \frac{ 1}{2(1-u) }\Big(\frac{   (\l+u)(1-u)  }{( 1-u +ku)(1-u+ k(u-\l))}\Big)\sum_{   0\le j<N}   \Big[\sum_{  l=1 }^k 
 x_{l} \Big]^2. $$
Thus
$$  \P\Big\{ \sup_{i=1}^n X_i\le \Theta, A \big(\sum_{ v=1}^{Nk} X_{v}\big)^2  \le M\Big\}
     \le   e^{M } \bigg( \int_{ \sup_{i=1}^{ k} x_i \le \Theta}\frac{e^{  \frac{ 1}{2(1-u) } ( -     \sum_{  l=1 }^k  x^2_{l} + \frac{u  }{  1-u +ku  }   [\sum_{  l=1 }^k 
 x_{l}  ]^2 )  }    }{  [(2\pi) (  1- u)^{ k-1 }\big(  1+u( k-1 )\big)  ]^{ 1/2}}\,\dd \underline x \bigg)^N.  $$
Further 
$$ \P\Big\{    \Big|\sum_{ v=1}^{Nk} X_{v}\Big| >\sqrt{ M/A}\Big\}= \Psi \Big(\frac{\sqrt{ M }}{\sqrt{  A}\|\sum_{ v=1}^{Nk} X_{v}\|_2}\Big).$$
  We therefore obtain by combining and using Lemma \ref{boyd.mills}-(1),  
\begin{corollary}\label{decoupling1.cor}  We have for any real $\Theta$, any positive real $M$,     
    \ben \label{decoupling1}
      \P\Big\{ \sup_{t\in T}^n \dot X(t)\le \Theta \Big\}
      &\le&    e^{M } \bigg( \int_{ \sup_{i=1}^{ k} x_i \le \Theta}\frac{e^{  \frac{ 1}{2(1-u) } ( -     \sum_{  l=1 }^k  x^2_{l} + \frac{u  }{  1-u +ku  }   [\sum_{  l=1 }^k 
 x_{l}  ]^2 )  }    }{  [(2\pi) (  1- u)^{ k-1 }\big(  1+u( k-1 )\big)  ]^{ 1/2}}\,\dd \underline x \bigg)^N
 \cr & & \qq\qq\qq  + C\frac {\sqrt{  A}\|\sum_{ v=1}^{Nk} X_{v}\|_2}{\sqrt{ M }
 }e^{-\frac{ { M }}{ 2{  A}\|\sum_{ v=1}^{Nk} X_{v}\|_2^2} },  
\een  
recalling that $A= \frac{   \l  }{2[1-u+ k(u-\l)][ (1-u)+Nku +(N-1)k\l ]}$.
\end{corollary} 
The Gaussian density appearing between  parenthesis is the one of a Gaussian vector in $\R^k$ with covariance matrix $C(u)$, see Corollary \ref{Ac1.enonce}; and so estimate \eqref{decoupling1}
 exhibits a decoupling phenomenon of a  type specific to this class of covariance matrices, which states as follows: 
for  any real $\Theta$, any positive real $M$,
 \begin{eqnarray*} 
     \mathcal F_{u\l, kN}(\Theta)
   &\le  & e^{M }\mathcal F_{ u, k }(\Theta)^N 
     +   C\frac {\sqrt{  A}\|\sum_{ v=1}^{Nk} X_{v}\|_2}{\sqrt{ M }
 }\, e^{-\frac{ { M }}{ 2{  A}\|\sum_{ v=1}^{Nk} X_{v}\|_2^2} }.\end{eqnarray*}
This establishes Theorem \ref{F.decoupling.th}.
\begin{remark}\label{dec.norm}\rm The appearance of the norm $\|\sum_{ v=1}^{Nk} X_{v}\|_2$   seems  unvoidable. It would be very interesting to understand  more its importance in \eqref{F.decoupling}.
\end{remark} \vskip 5 pt
 \vskip 3 pt
 {\bf (C)}  As a consequence we further establish the following upper tail probability estimate. Suppose  that in assumption \eqref{assumption2} the first requirement is strenghtened in 
\ben  \label{assumption2as}  |\E X(s)X(t)|  \le u  \qquad \forall {s\neq t\in T_i}, \ \forall i\in J .
 \een 
      We deduce from Corollary \ref{ass1th} the following estimate of the supremum.
 \begin{corollary}    Let $X =\{X(t), t\in T\}$ be a centered Gaussian process  with unit variance satisfying   assumption \eqref{assumption2} with strenghtening \eqref{assumption2as}. Let  $0<\eta<1$ be fixed.   There exist positive constants $C_\eta, C'_\eta$ depending on $\eta$ only, such that if  $ u\le \eta/2k  $,  we have for any positive real $M$, 
      \ben\label{assumption2a.enonce}
     &&\P\Big\{ \sup_{t\in T}  X(t)\le \sqrt{    2  \log k -2 \log\log k   -   \eta  (\log k)/      k }  \Big\}
\cr  &  \le  & e^{M }C^N_\eta\, e^{- C'_\eta N\sqrt{ \log k  }  }+  C\frac {\sqrt{  A}\| \sum_{t\in T}  X(t)\|_2}{\sqrt{ M }
 }\, e^{-\frac{ { M }}{ 2{  A}\| \sum_{t\in T}  X(t)\|_2^2} }.  
\een
  \end{corollary}

{\bf (D)} Finally we prove claim   {(iii)} of Theorem  \ref{gen.case.th}.   By Lemma \ref{gen.lem}-(ii), applying twice Cauchy-Schwarz's inequality,
 \begin{eqnarray*}
-
Q(\l,\m)(\underline x)&=& \frac{   \l  }{[1-u+ k(u-\l)][ (1-u)+Nku +(N-1)k\l ]}\Big[\sum_{   1\le l\le k\atop0\le j<N}x_{jk+ l}\Big]^2
\cr & &+ \frac{ (u-\l)  }{(1-u)(1-u+ k(u-\l))}\sum_{   0\le j<N}\Big[\sum_{   1\le l\le k } x_{jk+ l}\Big]^2 -\frac{1}{1-u}\Big[\sum_{   1\le l\le k\atop0\le j<N}x_{jk+ l}^2\Big]
\cr 
&\le & \frac{   \l Nk }{[1-u+ k(u-\l)][ (1-u)+Nku +(N-1)k\l ]}\Big[\sum_{   1\le l\le k\atop0\le j<N}x^2_{jk+ l}\Big] 
\cr & &+ \Big(\frac{ (u-\l) k }{(1-u)(1-u+ k(u-\l))}  -\frac{1}{1-u}\Big)\Big[\sum_{   1\le l\le k\atop0\le j<N}x_{jk+ l}^2\Big]
\cr 
&=& \frac{   \l Nk }{[1-u+ k(u-\l)][ (1-u)+Nku +(N-1)k\l ]}\Big[\sum_{   1\le l\le k\atop0\le j<N}x^2_{jk+ l}\Big] 
\cr & &- \Big( \frac{1}{1-u}\Big)\Big(1-\frac{ (u-\l) k }{ 1-u+ k(u-\l) }  \Big)\Big[\sum_{   1\le l\le k\atop0\le j<N}x_{jk+ l}^2\Big]
\cr 
&=& \frac{   \l Nk }{[1-u+ k(u-\l)][ (1-u)+Nku +(N-1)k\l ]}\Big[\sum_{   1\le l\le k\atop0\le j<N}x^2_{jk+ l}\Big] 
\cr & &- \Big(  \frac{1}{ 1-u+ k(u-\l) }  \Big)\Big[\sum_{   1\le l\le k\atop0\le j<N}x_{jk+ l}^2\Big]
\cr 
&=& - \Big(  \frac{1}{ 1-u+ k(u-\l) }  \Big)\Big[\sum_{   1\le l\le k\atop0\le j<N}x_{jk+ l}^2\Big]
  \Big\{1- \frac{   \l Nk }{ [ (1-u)+Nku +(N-1)k\l ]}\Big\}
\cr 
&=& - \Big(  \frac{1}{ 1-u+ k(u-\l) }  \Big) \Big(  \frac{  1-u+Nku-k\l }{    1-u +Nk(u +\l)-k\l  }  \Big)\Big[\sum_{   1\le l\le k\atop0\le j<N}x_{jk+ l}^2\Big]
 \end{eqnarray*}
Now    $1-u+Nku-k\l>0$ if $Nu>\l$,  and so $0<\frac{  1-u+Nku-k\l }{    1-u +Nk(u +\l)-k\l  }<1$. Moreover $0<\frac{1}{ 1-u+ k(u-\l) }<1$ if $ u>\frac{k}{k-1}\l$ or $k(u-\l)>1$.
Let $$\b(u,\l):=\Big(  \frac{1}{ 1-u+ k(u-\l) }  \Big) \Big(  \frac{  1-u+Nku-k\l }{    1-u +Nk(u +\l)-k\l  }  \Big) $$
Hence if $Nu>\l$ and $(k-1)u>1$, then $0<\b(u,\l)<1$, $\b(u,\l)\asymp\frac{1}{ 1-u+ k(u-\l) }  \frac{   ( N u- \l )}{     (N  u   - \l ) +N\l}$,  and
$$-\frac12\, {}^t \underline{x}\, C(u,\l)^{-1}\underline{x}   \,\le \,- \Big(\frac{ \b(u,\l)}{2 }\Big) \Big( \sum_{   1\le v\le Nk}x_{v}^2 \Big).$$
 
 It follows that 
  \begin{eqnarray*}\P\Big\{ \sup_{i=1}^{Nk} X_i\le \Theta\Big\}&= &\int_{\sup_{i=1}^{Nk} x_i\le \Theta}  \frac{e^{-  \frac{ 1}{2 }  Q(\l,\m)(\underline x)   }    }{ (2\pi)^{Nk/2}[ (  1- u)^{(k-1) }\big(  1+u( k-1 )\big)) ]^{N/2} }\,\dd \underline x
 \cr &\le &\int_{\sup_{i=1}^{Nk} x_i\le \Theta}  \frac{e^{-  \frac{ \b(u,\l)}{2 }    \sum_{   1\le v\le Nk}x_{v}^2  }    }{ (2\pi)^{Nk/2}[ (  1- u)^{(k-1)N}\big(  1+u( k-1 )\big))^N]^{1/2} }\,\dd x
 \cr  &= & \frac{  1  }{ [ (  1- u)^{(k-1) }\big(  1+u( k-1 )\big)) ]^{N/2} }
 \Big(\int_{x\le \Theta}  \frac{e^{-  \frac{ \b(u,\l)}{2 }   x^2  }    }{ (2\pi)^{1/2} }\,\dd x \Big)^{Nk}
 \cr   &= &\Big(1+\frac{  \l n  }{   1-\l }\Big)^{(n-1)/2}\,
  \Phi  \Big(\frac{\Theta}{\sqrt{ 1+\l(n-1) }}\Big)^n  .
 \end{eqnarray*}
    
Hence we get
 \begin{corollary}\label{gen.cor} Recall that we set $Q(\l,\m)(\underline x)={}^t \underline{x}\, C(u,\l)^{-1}\underline{x} $. For any $\underline x\in \R^{Nk}$,
$$-\frac12\, Q(\l,\m)(\underline x)   \,\le \,- \Big(\frac{ \b(u,\l)}{2 }\Big) \Big( \sum_{   1\le v\le Nk}x_{v}^2 \Big),$$
where  $$\b(u,\l) =\Big(  \frac{1}{ 1-u+ k(u-\l) }  \Big) \Big(  \frac{  1-u+Nku-k\l }{    1-u +Nk(u +\l)-k\l  }  \Big) $$
 Assume that $Nu>\l$ and $(k-1)u>1$. Then $0<\b(u,\l)<1$, $\b(u,\l)\asymp\frac{1}{ 1-u+ k(u-\l) }  \frac{   ( N u- \l )}{     (N  u   - \l ) +N\l}$.   \end{corollary}

This achieves the proof of Theorem \ref{gen.case.th}.


  \section{\bf  A     number theoretic example}\label{s7}
 In this Section we prove Theorem \ref{demi.logx1}. The main ingredients are Theorem \ref{th.1}, Corollary \ref{ass1th}  and correlation estimates, Lemma \ref{bound.cor.} (sub-section \ref{s5.1}).    Theorem \ref{gen.case.th} and decoupling Theorem \ref{F.decoupling.th} and not applied here, see Remark \ref{F.decoupling.gen.case.th}.
 
 \subsection{Background on Wintner's random M\"obius function.}    \label{sa}  Concerning standard convergence results on Dirichlet series, multiplicative functions and Eulerian products used here, we     refer  to the well-written and accessible short monograph  of Blanchard \cite{Bl}, other references are of course available,   for instance Titchmarsh \cite{Ti}, Montgomery and  Vaughan \cite{MV}. We otherwise based this part on related Notes in  Weber \cite{W11a}.
   \vskip 2 pt 
   Let  $ \{\e_p, \e'_p,p\in\mathcal P\}$  be 
 a sequence of independent Rademacher random variables ($\P\{ \e_p =\pm 1\} =1/2$) with basic
probability space $(\O, \A, \P)$. Consider the random multiplicative function 
\begin{equation*}f(n)=\begin{cases}\prod_{p|n}\e_p,\quad &\hbox{if $n$ is squarefree,}\cr 
0,\quad &\hbox{otherwise.}\end{cases}\end{equation*}
(Multiplicativity means that $f(mn)=f(m)f(n)$ for all $m,n\in \N$ with $\gcd(m, n)=1$.)
  There is no loss to exclude from the sequence   those terms which are null. 
Let $n$ be squarefree. Then $\E f^i(n)  =0$, if $i$ is odd, and $\E f^i(n)  =1$, if $i$ is even. 
 In particular,  the  sequence $\{f(n) ,  n \,  \hbox{squarefree}\}$ is an orthonormal sequence. 
Wintner's random M\"obius function $M(x)$ is defined for  $x\ge 2$ as follows,
$$M(x) =\sum_{n\le x} f(n)= \sum_{n\le x\atop n\, {\rm squarefree} }f(n).$$

\subsubsection{\tffs An orthogonal point of view.} We shall first give estimates of  $M(x)$ by using results from the theory of orthogonal sums. By Rademacher-Menshov's
theorem,  the series 
     $S =\sum_{k } a_kf(k)  $
converges almost surely if the following condition is fulfilled
 $ \sum_k a_k^2(\log k)^2 <\infty$. 
In particular for $\e>0$,
     $S =\displaystyle{\sum_{n } \frac{f(n) }{\sqrt n\,(\log n)^{3/2+\e} }}=\displaystyle{\sum_{n\, {\rm squarefree} }} \frac{f(n) }{\sqrt n\,(\log n)^{3/2+\e}}
     $
   converges almost surely.  
From Kronecker's lemma  we get,
\begin{equation}\label{bound.kronecker} M(x)=o_\e\big(\sqrt x\,(\log x)^{3/2+\e}\big).
\end{equation}  Wintner \cite{Wi} proved that for each $\e>0$, 
\begin{equation}\label{bound.M} M(x)=\mathcal O(x^{1/2+\e}),
\end{equation} which is   weaker.  By a well-known result of   Olevski\v i \cite{O}, if $\{\p_k,k\ge 1\}$ is an orthonormal system of collectively bounded functions, $\|\p_k\|\le M$, $k=1,2,\ldots$, we have
\begin{equation}\label{Olevskii}\max_{1\le n\le N}\Big\|\sum_{k=1}^n \p_k \Big\|_{L_1}\, \ge \, C(M) \log N, \quad N=1,2,\ldots, \quad C(M)>0.
\end{equation}

It follows that 
\begin{equation}\label{Olevskii.Mobius}\max_{1\le n\le N}\big\|  M_n \big\|_{L_1}\, \ge \, C(1) \log N, \quad N=1,2,\ldots, \quad C(1)>0.
\end{equation}

     \begin{remarks}[Other refinements] 
${}$\hfill 
  \vskip 6 pt\begin{itemize}
 \leftskip1mm
 \itemsep=2pt
 \item[{\rm (1)}\ ]Let $r>s>0$ and $0\le \e \le 1$. Then for any non-negative random variable $X\in L^r$, (\cite[inequality (8.3.7)]{We}).
 $
\P\big\{ X\ge \e\|X\|_s\big\}^{ {1\over s }-{1\over  r}}  \ge  (1-\e^s)^{{1\over s}}{\|X\|_s \over \|X\|_r}.
 $
Applying it with $X= |M(x)| $, $r=4$, $s=2$ gives
$$\P\big\{ |M(x)|\ge \e\| M(x) \|_2\big\}^{ {1\over  4 } }  \ge  (1-\e^2)^{{1\over 2}}{\|M(x)\|_2 \over \|M(x)\|_4}. $$
We have $\E M(x)^2= x$, $\E M(x)^4 \sim c x^2(\log x)^2$. Thus  
$$\P\big\{ |M(x)|\ge \e\sqrt x\big\}   \ge {c_1 (1-\e^2)^{{1\over 2}} \over (\log x)^{{1\over 2}}}. $$
 
\item[{\rm (2)}\ ]By a result of Borgen and Kaczmarz, if  
 $ \sum_k a_k^2(\log \log k)^2 <\infty$,  the sequence of partial sums
 $S_j =\sum_{k\le 2^j } a_k f(k)$,  $j=1,2,\ldots  $
converges almost surely. Further the sums $S_j, j\ge 1$ are $C_1$-summable almost surely.
  In particular the sequence
$$ \frac{f(n) }{\sqrt{ n \log n }\,(\log\log  n)^{3/2+\e}},\qq \quad n=1,2,\ldots$$
is $C_1$-summable almost surely. 
 \end{itemize}\end{remarks}
\vskip 1 pt \noi We are not sure whether   results in (2) are reachable using analytic number theory.
\bigskip \par We    include a brief and comprehensive  account on the link between  the Dirichlet series
\begin{equation*}F(s)= \sum_{n\ge 1}\frac{f(n)}{n^s},
\end{equation*}
and summatory function $M(x)$. The easiest  way 
 to understand how   $M(x)$ can be bounded from below (towards proving  an $\Omega$-result), follows from the simple application  of Abel summation to the  random Dirichlet sum $\sum_{1\le n\le x}\frac{f(n)}{n^s}$. 
  We indeed have that
\begin{equation}\label{abel}\sum_{1\le n\le x}\frac{f(n)}{n^s}= \frac{M(x)}{x^s}+ s\int_1^x  \frac{M(v)}{v^{s+1}}\dd v, \qq \qq (s=\s+it).
\end{equation}
  That Wintner's result implies that  the Dirichlet series
$F(s)= \sum_{n\ge 1}\frac{f(n)}{n^s}$
  is almost surely convergent in the half plane $\Re(s)>1/2$, follows from \eqref{abel} and  \eqref{bound.M} since
\begin{equation*}\Big| \sum_{y< n\le x}\frac{f(n)}{n^s}-s \int_y^x  \frac{M(v)}{v^{s+1}}\dd v\Big|\le  \Big|\frac{M(y)}{y^s}\Big|+\Big|\frac{M(x)}{x^s}\Big|= \mathcal O\big(|y|^{-\e}+ |x|^{-\e}\big)\ \to \ 0,\end{equation*}
as $y,x$ tend to infinity. By \cite[Prop.\,II.6]{Bl}, that  series is a holomorphic function in the open half plan $\Re(s)>1/2$, and
\begin{equation}\label{WM.F} F(s)=
 \sum_{n\ge 1}\frac{f(n)}{n^s}
 =s\int_0^\infty  \frac{M(v)}{v^{s+1}}\dd v \ = \ s\int_0^\infty  \frac{M(v)}{v^{2s}}v^{s-1}\dd v.
 \end{equation}

  \begin{remark}[{\pH Mellin's inversion formula}] \rm
We refer to Ivi\'c \cite[Appendix\,A2]{I}. By Mellin's inversion formula  (extending $M(v)$ on $\R_+$ by putting  $M(v)=\e_2$ if $0\le v<2$), one can get a direct expression of $M$ in terms of $F$: \begin{equation}\label{Mif}
M(v) =\frac1{2\pi}\ \lim_{T\to \infty}
\int_{-T}^{T} F(\s+it)v^{\s+it} \dd t, \end{equation} 
for all but countably many $ v>0$. 
\vskip 2 pt
Let $\s>1/2$. By integrating term by term with respect to $\int_{-T}^{T}\dd t$,  next letting $T$ tend to $+\infty$, this produces the formula with     remarkable  known shape:
\begin{equation}\label{Mifa}
\frac{M(v)}{v^\s} =\frac1{2\pi}\ \lim_{T\to \infty}
 \sum_{n\ge 1} \frac{f(n)}{n^\s} \frac{  \sin(T\log (v/n))}{\log (v/n)}, \end{equation} 
for all but countably many $ v>0$. Now 
$$F(s)= \sum_{n\ge 1}\frac{f(n)}{n^s}=\exp\Big\{\sum_{ p}\frac{\e_p}{p^{s}}+R_s\Big\}
$$
$$\int_{-T}^{T} F(\s+it)v^{\s+it} \dd t=\int_{-T}^{T} e^{\sum_{ p}\frac{f(p)\cos(t\log p)}{p^{\s}}+R_s}v^{\s+it} \dd t$$
$$\Big|\frac{M(v)}{v^\s}\Big|^2\le \int_{-T}^{T} \Big|F(\s+it) \Big|^2 \dd t=\int_{-T}^{T} e^{\sum_{ p}\frac{f(p)\cos(t\log p)}{p^{\s}}+R_s} \dd t$$Further analog of Parseval identity for Fourier transforms       applies   to give
\begin{equation}\label{FM}\frac1{2\pi} \int_\R \frac{|F(\s+ it)|^2}{\s^2 + t^2}\dd t=\int_0^\infty M^2(\frac{1}{u})u^{2\s-1}\dd u=\int_0^\infty \frac{M^2(w)}{w^{1+2\s}}\dd w.\end{equation}
\vskip 2 pt

Whence also 
\begin{eqnarray*}\int_0^\infty \frac{M^2(w)}{w^{1+2\s}}\dd w
\ =\ \sum_{n}\frac{f^2(n)}{n^{2\s}}+2\sum_{m>n}\frac{f(n)f(m)}{m^{2\s}}  .\end{eqnarray*}

\end{remark}


  \subsubsection{\tffs Halasz's approach.}   Recalling that  
  $f(p^k)=0$ if $k>1$, we have \cite[Prop.\,II.7]{Bl} (see also   Halasz \cite{Ha} p.\,87) if $\Re(s)>1$,
   \begin{equation}\label{Ep} F(s)=  \prod_{  p}\Big(1+ \frac{\e_p}{p^{s}}\Big)
 \end{equation}
     By Kolmogorov's three series theorem, the series $\sum_{p}  \frac{\e_p}{p^{s}} $ converges almost surely if $\Re(s)>1/2$. By \cite[Prop.\,II.6]{Bl} again,  this series is a holomorphic function in the open half plan $\Re(s)>1/2$. Equation \eqref{Ep} thus extends to $\Re(s)>1/2$.
 \vskip 3 pt 
Next by expanding $\log \big(1+ \frac{\e_p}{p^{s}}\big)$,  \begin{align*}  \prod_{p}\Big(1+ &\frac{\e_p}{p^{s}}\Big)
=\exp\Big\{\sum_{ p}\frac{\e_p}{p^{s}}-\frac1{2}\sum_{ p}\frac{1}{p^{2s}}+ \sum_{k\ge 3}(-1)^{k+1}\sum_{ p}\frac{\e_p^k}{kp^{ks}}\Big\}
\cr &=\exp\Big\{\sum_{ p}\frac{\e_p}{p^{s}}-\frac{\log \zeta(2s)}{2}+\sum_{k\ge 2}\sum_{ p}\frac{1}{2kp^{2s}}+ \sum_{k\ge 3}(-1)^{k+1}\sum_{ p}\frac{\e_p^k}{kp^{ks}}\Big\}
\cr &:=\exp\Big\{\sum_{ p}\frac{\e_p}{p^{s}}+R_s\Big\}
\end{align*}
where $R_s$ is a bounded quantity. 
 Thus
\begin{eqnarray*}\Big|\prod_{{\rm prime}\, p}\Big(1+ \frac{\e_p}{p^{s}}\Big) \Big|&=&\exp\Big\{\sum_{ p}\frac{\e_p\cos(t\log p)}{p^{\s}}\Big\}\exp\Big\{\Re(R_s)\Big\}
\cr &\ge &\exp\Big\{\sum_{ p}\frac{\e_p\cos(t\log p)}{p^{\s}}\Big\}\exp\Big\{-|\Re(R_s)|\Big\}.
\end{eqnarray*}

As 
\begin{equation}\label{HA}\Big|\prod_{{\rm prime}\, p}\Big(1+ \frac{f(p)}{p^{s}}\Big) \Big|=|F(s)|=|s|\Big|\int_1^\infty  \frac{M(v)}{v^{s+1}}\dd v \Big|,
\end{equation}
$$\frac{|M(v)|}{v^\s}\le \int_{-T}^{T} |F(\s+it)|  \dd t=\int_{-T}^{T} e^{\sum_{ p}\frac{\e_p}{p^{s}}+R_s}  \dd t\le C_\s \int_{-T}^{T} e^{\sum_{ p}\frac{f(p)\cos(t\log p)}{p^{\s}} }  \dd t $$   Halasz \cite{Ha},\,p.\,87   arrives at the key bound 
      \begin{align}\label{keyB}\Big(\frac{e^{-|\Re(R_s)|}}{\sqrt{\s^2+t^2}}\Big)&\,.\, 
       e^{\sum_{ p}\frac{f(p)\cos(t\log p)}{p^{\s}} }
    \le\int_1^\infty  \frac{|M(v)|}{v^{\s+1}}\,\dd v 
 \le\sup_{v\ge 1}  \frac{|M(v)|}{\sqrt{v(\s-1/2)}}+ \sup_{v\ge e^{1/\sqrt{\s-1/2}} }\frac{|M(v)|}{\sqrt{v}(\s-1/2)}
\end{align}
the last inequality follows by cutting the integral at $ x_0=\exp(1/\sqrt{\s-1/2})$.
It remains to observe that $\sup_{\s \ge 1/2}|R_s|\le K$, where $K$ is a numerical constant.  Taking $\s =\frac1{2}+\frac1{\log x}$, $x\ge 2$ in \eqref{keyB}   gives,
\begin{align}\label{key} \Big(\frac{e^{-K}}{\sqrt{(\frac1{2}+\frac1{\log x})^2+t^2}}\Big)&\,.\,e^{ \sum_{ p}\frac{f(p)\cos(t\log p)}{p^{1/2+1/\log x}} }
 \le\sqrt {\log x}\, \sup_{v\ge 1}  \frac{|M(v)|}{\sqrt{v}}+(\log x)\, \sup_{v\ge e^{\sqrt{\log x}}} \frac{|M(v)|}{\sqrt{v}}
.
\end{align}

\subsubsection{\tffs  Harper's improvments.} Harper    introduced a  very  interesting  Gaussian approach in \cite{H}; he proved in  Corollary 2,
   \begin{eqnarray}\label{harper}
  \P\Big\{ \sup_{1\le t\le 2(\log \log x)^2}\sum_{p\le x} g_p \frac{\cos (\log p)t}{p^{1/2+1/\log x}}\ \le \ \Theta (x)\Big\}&=& \mathcal O\big((\log\log\log x)^{-1/2}\big) ,\end{eqnarray}
 as $x$ tends to infinity, where 
  $\Theta(x)=\log\log x-\log\log\log x +\mathcal O\big((\log\log\log x)^{3/4}\big)$. The proof is {\it very} technical, but less than  Halasz's. Many explanations are given,    calculation details are at some places however sparse.
\vskip 2 pt 
 Building on Hal\'asz and developing much this direction, he deduced a new Omega result for the    Wintner's random M\"obius function.
 Wintner's result  was strongly improved by Hal\'asz who proved in \cite{Ha} that for some positive numerical constants $d_1$, $d_2$,
\begin{eqnarray} \label{halasz.bound}  M(x)&=&\mathcal O\Big(\sqrt x\, e^{d_1 \sqrt{\log\log x \log\log\log x}}\Big)\cr 
M(x)&\neq&\mathcal O\Big(\sqrt x\, e^{-d_2 \sqrt{\log\log x \log\log\log x}}\Big),
\end{eqnarray}
almost surely. Harper considerably improved the last estimate to
$M(x)\neq\mathcal O\big(\sqrt x\,(\log\log x)^{-A}\big)$, $A>2.5$.  
   
 The proof goes as follows. Assume to the contrary that $M(v )=\mathcal O\big(\sqrt v\,(\log\log v)^{-A}\big)$.
 Then by the variant   (\cite[p.\,611]{H}) of \eqref{key}, for all $t\ge1$,      
\begin{align*}\label{keyBa}  
e^{\sum_{ p}\frac{f(p)\cos(t\log p)}{p^{1/2+1/\log x}}  -\log t- \log\log (t+2)/2 }
  & \le
  \sqrt{\log x}\,\sup_{v\ge 1}  \frac{|M(v)|}{\sqrt{v }}+\log x \,\sup_{v\ge e^{\sqrt{\log x} } }
  \frac{|M(v)|}{\sqrt{v} }
  \cr  & \le
  \sqrt{\log x}\,\sup_{1\le v\le e^{\sqrt{\log x} } }  \frac{|M(v)|}{\sqrt{v }}+2\log x \,\sup_{v\ge e^{\sqrt{\log x} } }
  \frac{|M(v)|}{\sqrt{v} },
 \end{align*}
this  would imply 

$$e^{\sum_{ p}\frac{f(p)\cos(t\log p)}{p^{1/2+1/\log x}} -\log t \log\log (t+2)/2 }  \le\sqrt{\log x}\,\sup_{v\ge 1}  \frac{|M(v)|}{\sqrt{v }}
   +\log x \,
   \sup_{v\ge e^{\sqrt{\log x} }}
   \frac{1}{(\log\log v)^A }   \le C\sqrt{\log x} .$$

It remains to conclude by using \eqref{harper}
  that this leads to a contradiction.

\subsection{Preliminaries.
  }
   
     We  simultaneously estimate   the  local suprema of  $X_{y,x}(t)$, $X_y(t)$,
the local suprema of the remaining process $X_{2,y}(t)$
will be examined in last. 
  We    use  Corollary  \ref{ass1th} and Theorem \ref{th.1}. 
 The first step  consists of estimating  correlations, whereas in the  second, shorter,  we prove the deviation bound. 
       \vskip 5 pt
 \underline{Some preliminary considerations:}
  \vskip 2 pt (a) Since   $c_1 {\sqrt y}/{\log y}\le \sum_{p\le y}p^{-1/2}\le c_2 {\sqrt y}/{\log y}$, $y\ge 2$, (Prachar \cite[Satz 4.12]{A}), it follows that   
$$\sum_{p\le y}p^{-1/2}= \mathcal O\big(\frac{\log\log x}{\log\log\log x} \big) ,$$
 if $y=\mathcal O((\log\log x)^2)$, and so in that case the beginning of sum may be neglected. 
 \vskip 2 pt (b)  
 We have    for $s,t\ge 0$, 
 \begin{eqnarray*} \E X(s)X(t)&=& \sum_{y<p\le x}
  \frac{\cos (s\log p)\cos (t\log p)}{p^{ 1 +\frac{2}{\log x}}}
\cr &=&\frac12 \sum_{y<p\le x} \frac{\cos \big((s-t)\log p\big)+\cos \big((s+t)\log p\big)}{p^{ 1 +\frac{2}{\log x}}}.
  \end{eqnarray*}
 
   \vskip 2 pt (c) We estimate  the    functions 
$$ \phi_\X (\theta):=\sum_{y<p\le \X} \frac{\cos (\theta\log p) }{p^{ 1 +\frac{2}{\log x}}}
 \qq(2\le y<x\le\X\le \infty).    $$
Then
\begin{equation}\label{Xtilde}\begin{cases} \E X_{y,x}(s)X_{y,x}(t)&=\frac12\big(\phi_x ( |s-t|)+\phi_x (|s+t|) \big) \cr
\E X_{y,x}^2(t) &=\frac12\big(\phi_x ( 0)+\phi_x (2t) \big)\cr
\E \widetilde X_{y,x}(s)\widetilde X_{y,x}(t)&=\displaystyle{\frac{\phi_x ( s-t)+\phi_x (s+t)}{\big(\phi_x ( 0)+\phi_x (2s) \big)^{1/2}\big(\phi_x ( 0)+\phi_x (2t) \big)^{1/2}}}\cr
\E X_y(s)X_y(t)&=\frac12\big(\phi_\infty ( |s-t|)+\phi_\infty (|s+t|) \big)
\cr 
\E  X_y(t)^2&=\frac12\big(\phi_\infty ( 0)+\phi_\infty (2t) \big).
\end{cases} 
\end{equation} 
We  first note 
 that    for $\e>0$  (\cite{Te},\,p.\,17), $\zeta(1+\e) = \frac{1}{\e} +\mathcal O(1)$. Using   Euler's formula $\zeta(1+\e)= \prod_{p }\big(1-\frac{1}{p^{1+\e}}\big)^{-1}$, next   expanding $\log \big(1- \frac{
1}{p^{1+\e}}\big)$, 
 \begin{align}\label{form} \zeta(1+\e)  &= \exp\Big\{-\sum_{p }\log\big(1-\frac{1}{p^{1+\e}}\big) \Big\} 
\ =\exp\Big\{\sum_{p }\frac{1}{p^{{1+\e}}} + \sum_{k\ge 2} \sum_{ p }\frac{1}{kp^{k(1+\e)}}\Big\}
\cr &:=\exp\Big\{\sum_{p }\frac{1}{p^{{1+\e}}} +T\Big\}=\frac{1}{\e} +\mathcal O(1).
\end{align}
We have $T\le  \sum_{ p } p^{-2  }  <\infty $, as $\sum_{k\ge 2} \frac{1}{k }e^{-k(1+\e)\log p}\le \sum_{k\ge 2} \frac{1}{k }e^{-k \log p}\le \frac12p^{-2  }\sum_{k\ge 0}  e^{-k \log p}= \frac12 \frac{p^{-2  }}{1-p^{- 1 }}
$.   On taking logarithms in both sides we get
$$\sum_{p }\frac{1}{p^{{1+\e}}} =\log   \frac{1}{\e}    +\mathcal O(1) ,$$ 
noting  that $\log  (Z+\mathcal O(1))= \log   Z +\mathcal O(1/Z)$, $Z\ge 3$, say. 
\vskip 2 pt 
In particular taking  $\e =\frac{2}{\log x}$, it follows that
\ben\label{sum.pz}\sum_{p }\frac{1}{p^{{1+\frac{2}{\log x}}}} =\log  \log x 
 + \mathcal O(1).
\een
By a different argument  given after \eqref{int.part}, it  also follows   that \begin{eqnarray}\label{sum.yxx}  
 \phi_x(0)=\sum_{y<p\le x} \frac{1}{p^{ 1 +\frac{2}{\log x}}} = \log \Big(\frac{\log x}{\log y}\Big) +  \mathcal O\big( 1\big)  ,\end{eqnarray}

\vskip 2 pt (d)    Consider the measures $\m= \sum_{ p \in J} \frac{1 }{p^{\eta  }}\d_{p}$, $\eta>1$, where $\d$ stands for the Dirac measure and $J\subseteq \mathcal P$, or any step function $f\in L^1(\R)$ approximating $\m$ in $L^1(\R)$. By Riemann-Lebesgue lemma (Kawata \cite[p.\,43]{Ka}), $\lim_{\theta\to \infty} \hat{f}(\theta)=0$. We deduce that $\lim_{\theta\to \infty} \hat{\m}(\theta)=0$. Hence  we infer that  
\beq \label{phi.Riemann-Lebesgue}\lim_{\theta\to \infty} \phi_x(\theta)=0 . 
\eeq 
  
     For the computation of the main term of $\phi_x(\theta)$,  we start with an exact integral formula. Main technical points and arguments are displayed.

\vskip 3 pt

\subsection{\bf Correlation estimates}\label{s5.1}    
  We refer to Rosser and Schoenfeld \cite[pp.\,66-67]{RS} concerning the following facts.  Let $\pi(x)$ denote the number of primes less or equal to $x$.  Let $f$ be a function of primes.  By using Stieltjes integral, we have
 $ \sum_{p\le x}f(p) = \int_{2^-}^x f(v) \pi(\dd v)$. 
  By integrating by parts,
$$\sum_{p\le x}f(p) =\int_{2^-}^x\frac{ f(v)}{\log v}\, \dd v+ f(2) {\rm li}(2)+ f(x) \big\{\pi(x)-{\rm li}(x)\big\}-\int_{2 }^x f'(v)\big\{\pi(v)-{\rm li}(v)\big\} \dd v,$$
where ${\rm li}(x)  $ is the logarithmic integral defined by ${\rm li}(x)= {\rm Ei}(\log x)$ and ${\rm Ei}(v)$ is the exponential integral defined by
$${\rm Ei}(v)= \lim_{\e\to 0_+}\Big\{\int_{-\infty}^{-\e}\frac{e^t}{t} \dd t+ \int_{\e}^v \frac{e^t}{t} \dd t\Big\}.$$ 
Then
 $\int_2^x \frac{\dd y}{\log y} = {\rm li}(y)-{\rm li}(2)$.   
   By Vinogradov's estimate
$ \pi(x)= {\rm li}(x)+ \mathcal O \big(x\exp\big\{ -a \big(\log x\big)^{3/5}\big\}\big)$,
$a$ being an absolute constant.  Also,
\beq \label{int.part}\sum_{y<p\le x}f(p) =\int_{y }^x \frac{f(v)}{\log v}\, \dd v+   f(x) \big\{\pi(x)-{\rm li}(x)\big\}-f(y) \big\{\pi(y)-{\rm li}(y)\big\} -\int_{y }^x f'(v)\big\{\pi(v)-{\rm li}(v)\big\} \dd v.
\eeq

We use it to estimate $\phi_x (0)$. Applying  \eqref{int.part}  with  $f(v)=\frac{1 }{v^{ 1 + {2}/{\log x}}}$, ($f'(v)=-\big( 1+ {2}/{\log x}\big)\frac{1}{v^{2+ {2}/{\log x} }}  $),
 we draw
\begin{eqnarray*}  
\phi_x(0)= \sum_{y<p\le x} \frac{1 }{p^{ 1 + {2}/{\log x}}}
  &=&\int_{y }^x \frac{1 }{v^{ 1 +\frac{2}{\log x}}\log v} \dd v+  \mathcal O\Big( e^{ -a   (\log x\ )^{3/5}}\Big) + \mathcal O\Big( e^{ -2 \frac{\log y}{\log x}-a   (\log y  )^{3/5}}\Big)  
  \cr  & & +\mathcal O\Big(  \big( 1+\frac{2}{\log x}\big) \int_{y }^x \frac{v\,e^{ -a   (\log v )^{3/5}}}{v^{2+\frac{2}{\log x} }}    \dd v\Big)
    .\end{eqnarray*}
Now
$$ \int_{y }^x \frac{e^{ -a   (\log v )^{3/5}} }{v^{1+\frac{2}{\log x} }}   \dd v
\le e^{ -(a /2)  (\log y )^{3/5}}\int_{2 }^\infty \frac{e^{ -(a /2)  (\log v )^{3/5}} }{v }   \dd v= C_a\,e^{ -(a /2)  (\log y )^{3/5}} ,   $$  
\begin{eqnarray*}\int_y^x \frac{1}{v^{1+\frac{2}{\log x}}\log v} \dd v&\buildrel{v=e^u}\over{=} &\int_{\log y}^{\log x} \frac{1}{u}e^{- (\frac{2}{ \log x} )u}\dd u\buildrel{ u= t(\frac{\log x}{2})}\over{=} \int_{2(\frac{ \log y}{\log x})}^{2} \frac{1}{t}e^{-   t}\dd t
\cr &=&\log \Big(\frac{\log x}{\log y}\Big)+ \int_{2(\frac{ \log y}{\log x})}^{2} \frac{1}{t}(e^{-   t}-1)\dd t  \,= \, \log \Big(\frac{\log x}{\log y}\Big) +  \mathcal O\big( 1\big)
.
\end{eqnarray*}Thus \eqref{sum.yxx} holds.   

    We   prove the lemma below. 
\begin{lemma}\label{bound.cor.} Let $ 2\le y<x\le\infty$.  There exist  absolute constants $a$, $c_0$   
and $C^*$  such that if $\min(|s-t|, |s+t|) \log x \ge c_0$, 
\vskip 2 pt  {\rm (a)} we have 
 \begin{eqnarray*}\big|\E  X_{y,x}(s)  X_{y,x}(t)\big|&\le& C^*  \bigg(\frac{ 1}{  \min(|s-t|, |s+t|)   \log y  }+
 \, e^{-a (\log y)^{3/5}   }  \bigg) ,
 \end{eqnarray*}
  \begin{eqnarray*}\big|\E  X_y(s)  X_y(t)\big|&\le& C^*  \bigg(\frac{ 1}{  \min(|s-t|, |s+t|)   \log y  }+
 \, e^{-a (\log y)^{3/5}   }  \bigg) .
 \end{eqnarray*}\vskip 2 pt {\rm (b)}
 Further there exist  absolute constants $y_0$, $\kappa$ such that  if    $x> y\ge y_0$  and $2s\ge y_0$, then    \begin{eqnarray*}  \E X_{y,x}^2(s)\ge   \frac{1 }{2}   \log \Big(\frac{\log x}{\k\log y}\Big) . \end{eqnarray*}
 
 \end{lemma}
 \begin{remark}\rm The condition  $\min(|s-t|, |s+t|) \log x \ge c_0$  is only used in \eqref{A2}. See also after \eqref{phi.est} below.
\end{remark} 

 \begin{proof}    It suffices to prove,   $\phi_\X$ being even,      that there exist  positive  absolute constants $a$, $c_0$ and $C$, such that for $\X\ge x$, $ \theta > 0$, 
  $ \theta \log x \ge c_0$,     \begin{eqnarray}\label{phi.est} \big|\phi_\X(\theta)\big|& \le &
 C\min \Big(\frac{1 }{  \theta   \log y  }+  
 \, e^{-a (\log y)^{3/5}   },\log \big(\frac{\log x}{\log y}\big) \Big). \end{eqnarray}
Then in view of \eqref{Xtilde},  
$$ | \E X_{y,x}(s)X_{y,x}(t) |\le \frac12\big(|\phi_x ( |s-t|)|+|\phi_x (|s+t|)| \big), $$  
next if $\min(|s-t|, |s+t|)\log x\ge c_0$, by applying \eqref{phi.est} successively with  $\theta= |s-t|  $,  $\theta=  |s+t|  $ we get 
 $$| \E X_{y,x}(s)X_{y,x}(t) |\le C\min \Big(\frac{1 }{ \min(|s-t|, |s+t|)   \log y  }+  
 \, e^{-a (\log y)^{3/5}   },\log \big(\frac{\log x}{\log y}\big) \Big).$$
 hence the first estimate. \vskip 2 pt This also implies that   $  |\phi_x(\theta)|\le 1$, whenever $y\ge y_0$  and $\theta \ge y_0$,  for some positive $y_0$. So that    if $2s\ge y_0$ we have  
by \eqref{Xtilde}, $\E X_{y,x}^2(s) 
 \ge \frac{1 }{2}    \big(\phi_x ( 0)-1  \big).$ Thus for some numerical  constant $\k>0$,
     \beq\label{variance.ex}\E X_{y,x}^2(s) \ge   \frac{1 }{2}     \log \Big(\frac{\log x}{\k\log y}\Big),    \qq \quad (2s\ge y_0), \eeq 
  hence   the second estimate.
\vskip 2 pt  Taking any finite set $\{s_i, i=1,\ldots, \m\}$, we   bound (towards applying Corollary \ref{ass1th}) as follows,
    \begin{eqnarray}\label{commutation} \P\Bigg\{ \sup_{i=1}^\m \frac{  X(s_i)}{\sqrt{  \frac{1 }{2}     \log  (\frac{\log x}{\k\log y} )    }}\le \sqrt{ 2   \log \frac{\m}{   \log \m}    }\Bigg\} 
&\le &\P\Bigg\{ \sup_{i=1}^\m \frac{  X_{y,x}(s_i)}{ \sqrt{\E X_{y,x}^2(s_i)}}\le \sqrt{ 2   \log \frac{\m}{   \log \m}    }\Bigg\}
\cr   & = &   \P\Big\{ \sup_{i=1}^ \m \widetilde X_{y,x}(s_i)\le 
  \sqrt{2    \log \frac{\m}{   \log \m}   }\Big\}
 .
\end{eqnarray}  

   We   apply \eqref{int.part}  to $f(v)=\frac{\cos (\theta\log v) }{v^{ 1 +\frac{2}{\log x}}}$. As  $ f'(v)= \frac{1}{v^{2+\frac{2}{\log x} }} \big\{-\sin (\theta \log v) + \cos (\theta \log v)\big( 1+\frac{2}{\log x}\big)\big\} $,
we   get 
\ben\label{precise.phi}  \phi_\X (\theta)
&=&\int_{y }^ \X \frac{\cos (\theta\log v) }{v^{ 1 +\frac{2}{\log x}}\log v} \dd v+   \frac{\cos (\theta\log \X) }{\X ^{ 1 +\frac{2}{\log x}}} \big\{\pi(\X)-{\rm li}(\X)\big\}-\frac{\cos (\theta\log y) }{y^{ 1 +\frac{2}{\log x}}} \big\{\pi(y)-{\rm li}(y)\big\}
\cr & &  -\int_{y }^ \X \frac{1}{v^{2+\frac{2}{\log x} }} \Big\{-\sin (\theta \log v) + \cos (\theta \log v)\Big( 1+\frac{2}{\log x}\Big)\Big\}\big\{\pi(v)-{\rm li}(v)\big\} \dd v.
\een

   First note concerning main terms that 
$$\int_y^\X \frac{\cos(\theta \log v)}{v^{1+\frac{2}{\log x}}\log v} \dd v\buildrel{v=e^u}\over{=} \int_{\log y}^{\log \X} \frac{\cos \theta u}{u}e^{- (\frac{2}{ \log x} )u}\dd u\buildrel{ u= \frac{t}{\theta}}\over{=} \int_{\theta\log y}^{\theta \log \X} \frac{\cos t}{t}e^{- (\frac{2}{ \theta\log x} ) t}\dd t.$$
  We have the formula  
 \beq\label{863}\int e^{\a v} \cos v \dd v = e^{\a v} \Big( \frac{\a \cos v +\sin v}{\a^2+1}\Big).
 \eeq
Choose 
 $\a= - \frac{2}{\theta\log x} $. 
Integrating by part
 we get 
\ben \label{phi.863}\int_{\theta\log y}^{\theta \log \X} \frac{\cos t}{t}e^{\a t}\dd t
  &= &e^{\a t} \big( \frac{\a \cos t +\sin t}{t(\a^2+1)}\big)\Big|_{\theta\log y}^{\theta\log \X}+ \int_{\theta\log y}^{\theta \log \X} e^{\a t} \Big( \frac{\a \cos t +\sin t}{t^2(\a^2+1)}\Big)\dd t
\cr &= & e^{-2\frac{\log \X}{\log x}}\bigg( \frac{-\frac{2}{\theta \log x}\cos \theta \log \X+ \sin \theta \log \X}{\theta \log \X\left(\frac{4}{\theta^2(\log x)^2}+1 \right) }\bigg)  
\cr &  &-e^{-2\frac{\log y}{\log x}}\bigg( \frac{-\frac{2}{\theta \log x}\cos \theta \log y
 + \sin \theta \log y}{\theta \log y\left(\frac{4}{\theta^2(\log x)^2}+1 \right) }\bigg)  
 \cr &  &+ \int_{\theta\log y}^{\theta \log \X} e^{- \frac{2}{\theta\log x} t} \Big( \frac{- \frac{2}{\theta\log x} \cos t +\sin t}{t^2(  \frac{4}{(\theta\log x)^2}+1)}\Big)\dd t
\cr &:= & A_1- A_2+ A_3. \een 
       As $\X\to +\infty$, $A_1$ vanishes    so that \ben \label{phi.863infty}
\int_y^\infty\frac{\cos(\theta \log v)}{v^{1+\frac{2}{\log x}}\log v} \dd v
    &= & - A_2+ \int_{\theta\log y}^{\infty}    e^{- \frac{2}{\theta\log x} t} \Big( \frac{- \frac{2}{\theta\log x} \cos t +\sin t}{t^2(  \frac{4}{(\theta\log x)^2}+1)}\Big)\dd t. \een
  Multiplying both numerator and denominator in $ A_2$ by $\theta \log x$ gives
 
\beq\label{A2}|A_2|= e^{-2\frac{\log y}{\log x}}\bigg|   \frac{- 2  \cos( \theta \log y)
 + (\theta \log x)\sin (\theta \log y)}{\theta^2 (\log y)(\log x)\left(\frac{4}{\theta^2(\log x)^2}+1 \right) }\bigg| \le   \frac{2+|\theta|\log x}{\theta^2(\log y)(\log x)}\le \, \frac{ C}{|\theta| (\log y) },
 \eeq
since   $|\theta| \log x\ge c_0$.  
 Also $\big|  \frac{- \frac{2}{\theta\log x} \cos t +\sin t}{t^2(  \frac{4}{(\theta\log x)^2}+1)}\big|\le c_0't^{-2}$, so that $$   |A_3|
     \le
  \int_{\theta\log y}^{\theta \log \X}\frac{1}{t^2}  \dd t
 = \big(\frac{1}{\theta\log y}-\frac{1}{\theta\log \X}\big)\le   \frac{1}{\theta\log y}.$$
Consequently by letting $\X\to +\infty$ we get
\beq\bigg|\int_y^\infty \frac{\cos(\theta \log v)}{v^{1+\frac{2}{\log x}}\log v} \dd v\bigg| \le  
\frac{C    }{\theta  \log y  }.\eeq
Further 
 $$ |A_1|\le  e^{-2\frac{\log \X}{\log x}}\bigg|  \frac{- {2} \cos \theta \log \X+ ({\theta \log x})\sin \theta \log \X}{\theta \log \X\left(\frac{4}{\theta (\log x) }+{\theta \log x} \right) } \bigg| \le \,C\frac{   |\theta|\log x   }{   \theta ^2 (\log \X)( \log x)   }= \, \frac{  C  }{   |\theta| (\log \X)  }.$$

Now   obviously if
 $$ B_1= \frac{\cos( \theta \log x)}{x^{1 +\frac{2}{\log x}}}\big\{\pi(x) -{\rm li}(x)\big\}, \qq B_2= \frac{\cos( \theta \log y)}{y^{1 +\frac{2}{\log x}}}\big\{\pi(y) -{\rm li}(y)\big\},$$
 then   $\max\big( |B_1|,|B_2|\big)\le C e^{-a (\log y)^{3/5}}
 $.
If $$B_3 = \int_y^x\frac{\frac{-\sin(\theta\log v)}{v}+\cos (\theta\log v)\big(1 +\frac{2}{\log x}\big) }{v^{2 +\frac{2}{\log x}}}\Big\{\pi(v) -{\rm li}(v)\Big\} \dd v,$$
  then $$|B_3 |  \le 3\,   \int_y^x  e^{-a (\log v)^{3/5}} \frac{ \dd v}{v }\le  3\,     e^{-(a/2)  (\log y)^{3/5}} \int_2^\infty  e^{-(a/2)  (\log v)^{3/5}} \frac{ \dd v}{v }=  
C\, e^{-(a/2) (\log y)^{3/5}   }
. $$
  So that
  $ |B_1|+|B_2|+|B_3| \le   C e^{-(a/2) (\log y)^{3/5}}
  $. 
\vskip 2 pt

 Hence we get   for  $2\le y<x\le \infty$,   $\theta \log x\ge c_0$,
 \begin{eqnarray*} \big|\phi_x(\theta)\big|& \le &C\Big(\frac{1 }{  \theta   \log y  }+
 \, e^{-(a/2) (\log y)^{3/5}   } \Big).
  \end{eqnarray*}
 The proof is complete.  \end{proof}
 

\subsection{Proof of Theorem \ref{demi.logx1}}
Technical Propositions 1 and  2 in  \cite{Ha} are not invoked, and  in  place of Exact formula 1, Comparison inequality 1,    an older one we proved,  easy to apply,    suffices.  It is already  integrated in the proof of Theorem \ref{th.1} which is the main ingredient of the proof.  Theorem \ref{F.decoupling.th} is not implemented.
\vskip 3 pt 
{\bf (i)} Let $x\ge y\ge 2$ and $0<\eta<1$. Pick   $\varpi  $    satisfying  the condition
 $$c_0/\log y \le  \varpi\le   e^{ a (\log y)^{3/5}   }/  \log y , $$
where $c_0$, and $C^*$ below are the same constants as in Lemma \ref{bound.cor.}.   We put 
 \beq   k = \Big\lfloor \frac{\eta}{2C^*}\,   \Big(\log  \frac{\log x}{\k\log y}  \Big)   \varpi\log y\Big\rfloor,
\qq \Theta= \sqrt{    2\log \frac{k}{   \log k}   }. 
\eeq 

Now let reals $   s_i  \ge 1$,  $1\le i\le k   $ be such that $$\theta = \min \big((|s_i-s_j|, |s_i+s_j|, 2s_i  ) : 1\le i,j\le k ,  i\neq j\big)\ge \varpi  . $$
   By Lemma \ref{bound.cor.} and  \eqref{Xtilde},    for $i\neq j$,   
\ben \label{cov.s,t} \big|\E \widetilde X_{y,x}(s_i)\widetilde X_{y,x}(s_j)\big|&\le& C^*\min \bigg(\frac{1}{ \big(  \log  \frac{\log x}{\k\log y}     \big) } \Big(   \frac{1}{\varpi
 \log y }+
 \, e^{-a (\log y)^{3/5}   } \Big),1\bigg)
 \cr &\le &\frac{C^*}{ \big(  \log  \frac{\log x}{\k\log y}     \big) \varpi\log y}:=\l.  
 \een
We have $\l\le
\eta/2k$, and thus in view of Corollary \ref{ass1th},    there exist positive constants $C_\eta, C'_\eta$ depending on $\eta$ only such that 
$$ \P\Big\{ \sup_{i=1}^k \widetilde X_{y,x}(s_i)\le  \sqrt{    2  \log\frac{k}{\log k}   -   \frac{\eta   \log k }{k} }\Big\}
 \le   C_\eta\, e^{- C'_\eta \sqrt{ \log k  }  }.
$$ 
  
 
 Hence in view of inequality \eqref{commutation},  
\begin{eqnarray}  \P\bigg\{ \sup_{i=1}^k   X_{y,x}(s_i)  \le \sqrt{   \Big(  \log  (\frac{\log x}{\k\log y} )   \Big)\Big(       2  \log\frac{k}{\log k}   -   \frac{\eta   \log k }{k} \Big)}\bigg\} 
&\le & C_\eta\, e^{- C'_\eta \sqrt{ \log k  }  }.
\end{eqnarray}  
 Consequently,
\begin{eqnarray*} \P\bigg\{ \sup_{ 1\le t\le k \varpi\atop{
}
}    \sum_{y<p\le x}\frac{g_p\cos t\log p}{p^{\frac12+ \frac{1}{\log x}}} \le \sqrt{   \Big(  \log  (\frac{\log x}{\k\log y} )   \Big)\Big(   2  \log\frac{k}{\log k}   -   \frac{\eta   \log k }{k} \Big)}\bigg\} 
  &\le &  C_\eta\, e^{- C'_\eta \sqrt{ \log k  }  }.
\end{eqnarray*}

  Also
 \begin{eqnarray}  \P\bigg\{ \sup_{i=1}^k   X_y(s_i)  \le \sqrt{   \Big(  \log  (\frac{\log x}{\k\log y} )   \Big)\Big(       2  \log\frac{k}{\log k}   -   \frac{\eta   \log k }{k} \Big)}\bigg\} 
&\le & C_\eta\, e^{- C'_\eta \sqrt{ \log k  }  }.
\end{eqnarray}  
Assume now that  $y=\mathcal O((\log\log x)^2)$ and recall that  
 $$\sum_{p\le y}p^{-1/2}= \mathcal O\big(\frac{\log\log x}{\log\log\log x} \big) .$$
 We deduce that
 \begin{align*}  \P\bigg\{ \sup_{i=1}^k   \sum_{p }\frac{g_p\cos s_i\log p}{p^{\frac12+ \frac{1}{\log x}}}  \le  \sqrt{   \Big(  \log  (\frac{\log x}{\k\log y} )    \Big)\Big(       2  \log\frac{k}{\log k}   -   \frac{\eta   \log k }{k} \Big)}\bigg\} 
   \le  \ C_\eta\, e^{- C'_\eta \sqrt{ \log k  }  }.
\end{align*} 
%

 {\bf (ii)}   
 Let $0<\e<1$, $\{s_i,1\le i\le k\}\subset [1,2]$ and assume that $\E  \widetilde X_{y,x}(n+s_i)\widetilde X_{y,x}(m+s_j)\le 1-\e$, for all $i\neq j$ and all $n,m$ even, $0\le n\neq m\le N$, and  further that 
 $$\theta = \min \big((|n+s_i-m-s_j|, |n+s_i+m+s_j|  ) : 1\le i,j\le k , 0\le n\neq m\le N\big)\ge \varpi .$$
By Lemma \ref{bound.cor.},
$$|\E  \widetilde X_{y,x}(n+s_i) \widetilde X_{y,x}(n+s_j)|\le \l, $$    
for any  $0\le n\neq m\le N$, $n$, $m$ even and any $1\le i, j\le k$, where $\l$ is as in (i). 
  Let   $\p(u)>c$ non decreasing with $u\ge1$.  
By   Theorem \ref{th.1},  
 \begin{align}\label{th.1appl}  \P\Big\{  \sup_{n=1}^N   \sup_{i=1}^k  \widetilde  X_{y,x}(n+s_i)\le & \sqrt{(2/\e)\log (k\p(k))} \Big\} \le 
e^{ -\frac{ CcN   }{    \p(k) \sqrt {2\log(k \p(k))}}}
   +    \frac{\l k^2N}{4\pi\sqrt{1-\l^2}}   e^{-\frac{  2 \log (k\p(k))  }{\e (1+\l )} } .\end{align}


 {\bf (iii)}   In this part we use  Harper's    approximating sets and we refer to \cite[p.\,602-603]{H}. We   define for $0\le n\le (\log\log x)^2=N$, the sets $$ T_n  =\Big\{2n+1+i\,\frac{ E }{\log x},  1\le i\le  k\Big\}\subseteq [2n+1, 2n+2].  
 $$
 Let $s,t $ be distinct in $  T_0$, $s=i \frac{E}{\log x}$, $t=j\frac{E}{\log x}$.
   We choose here $\e=  (\log\log x)^{ -b}$, $  1/2\le  b<1$,     $E= k=(\frac{\log x}{\log y})^{\e}\le 
 e^{ (\log \log x)^{1-b}}$. Recalling    that $\E \widetilde X_{y,x}( s )\widetilde X_{y,x}(t)=\frac{\log \frac{1 }{|s-t| \log y} }{\log\frac{\log x}{\log y}   } +\mathcal O \big(\frac{1}{\log\log x}\big) $, we have 
  $|s-t| \log y=|i-j|E\frac{\log y}{\log x}=|i-j|\big(\frac{\log x}{\log y}\big)^{-(1-\e)}\ge \big(\frac{\log x}{\log y}\big)^{-(1-\e)}$. Thus
$$\E \widetilde X_{y,x}( s )\widetilde X_{y,x}(t)\le  \frac{\log\big(\frac{ \log x }{   \log y}\big)^{1-\e} }{\log\frac{\log x}{\log y}   }+\mathcal O \big( {1}/{\log\log x}\big)=  1- \e+\mathcal O \big( {1}/{\log\log x}\big)\le 1-\e/2.$$ 
 
 By assumption $  \p(u)= \log u$, so that 
 $$\log \p(k)=  \log\log (\frac{\log x}{\log y})^{\e}= -b\log\log\log x+ \log\log  \frac{\log x}{\log y} = \log\frac{(\log\log x)^{1 -b}}{\log \log y}  . $$Thus,  
   $$  {\frac{2}{\e }\log (k\p(k))} =   2 \log \frac{\log x}{ \log y}  +2(\log\log x)^{ b}\Big(  \log\frac{(\log\log x)^{1 -b}}{\log \log y}\Big)$$
  
  ($Nk=(\log\log x)^2( {\log x}/{\log y})^{1-\e}\sim  e^{c (\log \log x)^{1-b}}$)

By applying Theorem \ref{th.1}  and noting that $k\le (\log x)^{\e}=e^{ (\log\log x)^{1- b}}$, $\log k\le (\log\log x)^{1- b}$,
we get, ($\l= \frac{C^*}{  (  \log  \frac{\log x}{\k\log y}      ) \log y}\asymp \log\log x$) 
\begin{align*}  \P&\bigg\{  \sup_{n=1}^N     \sup_{i=1}^k  \widetilde   X_{y,x}(n+s_i)  \le \sqrt{  2 \log \frac{\log x}{ \log y}  +2(\log\log x)^{ b}\Big(  \log\frac{(\log\log x)^{1 -b}}{\log \log y}\Big)} \bigg\} 
\cr &= \P\Big\{  \sup_{n=1}^N    \sup_{i=1}^k  \widetilde   X_{y,x}(n+s_i)  \le  \sqrt{(2/\e)\log (k\p(k))} \Big\} \le e^{ -\frac{ CN   }{    \p(k) \sqrt {2\log(k \p(k))}}}
   +    \frac{\l k^2N}{4\pi\sqrt{1-\l^2}}   e^{-\frac{  2 \log (k\p(k))  }{\e (1+\l )} }
\cr   &\le e^{ -\frac{ CN   }{  (  \log\log k) \sqrt {3\log k }}}
   +    \frac{\l k^2 N}{4\pi\sqrt{1-\l^2}} e^{   -\frac{  2 \log (k\p(k))  }{\e (1+\l )} }
\cr  &\le e^{ -\frac{  C_b(\log\log x)^{2-(1-b)/2}     }{     \log\log \log x  }}
   +   C  \l  (\log\log x)^2    e^{ 2(\log \log x)^{1-b} -\frac{  2  \log \frac{\log x}{ \log y}  +  \frac{2}{\e }   \log \p(k)  }{  (1+\l )} }
 \cr  &=e^{ -\frac{  C_b(\log\log x)^{(3+b)/2}     }{     \log\log \log x  }}
   + C    \l  (\log\log x)^2     e^{ 2(\log \log x)^{1-b} -\frac{  2}{1+\l}  \log   \log x- \frac{   - 2 \log \log y   +  \frac{2}{\e }   \log \p(( {\log x}/{\log y})^{\e})  }{  (1+\l )} }.
\end{align*} 
Writing that $ \frac{  2}{1+\l}    =2 -\frac{2\l}{1+\l}$  we get $$2(\log \log x)^{1-b}  -\frac{2}{1+\l} \log   \log x - \frac{   - 2 \log \log y   +  2 (\log\log x)^{ b}   \log \p(k)  }{  (1+\l )}$$
$$=  2(\log \log x)^{1-b}  -2\log   \log x +\frac{2\l}{1+\l} \log   \log x - \frac{   - 2 \log \log y   +  2 (\log\log x)^{ b}   \log \p(k)  }{  (1+\l )},$$
and since $\l \log\log x= \mathcal O(1)$, $$=  2(\log \log x)^{1-b}  -2\log   \log x +\mathcal O(1)  - \frac{   - 2 \log \log y   +  2 (\log\log x)^{ b}   \log \p(k)  }{  (1+\l )}.$$
Thus $$     \l  (\log\log x)^2     e^{2(\log \log x)^{1-b}  -2\log   \log x +\frac{2\l}{1+\l} \log   \log x - \frac{   - 2 \log \log y   +  2 (\log\log x)^{ b}   \log \p(k)  }{  (1+\l )}} $$
   $$\le \frac{C}{(\log x)^2} \,
 e^{2(\log \log x)^{1-b}     + \frac{     2 \log \log y  }{  (1+\l )}  -   \frac{ 2(\log\log x)^{ b}   \log \p(k)  }{  (1+\l )}} $$
$$\le C\,\frac{  e^{2(\log \log x)^{1-b}     + \frac{     2 \log \log \log x  }{  (1+\l )}  -   \frac{ 2(\log\log x)^{ b}   \log \p(k)  }{  (1+\l )}} }{(\log x)^2} .
 $$
As $  1/2\le  b<1$, this is less than
 $ \,\frac{ C  }{(\log x)^2 e^{     M(\log\log x)^{ b}  }} $, for any large positive number $M$,  as soon as $x\ge x(M)$. Consequently,
\begin{align}\label{boundwide}   \P\Big\{  \sup_{n=1}^N    \sup_{i=1}^k  \widetilde  X_{y,x}(n+s_i)\le  \sqrt{2\log \frac{\log x}{ \log y}  +(\log\log x)^{ b} \log \p(k)
} \Big\}
 \le \,\frac{ C  }{(\log x)^2 e^{     M(\log\log x)^{ b}  }},\end{align}
  {\it a fortiori} \begin{align*}  \P\Big\{  
  \sup_{n=1}^N    \sup_{i=1}^k   X_{y,x}(n+s_i)&\le  \sqrt{\Big( \log \frac{\log x}{\k \log y}\Big) \Big(2\log \frac{\log x}{ \log y}   +(\log\log x)^{ b} \log \p(( \frac{\log x}{\log y})^{\e})
\Big)} \Big\} 
\cr  & \qq \qq
\le \,\frac{ C  }{(\log x)^2 e^{     M(\log\log x)^{ b}  }}.\end{align*}

 Noticing that
$$\Big( \log \frac{\log x}{\k \log y}\Big) \Big(2\log \frac{\log x}{ \log y}   +(\log\log x)^{ b} \log \p(( \frac{\log x}{\log y})^{\e})
\Big)$$
$$= 2\Big(\log  \frac{\log x}{\log y}\Big)^2\!\!+\!\! \Big(\log  \frac{\log x}{\log y}\Big)\bigg\{ (\log\log x)^b\log\Big(\frac{\log  \frac{\log x}{\log y}}{(\log\log x)^b}\Big)-2\log \k\bigg\} ,$$
we obtain
 \begin{align*}  \P\bigg\{ 
 \sup_{n=1}^N    \sup_{i=1}^k    X_{y,x}(n+s_i)&\le  \sqrt{2\Big(\log  \frac{\log x}{\log y}\Big)^2\!\!+\!\! \Big(\log  \frac{\log x}{\log y}\Big)\Big\{ (\log\log x)^b\log\Big(\frac{\log  \frac{\log x}{\log y}}{(\log\log x)^b}\Big)\!\!-\!\!2\log \k \Big\}} \bigg\} \cr  & \qq \qq \le \,\frac{ C  }{(\log x)^2 e^{     M(\log\log x)^{ b}  }}.\end{align*}
Hence also,
\begin{align*}  \P\Big\{  
\sup_{n=1}^N    \sup_{i=1}^k     X_{y,x}(n+s_i)&\le  2\log \frac{\log x}{ \log y}\Big\} 
\le \,\frac{ C  }{(\log x)^2 e^{     M(\log\log x)^{ b}  }}.\end{align*}
Thus
\begin{align*}  \P\Big\{   \sup_{1\le t\le 2 (\log\log x)^2}  
     X_{y,x}(t)&\le  2\log \frac{\log x}{ \log y}\Big\} 
\le \,\frac{ C  }{(\log x)^2 e^{     M(\log\log x)^{ b}  }}.\end{align*}




  \vskip  8pt 
We note that the bound obtained is  sharper than in \cite{H},  Corollary 2.

\begin{remarks} \rm  (i) In \cite{H}, page 607,  the inequality $$ \P\Big\{ \sup_{1\le t\le2(\log\log x)^2}\sum_p \frac{\e_p\cos(t\log p)}{p^{1/2+1/\log x}}\le   \log\frac{\log x}{\log y} -\mathcal O(1) -(\log\log\log x)^{3/4}\Big\}= \mathcal O ((\log\log\log x)^{-1/2})$$
is said to follow from Tchebycheff's inequality,   the fact that 
 $\E\big(\sum_{p>x} \frac{\e_p\cos(t\log p)}{p^{1/2+1/\log x}}\big)^2=\mathcal O(1),$ and of the Rademacher version of inequality 
 $$\P\Big\{ \frac{1}{d(x)}\sup_{1\le t\le 2(\log\log x)^2}\sum_{p\le x} \frac{g_p\cos(t\log p)}{p^{1/2+1/\log x}}$$$$\le  \sqrt{2 \log\frac{\log x}{\log y}} -\frac{(\log\log\log x)^{3/4}}{d(x)}\Big\}= \mathcal O ((\log\log\log x)^{-1/2}), $$
 in which  $d(x)=\inf_{1\le t\le 2(\log\log x)^2}\sqrt{\E(\sum_{y\le p\le x} \frac{g_p\cos (t\log p}{p^{1/2+1/\log x}}{p^{1/2+1/\log x})^2}}=\sqrt{ ( \log\frac{\log x}{\log y})/2 +\mathcal O(1)}$ (p.\,606)\,--\,a  claim that the author was unable to verify.
\end{remarks} \medskip\par\noi {\tf Transfer  to Rademacher process.} The necessary transfer of the result to Rademacher processes is made in Harper \cite{H}   by using a multivariate central limit theorem of Reinert and R\"ollin \cite[Th.\,2.1]{RR}. Their approach is based on Stein's method. More precisely it is proved  in \cite[p.\,615]{H},
\begin{align}
\P\Big\{\sup_{t\in\T}&\widetilde  X_\e(t)\le \sqrt{2(\log\log x-\log\log y) -1}\Big\}
\cr &\le \P\Big\{\sup_{t\in\T}\widetilde  X (t)\le \sqrt{2(\log\log x-\log\log y) -1}\Big\}+\mathcal O\Big((\log\log x)^{3/2}\#(\T)^3\sum_{y\le p\le x}\frac{1}{p^{3/2}}\Big),
\end{align} 
where we noted $\widetilde  X_\e$ the Rademacher counterpart of $\widetilde  X$. Since $\#(\T)\ll(\log\log x)^2\log x$, the last term is $o(\frac{(\log\log x)^{15/2}(\log x)^3}{\sqrt y\log y})$,  and with the choice $y\ge (\log x)^8$, this is $o(1)$.
\medskip\par

{\pH An elementary multivariate inequality.} We recently established in Weber \cite{W01} an   elementary general  multivariate inequality. Let $\{A_i, 1\le i\le n\}$, $\{B_i, 1\le i\le n\}$ be arbitrary sets.  Let further $\{ u_i, 1\le i\le n\}$, $\{ v_i, 1\le i\le n\}$ be arbitrary numbers. Then 
 \begin{align}\Big|\P \Big\{ \bigcap_{i=1}^n A_i \Big\}\Big(\prod_{i=1}^nu_i\Big)- \P \Big\{ \bigcap_{i=1}^n B_i \Big\}\Big(\prod_{i=1}^nv_i\Big)\Big|\,\le\,     \sum_{i=1}^n \big(\P\{  A_i\Delta B_i\}|u_i|+ \P\{    B_i\}|u_i-v_i|\big).
\end{align} By specifying it (Theorem 1.3-b  \cite{W01}) to the case under consideration, it follows that there exists an absolute positive constant $C_1$ such that   given any  two sequences $X_i=\frac{\sum_{k=1}^mc_k(t_i)g_k}{\sqrt{\sum_{k=1}^mc^2_k(t_i)}} $ and $Y_i=\frac{\sum_{k=1}^mc_k(t_i)\e_k }{\sqrt{\sum_{k=1}^mc^2_k(t_i)}}$, $1\le i\le n$, where $\{g_k,1\le k\le m\}$, $\{\e_k,1\le k\le m\}$ are  
respectively iid Gaussian standard and iid Rademacher random variables,  one has   for any $a\ge 0$,
$$  \Big|\P \Big\{   \sup_{ 1}^n\frac{\sum_{k=1}^mc_k(t_i)g_k}{\sqrt{\sum_{k=1}^mc^2_k(t_i)}}\le a  \Big\}-   \P \Big\{  \sup_{i=1}^n\frac{\sum_{k=1}^mc_k(t_i)\e_k }{\sqrt{\sum_{k=1}^mc^2_k(t_i)}}\le a\Big\}\Big| \le C_1   n\,  e^{- a^2/2 }  .
$$
As $Nk=(\log\log x)^2( {\log x}/{\log y})^{1-\e}\sim  e^{c (\log \log x)^{1-b}}$, it follows here that we have 
 \begin{align*} \bigg|&  \P\Big\{   \sup_{n=1}^N    \sup_{i=1}^k  \widetilde  X_{y,x}(n+s_i)\le  \sqrt{2\log \frac{\log x}{ \log y}  +(\log\log x)^{ b} \log \p(k)
} \Big\}
\cr  &-\P\Big\{  \sup_{n=1}^N    \sup_{i=1}^k  \widetilde  X^\e_{y,x}(n+s_i)\le  \sqrt{2\log \frac{\log x}{ \log y}  +(\log\log x)^{ b} \log \p(k)
} \Big\}
 \bigg| 
\cr &\le \,  C \,e^{c (\log \log x)^{1-b}} e^{- \log \frac{\log x}{ \log y}  -(1/2)(\log\log x)^{ b} \log \p(k) } 
 \le \,  
  C \,e^{ - \log \frac{\log x}{ \log y}  -C_c(\log\log x)^{ b} \log \p(k) }   ,\end{align*}
for some positive constant $C_c$ depending on $c$ only.
Therefore we obtain
 \begin{align*}\P\Big\{  \sup_{n=1}^N    \sup_{i=1}^k  \widetilde  X^\e_{y,x}(n+s_i)&\le  \sqrt{2\log \frac{\log x}{ \log y}  +(\log\log x)^{ b} \log \p(k)
} \Big\}\cr &\le \frac{ C  }{(\log x)^2  } \,e^{  -   M(\log\log x)^{ b}  }+C \,e^{ - \log \frac{\log x}{ \log y}  -C_c(\log\log x)^{ b} \log \p(k) }.\end{align*}
Thereby also 
\begin{align*}\P\Big\{\sup_{1\le t\le (\log\log x)^2}       \widetilde  X^\e_{y,x}(t)&\le  \sqrt{2\log \frac{\log x}{ \log y}  +(\log\log x)^{ b} \log \p(k)
} \Big\}\cr &\le \frac{ C  }{(\log x)^2  } \,e^{  -   M(\log\log x)^{ b}  }+C \,e^{ - \log \frac{\log x}{ \log y}  -C_c(\log\log x)^{ b} \log \p(k) },\end{align*}
with   much more smaller error term. 
 \smallskip\par   The  study of   Wintner's random M\"obius function is continued in a separate work, as well as   applications of decoupling Theorem \ref{F.decoupling.th}.


  \appendix 


   \section{\bf From randomness to non-randomness, some directions,   tools and bridges}  \label{s6.}   
   
 This Appendix is  a part from the expository paper \cite{W0...} 'From randomness to non-randomness in Analysis and Number Theory'.    The interface Probability/Analysis and Number Theory always 
     guided and motivated a large part of our researches. Let $\mathcal B =\{B_n , n\ge 1\}$, $B_n =\sum_{k=1}^n \b_k $, for each $n$,    $\b =
\{\b_i  , i\ge 1\}$ 
  being iid Bernoulli random variables 
      ($\P\{ \b_i =  0\}=\P\{ \b_i =  1\} =1/2$).  
A first remarkable tool is undoubtly the one proved by Call and Velleman in \cite{CV},\,Th.\,1.  Some  related problems are listed in   the recent   paper  Weber \cite{W0}.
\subsection{Call and Velleman theorem} 
Let $f(j)$ be an arithmetical function. Consider the column vector 
  ${\boldsymbol \nu}=\big(f(1),f(2),$ $\ldots, f(n)\big)$.  Let $P_n$ be the $n$-size   lower triangular Pascal matrix:
\begin{eqnarray*}P_n={ \left[\begin{matrix}
1        \cr
1   &1       \cr
1        &2            &1    \cr
1       &3       &3 &1       \cr
 \vdots         &   \vdots                  &\vdots   &\vdots   &\ddots    
 \cr {{n}\choose{1}}& {{n}\choose{2}}&{{n}\choose{3}}&{{n}\choose{4}}  &\cdots  &{{n}\choose{n}}\end{matrix} \right]}
\end{eqnarray*}
Then $P_n{\boldsymbol \nu}=  \,{\boldsymbol w}$ where ${\boldsymbol w}=\big(\E f(B_1),2\E f(B_2),\ldots, 2^n\E f(B_n)\big)$. Call and Velleman 
 identified $P^{-1}_n$, which is $DP_nD^{-1}$ with $D$ as indicated below:
\begin{eqnarray*}\qquad P^{-1}_n={ \left[\begin{matrix}
1        \cr
-1  & 1       \cr
1        &-2            &1    \cr
-1       &3       &-3 &1       \cr
 \vdots         &   \vdots                  &\vdots   &\vdots   &\ddots    
 \cr {{n}\choose{1}}& {{n}\choose{2}}&{{n}\choose{3}}&{{n}\choose{4}}  &\cdots  &{{n}\choose{n}}\end{matrix} \right]}
 \qq
 D={ \left[\begin{matrix}
1        \cr
  &-1       \cr
       &{}      &1    \cr
  &{}          &{}         &-1       \cr
 &{}          &{}   &{}          &{}  &\ddots    
 \cr &{}          &{}   &{}          &{}    &   &(-1)^{n+1}\end{matrix} \right]}
\end{eqnarray*}
 So that knowing ${\boldsymbol w}$  we have -at least theoretically- a way back to know ${\boldsymbol \nu}$. In  condensed form     
$P_n^{-1}=Q_n$,  where  
$Q_n$ is the $n\times n$ matrix defined by   \begin{equation}Q_n(i,j)=\begin{cases} (-1)^{i-j}  {{i-1} \choose{j-1}} \quad&\hbox{if $i\ge j$,}\cr
 0 \quad&\hbox{otherwise.}
 \end{cases}
 \end{equation}  
   \begin{proposition}\label{way.back}Let $f(j)$ be an arithmetical function. Then we have the reverse formula
\beq\label{} f(i)=\sum_{j=1}^i  (-1)^{i-j}2^j{{i-1}\choose{j-1}}\E f(B_j),\qq\quad i\ge 2.
\eeq \end{proposition}

\subsection{High powers  of  large zeta sums}
      
Let  $0<\s<1$,   positive integers $N$, $k$ large and $t\in\R$.  The simple relation 
\begin{eqnarray*}\Big(\sum_{n=1}^N  {1\over
n^{ \s+ist  } }\Big)^{k}
&=& \sum_{\nu=1}^{N^k}{\d_{k,N} (\nu)\over \nu^{ \s+ist  } 
},\end{eqnarray*}
where  $ \d_{k,N} (m)= \#\big\{(n_1,\ldots, n_k)\in \{1,  N\}^k : m= n_1 \ldots  n_k \big\}$ is  the truncated arithmetical Piltz function,
shows   that   for $N$ and $k$ simultaneously large, the size's order of  these sums, a cornerstone  
in the study of RH,  is difficult to estimate, see  Titchmarsh \cite{Ti}, Ivi\'c \cite{I1,I}, Laurin\v cikas \cite{La}, Montgomery and    Vaughan \cite{MV}, Montgomery's   lectures \cite{[Mon]} and        Weber \cite{W,W3,We1}  with links to Fourier inversion formula. 
  It does not seem that  the question of the variation of these sums along power parameter $k$ have been earlier considered somewhere. It must be however remarked that
\begin{equation} \label{rmk1}  
\lim_{k\to\infty}\Big|\sum_{n=1}^N\frac{1}{n^{\s+it\varpi /\sqrt k}}\Big|^{2k} \Big\slash
\Big(\sum_{n=1}^N\frac{1}{n^\s }\Big)^{2k}\,=\, e^{-t^2/2},\end{equation}
for any real $t$, any integer $N$ and any positive $\s$, where $\varpi$ is a quantity which depends only on $N$ and $\s$, and tends to  $ {\sqrt 2}/{( 1-\s )}$ as $N$ tends to infinity, with an error term of order $\mathcal O_\s( N^{\s-1}(\log N)^2)$.  This is resulting from the formula   in   \cite{W} expressing a link to Fourier transforms 
\begin{eqnarray}\label{eit}
  \Big(\sum_{n=1}^N\frac{1}{n^{\s+it}}\Big)^k&=& \Big(\sum_{n=1}^N\frac{1}{n^\s }\Big)^k\big( \E e^{- it  Y}\big)^k.
 \end{eqnarray}
Here  $k$ and $N$ are any positive integers, 
  and    $Y=Y^{\,\!(N)}$ is an arithmetic free   random variable 
defined 
by  \begin{equation}\label{iid}\P\{ Y  =  \log n\}=\frac{1}{n^\s L_N},\ n=1,\ldots, N \qq \qq \Big(L_N= \sum_{n=1}^{ N} \frac{1}{n^\s}\Big).
 \end{equation}
This  immediately provides the following formula connecting to Probability Theory.   Let   $\widetilde Y
 $
be  a symmetrization of $Y$  
and $\widetilde S_k$    a symmetrization of   $S_k$.
 Then, \begin{equation} \label{rmk}  
 \Big|\sum_{n=1}^N\frac{1}{n^{\s+it}}\Big|^{2k}  =
\Big(\sum_{n=1}^N\frac{1}{n^\s }\Big)^{2k}\, \big( \E e^{  it \widetilde Y}\big)^k
    ,\end{equation} 
    for all positive integers $k$ and all $ t\in \R$. Also,\begin{eqnarray}\label{rm1} \frac{1}{T}\int_{-T}^{ T}\big|\sum_{n=1}^N\frac{1}{n^{\s+it}}\big|^{2k} \dd t&=& 
  \Big(\sum_{n=1}^N\frac{1}{n^\s }\Big)^{2k}\, \E  \, \frac{ \sin T\widetilde S_k }{T\widetilde S_k}.
\end{eqnarray}
 
   Although\begin{align}\label{Sk.prob} \P\{S_k=   \log m\}   
  =  \frac{ 1}{m L_N^k}\bigg(\sum_{1\le n_1,\ldots, n_k\le  N\atop n_1 \ldots  n_k=m}   1\bigg) \,=\, \frac{ \d_{k,N} (m)}{m^\s L_N^k} ,
 \end{align} 
Piltz functions obviously disappear by independence, from the expression of $\E e^{ it \widetilde S_k}$ equal to $\big( \E e^{ it \widetilde Y  } \big)^k$.  
However in \eqref{rm1},   it is necessary to work with $\widetilde S_k$. As 
$\P\big\{\widetilde S_k=0\big\}
= (L_N^{-2k})\, \sum_{m=1}^{N^k} {d_{k,N }^2(m)\over m^{2\s}}$, 
the formula below   is ensuing.

\medskip \par {\pH Mean value formula with random remainder term}
\begin{equation} \label{mvfrrt}  \frac{1}{T}\int_{-T}^{ T}\Big|\sum_{n=1}^N\frac{1}{n^{\s+it}}\Big|^{2k} \dd t =  \sum_{m=1}^{N^k} {d_{k,N }^2(m)\over m^{2\s}}
  +(L_N^{2k})\, \E \Big( \frac{ \sin T\widetilde S_k }{T\widetilde S_k}\cdot\chi \big\{\widetilde S_k\neq 0\big\} \Big)
  .\end{equation} 
 Also by Fourier inversion formula, for any real $x_0$ 
 \begin{align}\label{fif.Sk.dis.a} 
  {1\over 2T}&\int_{-T}^Te^{-itx_0} \Big|\sum_{n=1}^N\frac{1}{n^{\s+it}}\Big|^{2k}\, \dd t   =  \P\big\{\widetilde S_k=x_0\big\}+
  (L_N^{ 2k})\, \E  \Big\{ {\sin T(\widetilde S_k-x_0) \over  T(\widetilde S_k-x_0)}\cdot\chi \big\{\widetilde S_k\neq x_0\big\} \Big\}  .
\end{align}
  The integral in the left-hand side tends to 0 as $x_0$ approaches infinity, by    Lebesgue-Riemann  lemma. In the critical  case   $x_0\sim 0$,  
  the rough bound  $$  \Big| \frac{ \sin T\widetilde S_k }{T\widetilde S_k}\Big|\cdot\chi \big\{\widetilde S_k\neq 0\big\}  \,\le\,  \frac{ 1 }{T|\widetilde S_k|}\ \chi \big\{\widetilde S_k\neq 0\big\}, 
$$
together with the fact that the minimal positive value taken by $|\widetilde S_k|$ is $\ge   {1\over 2 N^k}$, imply
$$\Big| \frac{ \sin T\widetilde S_k }{T\widetilde S_k}\Big|\cdot\chi \big\{\widetilde S_k\neq 0\big\}  \,\le\,  \frac{ 2 N^k }{T  } . 
$$   Hence,
\beq \Big| \frac{1}{T}\int_{-T}^{ T}\Big|\sum_{n=1}^N\frac{1}{n^{\s+it}}\Big|^{2k} \dd t -  \sum_{m=1}^{N^k} {d_{k,N }^2(m)\over m^{2\s}}\Big|
  \le \,  \frac{ 2 L_N^{2k} N^k }{T  } 
  .
 \eeq

If  $\s =  1 $, $L_N\sim C(\log N)$, this bound has order $C\frac{ 2  (\log N)^{2k} N^{  k}   }{T  }$ and   
 we get 
\begin{equation*} \label{mvf.est1} \Big| \frac{1}{T}\int_{-T}^{ T}\Big|\sum_{n=1}^N\frac{1}{n^{1+it}}\Big|^{2k} \dd t - \sum_{m=1}^{N^k}{d^2_{k,N }(m) \over m^2  }\Big|
  \le \,  \frac{ (\log N)^{2k} N^k }{T  } 
  ,
\end{equation*} 
  providing a   better error term than in  Montgomery  \cite{[Mon]}, where we recall
\begin{equation*}\label{asymp1.recalled}   \bigg|{1\over T}\int_0^T \Big|  \sum_{n=1}^N  {1\over
n^{  1 +it  } } \Big|^{2k}\dd t -   \sum_{m=1}^{N^k}{d^2_{k,N }(m) \over m^2  } \bigg|\le    C_k  {(\log  N)^{k^2}N^k\over T}     .
\end{equation*}

\begin{eqnarray} \label{rm2}\int_\R \big|\sum_{n=1}^N\frac{1}{n^{\s+ist}}\big|^{2k} \frac{\dd t}{\pi( t^2+1)}     &=& \Big(\sum_{n=1}^N\frac{1}{n^\s }\Big)^{2k}\, 
 \E
\,  e^{ -s|\widetilde S_k|} .
\end{eqnarray} 

  By \cite[Lemma 4.3]{W}, 
 \begin{eqnarray}\label{var} \E\, \widetilde Y^2
 \ =\ \frac{2}{(1-\s)^2} +  \mathcal O_\s( N^{\s-1}(\log N)^2).\end{eqnarray}
The  hidden parameter $N$ on which $Y$ depends,  only appears in the error term, the main term being moreover constant.  It   follows from the proof  that the almost constant behavior of the variance is due to   cancellation of auxiliary  sums in the estimation of
$$ 2\Big\{\sum_{m=1}^N \frac{( \log m)^2}{m^\s L_N}- \Big( \sum_{m=1}^N \frac{( \log m)}{m^\s L_N}  \Big)^2\Big\}. $$ 
Let  next set $s_k^2={\rm Var}(\widetilde S_k ) =k \E\, \widetilde Y^2$, $k\ge 1$. Then 
\begin{eqnarray}  \label{rm3} { \widetilde S_k/ s_k}  \, \buildrel{\mathcal D}\over{ \,\longrightarrow\,}  \, \mathcal N(0,1), \qq \quad k \to \infty.
 \end{eqnarray} 
  Formula \eqref{rmk} and \eqref{rm2}  suggest  to     closely investigate the uniform behavior for all $N$ and $k$ large,  all $t$ lying in    $[1, T]$, $T$ large, of the characteristic functions $
   \E e^{ it \widetilde S_k/s_k}$. 
   We establish  in \cite[Th.\,1.5]{W0..}    uniform estimates of high powers  of  large zeta sums,  more precisely we prove   a semi-uniform  limit theorem for characteristic functions $\exp\{it{\widetilde S_k/ s_k}\}$, with speed of convergence. 
   \begin{theorem}\label{t3}
   Let $0<\s<1 $. Let  also $0<\a<1/2$     and $\p(n)\ge 1$ non-decreasing,   $\p(n)=o(\log\log n)$. There exist constants $N_\s$, $C_\s$ and an absolute constant $C$, such that if  \begin{equation}\label{k.2}
          k  \ge C_\s\, e^{ \frac{C n  \log\log n}{ \a\p(n)}}, \qq\quad k\ge     \big(\frac{n}{e^2(1-\s)}\big)^{\frac{4}{ 1-2\a}} 
 \end{equation}
then
 \begin{eqnarray}\label{expit4}
\sup_{N\ge N_\s}\,\sup_{1\le  |t| \le e^{\,\frac{\log\log n}{\p(n)}}}\big|\E e^{it{\widetilde S_k\over s_k}} -  e^{-{t^2\over 2}}\big|& \le&  C_\s\, e^{ \frac{-C n  \log\log n}{ \a\p(n)}}. \end{eqnarray}
 \end{theorem}  
 Estimate \eqref{expit4} is  certainly improvable. 
 An important part of the proof is devoted to the  estimation of   moments $\E \widetilde Y^m$,  $m$ even. We proceed  with an induction argument, involving a more delicate algebra with     cancellations effects, and   a combinatorial argument.
  
 
 \vskip 6 pt



\subsection{Skorokhod embedding scheme} 
This is a very flexible and  quite important tool.   Applications  of this part to Fourier inversion and mean value integrals are given in Weber \cite{W0.}. 
Let
$W=\{W(t), t\ge 0\}$ denotes a standard Brownian motion.  As is well-known, any centered   measure
$\m$ on the real line embeds into $W$: 
 \beq \label{sko.embed}\hbox{{\it There exists a stopping time $\tau$ such that}}\quad W(\tau)\buildrel{{\mathcal D}}\over {=} \mu. 
\eeq 
  To be precise,  $\tau$ is the first exit time of $W$ from a random interval containing $0$, and further $\{W(\tau\wedge t):\, t\ge 0\}$
is a uniformly bounded martingale.    We  refer to Obloj \cite{[O]} p.\,332 where an explicit construction
of
$T$, which is   the Skorokhod stopping time, is given; but also to Sawyer \cite{[Sa2]}, section 2 for instance. By the Markov property $\{W(\tau +t)- W(\tau), t\ge 0\}$ is a Brownian motion.
    
 By the Burkholder, Davis, Gundy and Millar inequalities (see Proposition 2.1 in Obloj \cite{[O]} and Durrett \cite{D} p.\,697, or estimates (1.10) in
Sawyer \cite{[Sa2]}, for any
$1\le p<\infty$, there exist universal constants $c_p, C_p$ such that
\begin{equation}\label{tau.int} c_p\E \tau^{p/2}\le \E |W(\tau)|^p= \int_\R |x|^p\m(dx)\le   C_p\E \tau^{p/2}.
\end{equation}
  Specifying for our  purpose   Theorem 1 of   \cite{[Sa2]},  we have:  for any $\a\ge 0$, there is a constant $C_\a$ depending on $\a$ only
such that
\begin{equation}\label{tau.int.}\E e^{(\a  \tau)^{1/2}}\le C_\a \int_\R e^{\a  |x|^{ 2}}\m(dx).
\end{equation} When
$\m$ is a symmetric measure, the latter estimate is even two-sided   (\cite{[Sa2]}, Theorems 2 and 3).  
\vskip 5 pt 
  The following result due to  Fisher is relevant in this work.
\begin{lemma}[\cite{F},\,Th.\,2.1]\label{fisher.lem}Let   $\xi$ be  a real random variable
 satisfying   $\E \xi=0$ and $\E \xi^2< \infty$,
and   $\lambda$ is a fixed positive real. Then  on a possibly larger probability space,
 there exist a Brownian motion $W$ and stopping   times $T$  
and $T_\lambda$ ({\it both} are exit times), 
  such that
\begin{equation}\label{fisher.st}
W(T)\buildrel{{\mathcal D}}\over=\xi ,\quad
W(T_\lambda)\buildrel{{\mathcal D}}\over=\lambda \xi ,
\quad T_\lambda \buildrel{{\mathcal D}}\over=\lambda^2 T.
\end{equation}
\end{lemma}
See Lin-Weber \cite{[LinW]},   Proposition 3.6, for  a direct approach using Laplace
 transform,  avoiding Billingsley's approximating argument in \cite{Bi}. 
 
 \bigskip 
  Let  us give an application. Consider the following important example ({\it cf}\, previous subsection). Assume that  $\mathcal Y= \{Y_k, k\ge 1\} $ is a sequence of  independent, centered, square integrable  random variables, and let 
  $ {\mathcal Z}_k= Y_1+ \ldots+  Y_k$, for each $k$.  Let ${\rm Var}({\mathcal Z}_k)=\sum_{j=1}^k \E Y_j^2$  and   $\p_k(t)$ be the characteristic function of ${\mathcal Z}_k$.   
  By using the Skorokhod embedding scheme we  can  transfer    the 
   study of  the familiar key quantity \beq
  \label{Zk.fif.} \E\, \Big(    {  \sin   T
  ({\mathcal Z }_k -M) \over    T({\mathcal Z }_k -M)} \Big) ,
 \eeq
$M$ real,    into a \lq\lq Brownian environment\rq\rq\,as follows:  
{\it There exist  a probability 
space with a   Brownian motion $\{W(t), t\ge 0\}$
and non-negative iid random variables $\{\tau_\ell, \ell\ge 1\}$ with $\E \tau_\ell=\E Y^2$,
such that if $T_j=\sum_{\ell =1}^j \tau_\ell$,  \begin{equation}\label{sko}
( S_1,S_2,\cdots,S_k,\cdots) \buildrel{\mathcal D}\over{ =}   
\big(W( T_1),W( T_2),\cdots,
W(T_j),\cdots \big).
\end{equation}
Moreover,  for each real number $r\ge 1$,    $  \E  \tau_1^{r/2} \le C^r_r\, \E |Y_j|^{r}  
 $.}
\vskip 2 pt We showed  in Weber \cite{W0.} that this device allows one to compute completely $\E\,     {  \sin   T
   {\mathcal Z }_k   \over    T {\mathcal Z }_k }$ in important cases. 
       The following general result is also proved in the same paper.
\begin{theorem}\label{inter.est}Let $k\ge 1$, $T>0$, $M\ge 0$, $0<\e(k)\le  (1/2)\E T_k$,    $0<\a <1$ and  $x>0$.  Then for some universal constant  $C$,
\begin{align}\label{inter.est.3} 
   \Big| \E\,    &{  \sin   T
  ({\mathcal Z }_k -M) \over    T({\mathcal Z }_k -M)} -  \P\{{\mathcal Z}_k  =M\}\Big| \le \,      
    \P\Big\{\Big|{ T_k\over \E T_k}-1\Big|>\e(k)\Big\}
  +C  \Big( C   \sqrt{   \e } \,  e^{-\frac{M^2}{8 \E T_k }}\Big) 
  \cr &+\frac{ C x  }{ (2^{   1-\a }-1)
T^\a\sqrt{  \E T_k  }} + \frac{1}{x^\a}\int_{|v|>x }  \frac{ e^{-\frac{(M+v)^2}{2\E T_k(1-\e(k))}}}{
\sqrt{2\pi \E T_k(1-\e(k))}}dv.
\end{align}   
 
If $M=0$, we have the simplified bound,
 \begin{align}\label{inter.est.5} 
   \Big| \E\,     {  \sin   T
  ({\mathcal Z }_k ) \over    T({\mathcal Z }_k )} -  \P\{{\mathcal Z}_k  =0\}\Big|  \le \,      
    \P\Big\{\Big|{ T_k\over \E T_k}-1\Big|>\e(k)\Big\}  +C    \sqrt{   \e(k)}    
   + \frac{ C_\a        }{  T^{ \frac{\a^2}{1+\a}} (\E T_k)^{ \frac{\a}{2(1+\a)}}  }. 
\end{align}  and $C_\a= \frac{C}{(2^{   1-\a }-1)}+1$.  \end{theorem}
 \begin{theorem}\label{estim2}    Let $Y$ be a centered, square integrable  random variable, {\it not} of lattice type. Assume that $\{  Y_j, j\ge 1\}$ are independent copies of $Y$ and let 
  $ {\mathcal Z}_k= Y_1+ \ldots+  Y_k$,  for each $k$. \vskip 2 pt {\rm (1)} Let  $T\ge 1$, $\eta_k, \e_k$ be positive  reals  with $0< \e_k  <1$, $M\ge 0$ and $k\ge 1$ integer. Then,
   $$  \E\,   \Big| {  \sin   T
  ({\mathcal Z }_k -M) \over    T({\mathcal Z }_k -M)}\Big|\,  \le\, \frac{1}{T\eta_k}  \int_{|v|>\eta_k}  \frac{ e^{-\frac{(M+v)^2}{2\E Y^2(1-\e_k)k}}}{
\sqrt{2\pi\E Y^2(1-\e_k)k}}\, dv $$
  $$ \qq\qq\qq + \P\Big\{\Big|{ T_k\over \E T_k}-1\Big|>\e_k  \Big\}
+\frac{ 5 \eta_k\,e^{-\frac{M^2}{2 \E Y^2(1-\e_k)k}} }{ 
\sqrt{ \E Y^2(1-\e_k)k }}+C   \sqrt{   \e_k }\,   e^{-\frac{M^2}{8 \E T_k }}.
$$
\vskip 2 pt {\rm (2)} Assume that $Y$ has finite moment of order $4$. Let $M=0$, $0<\p_k\ll  k$ for each $k$ and choose $\e_k =   \big( {\p_k}/{  k   }\big)^{1/2}$. Then there exists a constant $C_Y$ depending on $Y$ only, such that
$$  \E\,   \Big| {  \sin   T
   {\mathcal Z }_k   \over    T {\mathcal Z }_k }\Big|\,  \le\, C_Y
   \Big\{\frac{1}{T\eta_k}    +  \frac{  1} {\p_k}
+\frac{   \eta_k   }{ 
\sqrt{ k }}+C     \Big(\frac{\p_k}{  k   }\Big)^{1/4} \Big\}.
$$
 \vskip 2 pt {\rm (3)} Let $\mathcal N$ be an increasing sequence of positive integers such that 
 \beq\label{series.cv}\sum_{k\in \mathcal N}\frac{  1} {\p_k}<\infty,   \quad \quad \sum_{k\in \mathcal N}\big(\frac{\p_k}{  k   }\big)^{1/4}<\infty, \quad\sum_{k\in \mathcal N} \frac{   \eta_k   }{ 
\sqrt{ k }} <\infty,\quad \sum_{k\in \mathcal N}\frac{1}{ \eta_k}<\infty   .
\eeq
Then $$\sum_{k\in \mathcal N}  \E\,   \Big| {  \sin   T 
   {\mathcal Z }_k   \over    T  {\mathcal Z }_k }\Big|\,  \le\,  \frac{C_Y}{T }
\Big\{\sum_{k\in \mathcal N}    \frac{1}{ \eta_k}   \Big\}  + C_Y
\sum_{k\in \mathcal N}   \Big\{ \frac{  1} {\p_k}
+\frac{   \eta_k   }{ 
\sqrt{ k }}+C     \Big(\frac{\p_k}{  k   }\Big)^{1/4} \Big\}.
$$
\end{theorem}

 \vskip 2 pt 
Let us give an idea  on  how to proceed with applying Skorokhod embedding scheme.  It is necessary to let laws of infima of Brownian motion come into play. This happens the following way; put $B(k, M,\e_k)=     \inf_{|t- \E T_k|\le
\e_k  \E T_k  }
|W(t)-M|$,
 and introduce the   set 
 $$\D_k=\Big\{  \Big|{ T_k\over \E T_k}-1\Big|\le \e_k ,\  B(k, M,\e_k)\ge \eta_k\Big\}.$$
 Since $\{|W(T_k)- M|>0\}\cap \D_k=\D_k$, we have   equality between the integrals below,
  \begin{align*}
\E     \Big|  {   \sin  T( W(T_k)-M) \over    T( W(T_k)-M)}\Big|  \cdot \chi\{ 
 \D_k\} =\E     \Big|  {   \sin  T( W(T_k)-M) \over    T( W(T_k)-M)}\Big|  \cdot \chi\{ W(T_k)\not= M\}  \cdot \chi\{ 
 \D_k\}.
    \end{align*}
   
   Further on $\D_k$,
 \begin{align}\label{l1} 
     \Big| {   \sin  T( W(T_k)-M) \over    T( W(T_k)-M)}\Big| \,   \le  \   \frac{1}{T\eta_k}
.
   \end{align}
 The following   lemmas  are
  taken from Weber \cite{W8.}.
     \begin{lemma}[\cite{W8.},\,Proposition 2.6]\label{infima.alpha.} Let $b>a>0$, and $\eta>0$. Then for $M\ge 0$,  
 \begin{eqnarray*} 
  \P\big\{0<\inf_{a\le t\le b}|W(t)-M|<\eta \big\} 
    &\le &5 \eta\,e^{-\frac{M^2}{2a}} \min\Big(\frac{ 1 }{ 
\sqrt{    a }},   \frac{ 1 }{ 
\sqrt{  b-a }}  \Big) .
\end{eqnarray*}
 \end{lemma}

      \begin{lemma}[\cite{W8.},\,Lemma  2.3]  \label{l3} There exists an  absolute constant $C$, such that for every real $M$, and $0<a<b<\infty$, 
\begin{eqnarray*}\P\Big\{\inf_{a\le t\le b}|W(t)-M|= 0 \Big\} &=&2\int_{\R}    \Psi(\frac{|v| }{ \sqrt{b-a}}) \frac{ e^{-\frac{(M+v)^2}{
2a}}}{
\sqrt{2\pi a}}dv 
\cr &\le & C \min\bigg(1,\sqrt{\frac{b-a}{a}}\,  e^{-\frac{M^2}{8\max(a,b-a)}}\bigg)   .
\end{eqnarray*}
 \end{lemma}

 Thus with $a=\E T_k(1-\e_k)\le \E T_k$, $b=\E T_k(1+\e_k)$, $0<\e\le 1/2$,
\beq\label{inter.est.1}   \P\big\{
  B(k, M ) =0\big\}
\,\le \, C \min\Big(1,\sqrt{\frac{2\e_k }{  1-\e_k }}\   e^{-\frac{M^2}{8 \E T_k(1-\e_k)}}\Big)
\,\le \,C   \sqrt{   \e_k }\,   e^{-\frac{M^2}{8 \E T_k }} .
\eeq 
 Applying Lemma \ref{infima.alpha.}   with $a=\E T_k(1-\e_k )\le \E T_k$, $b=\E T_k(1+\e_k )$, we get 
    \begin{eqnarray}\label{formula.M.enonce} 
  \P\big\{0<B(k, M )<\eta_k \big\} 
    &\le &\frac{ 5 \eta_k\,e^{-\frac{M^2}{2 \E T_k(1-\e_k )}} }{ 
\sqrt{ \E T_k(1-\e_k ) }}
.
\end{eqnarray}
Thus \ben  \P\big\{\D_k^c \big\}&\le&  \P\Big\{\Big|{ T_k\over \E T_k}-1\Big|>\e_k  \Big\}
+\P\big\{0<B(k, M )<\eta_k \big\} +  \P\big\{
  B(k, M ) =0\big\}
\cr &\le&  \P\Big\{\Big|{ T_k\over \E T_k}-1\Big|>\e_k  \Big\}
+\frac{ 5 \eta_k\,e^{-\frac{M^2}{2 \E T_k(1-\e_k )}} }{ 
\sqrt{ \E T_k(1-\e_k ) }}+C   \sqrt{   \e_k }\,   e^{-\frac{M^2}{8 \E T_k }}.
\een

 The following Proposition is a slight refinement of Proposition 2.6 in  \cite{W8.}, in which $x=1$, and   \eqref{int.infima.alpha.enonce} is proved for $M=0$, the proof being the same. It is indeed necessary to have at disposal a version with $x>0$, as a suitable  value of $x$ will be determined later on. 
 
\begin{proposition}\label{infima.alpha} Let $b>a>0$, and $\eta>0$. Then for $M\ge 0$,  
 \begin{eqnarray*}
  \P\big\{0<\inf_{a\le t\le b}|W(t)-M|<\eta \big\} 
    &\le &5 \eta\,e^{-\frac{M^2}{2a}} \min\Big(\frac{ 1 }{ 
\sqrt{    a }},   \frac{ 1 }{ 
\sqrt{  b-a }}  \Big) .
\end{eqnarray*}
Further, for any real $\a$, $0\le\a<1$ and any reals $M\ge 0$ and $x>0$,    
 \begin{align}\label{int.infima.alpha.enonce}\E \bigg(&\frac{1}{ \displaystyle{\inf_{a\le t\le b}|W(t)-M|^{ \a}}} \cdot\chi\{\inf_{a\le t\le b}|W(t)-M|> 0\}\bigg)   
\cr &\,\le \,\frac{ 10 \, x  \,e^{-\frac{M^2}{2a}}}{ 2^{   1-\a }-1  }   \min\Big(\frac{ 1 }{ 
\sqrt{    a }},   \frac{ 1 }{ 
\sqrt{  b-a }}  \Big) +    \frac{1}{x^\a}\int_{|v|>x }  \frac{ e^{-\frac{(M+v)^2}{2a}}}{
\sqrt{2\pi a}}dv.
\end{align}
Furthermore,
\begin{align}\label{int.infima.alpha..enonce}\E \bigg(&\frac{1}{ \displaystyle{\inf_{a\le t\le b}|W(t) |^{ \a}}} \cdot\chi\{\inf_{a\le t\le b}|W(t) |> 0\}\bigg)  
    \,\le \,\frac{ 10 \, x   }{ 2^{   1-\a }-1  }   \min\Big(\frac{ 1 }{ 
\sqrt{    a }},   \frac{ 1 }{ 
\sqrt{  b-a }}  \Big)  
+   \frac{1}{x^\a} . 
\end{align}
 \end{proposition}

  \subsection{Invariance principle} 
   
    The one presented is due to Sakhanenko (see \cite{S}, Theorem 1).

 \vskip 3 pt  Let $W(t), t\ge 0$ be standard Brownian motion. Let $\{\xi_j, j\ge
1\}$ be independent centered random variables    with absolute second moments. Let $t_k= \sum_{j=1}^k \E \xi_j^2$,  
 $S_k=  \sum_{j=1}^k  \xi_j  $ and let $\{r_k, k\ge 1\}$ be some non-decreasing sequence of positive reals. Let $\a\ge 2$, $y>0$. Put 
successively,
\begin{equation} \begin{cases} \D_n&=\displaystyle{ \sup_{k\le n}} \ |S_k-W(t_k) |,\cr 
\D & =   \displaystyle{\sup_{n\ge 1}} \ \frac{\D_n}{  r_n},\cr 
\overline{\xi} &= \displaystyle{ \sup_{j\ge 1} } \ \frac{|\xi_j|}{  r_j},\cr
L_\a(y)&=  \displaystyle{\sum_{j\ge 1}}  \ \E \min\Big\{\frac{|\xi_j|^\a}{  y^\a r_j^\a}, \frac{|\xi_j|^2}{  y^2 r_j^2}\Big\}
\end{cases}.
\end{equation}
 \begin{theorem} \label{sha1}There exists an absolute constant $C$ such that for any fixed $\a$, there exists a Brownian motion $W$ such that
for all
  $x>0$,
$$ \P\big\{\D  \ge C \a x\big\}\le L_\a(x) .  $$
\end{theorem}
We give an application of Theorem \ref{sha1} to    the truncated prime divisor function
  $\o(m,t)=\#\{ p\le t : p|m\}$. 
   Put 
\begin{equation}\label{rho}\rho (m,t) :={ | \o(m,t)- \log\log t |\over \sqrt{\log\log  t}}. 
\end{equation}
 \smallskip\par The Kubilius model is defined as follows. Let $\{Y_p, p\ge 1\}$ be a
sequence of independent binomial random variables such that
$\P\{Y_p=1\}=1/p$ and
$\P\{Y_p=0\}=1-1/p$. We can view $Y_p$ as modelling whether or not an integer taken at random is divisible by $p$. Let 
$$ T_n=\sum_{p\le n}  Y_p, \qq\qq S_n=T_n -\E T_n.  $$
Then   $\E S_n^2= s_n^2  =\sum_{p\le n}\frac{1}{p}-\frac{1}{p^2}=\log\log n +\mathcal
 O(1)$.
The sequence $\{T_n, n\ge 1\}$ is known to asymptotically behave as the truncated prime divisor function $\o(m,t)$.  More precisely, let
 $\o_r(m)=\big( \o(m,1), \ldots, \o(m,r)\big)$, 
where $r$ is some integer with $2\le r\le x$, and put $u={\log x\over \log r}$. Then, given $c<1$ arbitrary, we have {\it uniformly} in
$x,r
$ and
$Q\subset \Z^r$,
\begin{equation} \label{km} {\#\{ m\le x: \o_r(m)\in Q\}\over x}=\P\big\{(T_1, \ldots, T_r)\in Q\big\}+ \mathcal O\big( x^{-c}+ 
e^{-u\log u}\big).
\end{equation}
See Lemmas 3.2, 3.5 in \cite{E} Chapter 3. 

  This  model is mostly adapted to the analysis of the distribution of the small divisors
of an integer. See  details and example in Weber \cite{W5a}. Formula  (\ref{km}) can be used   to   transfer   properties from $(T_k)$ to
$\o$. Let    $f$ be such that $f(N)= o_\rho(N^\rho)$,  $I_N =\big\{j: N\le
s_j^2\le
 Nf(N)\big\} $ and 
  $\mathcal N$ be some fixed increasing sequence of reals. 
 Moreover, let $M_i :\N\to\R^+$   be non-decreasing  with $ \lim_{x\to\infty} M_i(x)=\infty$, $i=1,2$, and  such that 
 \begin{equation}\label{mi}1\le M_1(x)<M_2(x), \quad \qq  M_2(x)= \mathcal O_\e (x^\e). 
\end{equation}Let $r=r(x)\sim M_2(x)f(M_2(x))$, $r $ integer. Then 
$$ u=u(x)={\log x\over \log r(x)}\sim {\log x\over \log M_2(x)f(M_2(x))}\to \infty
 $$
with $x$.  Put 
 $$Q_{x}=\bigcup_{M_1(x)<N\le M_2(x)\atop 
N\in \mathcal N}\Big\{(\nu_1,\ldots,\nu_r)\in \Z^r:     \sup_{j\in I_N } { |\nu_j- s_j^2 |\over s_j}
 \le z\Big\}.$$
By  applying     (\ref{km}) with $Q=Q_{x}$, we get the useful comparison relation
\begin{lemma}\label{otos}   For any $z>0$,  as $x$ tends to infinity,
$${1\over x}  \#\Big\{ m\le x:\inf_{ M_1(x)<N\le M_2(x)\atop 
N\in \mathcal N} \sup_{j\in I_N } { | \o(m,j)- \log\log j |\over \sqrt{\log\log  j}}\le z\Big\}$$$$ =\P\Big\{ \inf_{
M_1(x)<N\le M_2(x)\atop 
N\in \mathcal N}  \sup_{j\in I_N } { |T_j - \log\log j |\over \sqrt{\log\log  j}}\le z\Big\}+ o ( 1).
$$
  \end{lemma}   
 
The local variations of  $\rho (m,t) $
were   investigated by Ford and Tenenbaum, who obtained in   \cite{FT}, after a careful study of the size of intervals of slow growth for general sums of independent random
variables,   quite elaborated asymptotic estimates, on the basis of  the 
approximation formula (\ref{km}).  
\medskip\par
In \cite{W5a},   we  proceeded slightly differently. Let $U=\{  U(t) , t\in \R \}$, $U(t)= W(e^t)e^{-t/2}$, be the
Ornstein-Uhlenbeck process. Having first characterized  in terms of Upper/Lower classes the size of  instants of small amplitude of  $U$, we
directly compared the previous functionals  with analogous $U$-functionals  by means of Theorem \ref{sha1}.   This allows to transfer these results to truncated prime divisor function, not fully naturally, but sufficiently much to get new quite sharp results. 
More precisely, let
$0<M_1(x)<M_2(x)$,
$M_2(x)\uparrow\infty
$ with
$x$. 
\begin{theorem} \label{outom}  Assume that $M_2(x)= \mathcal O_\e (x^\e)$. Let $0<z''<z<z'$. As $x$ tends to infinity,
$$ \P\Big\{ \inf_{
M_1(x)<N\le M_2(x)\atop 
N\in \mathcal N}  \sup_{j\in I_N } { |W(s_j^2)  |\over s_j}\le z''\Big\} + o ( 1)\le \qq\qq\qq \qq \qq\qq\qq  $$
$$ {1\over x}  \#\Big\{ m\le x:\inf_{ M_1(x)<N\le M_2(x)\atop 
N\in \mathcal N} \sup_{j\in I_N } { | \o(m,j)- s_j^2 |\over s_j }\le z\Big\}  $$
$$\qq\qq\qq \qq\qq\qq \le \P\Big\{ \inf_{
M_1(x)<N\le M_2(x)\atop 
N\in \mathcal N}  \sup_{j\in I_N } { |W(s_j^2)  |\over s_j}\le z' \Big\} + o ( 1).
$$
 \end{theorem} 
We deduce for
instance 
 \begin{corollary} \label{oma} Assume that $M_2(x)= \mathcal O_\e (x^\e)$ and $   \log M_1(x)  = o\big( \log M_2(x)  \big)$. Let
$c<1/\l(z)$.   Then
$$\lim_{x\to \infty} {1\over x} 
\#\Big\{1\le  m\le x:\inf_{ \log M_1(x)<k\le \log M_2(x)}
\sup_{e^k\le s_j^2\le  e^kk^c} { |
\o(m,j)- s_j^2 |\over s_j }\le z\Big\}=1.  $$
\end{corollary}   


 \end{document}